\documentclass[11pt, reqno]{amsart}
\usepackage[bookmarksnumbered, plainpages]{hyperref}
\usepackage{amsmath,amssymb,graphicx,mathabx,accents}
\usepackage{enumerate,mdwlist}
\usepackage{enumitem}

\usepackage{esint}
\usepackage[symbol]{footmisc}
\usepackage{mathtools}
\usepackage{mathrsfs}

\usepackage{tikz}

\numberwithin{equation}{section}

\usepackage{amsthm}

\usepackage{verbatim}

\usepackage{nag}

\usepackage{tikz-cd} 

\makeatletter
\DeclareRobustCommand{\hvec}[1]{{\mathpalette\hvec@{#1}}}
\newcommand{\hvec@}[2]{%
  \vbox{\offinterlineskip
    \ialign{%
      \hfil##\hfil\cr
      $\m@th#1{}_{\rightharpoonup}$\kern-\scriptspace\cr
      $\m@th#1#2$\cr
    }%
  }%
}
\makeatother

\DeclareMathOperator{\minkdim}{\dim_{\mathbb{M}}}
\DeclareMathOperator{\hausdim}{\dim_{\mathbb{H}}}
\DeclareMathOperator{\lowminkdim}{\underline{\dim}_{\mathbb{M}}}

\DeclareMathOperator{\fordim}{\dim_{\mathbb{F}}}

\DeclareMathOperator{\RR}{\mathbb{R}}
\DeclareMathOperator{\ZZ}{\mathbb{Z}}
\DeclareMathOperator{\QQ}{\mathbb{Q}}
\DeclareMathOperator{\TT}{\mathbb{T}}
\DeclareMathOperator{\CC}{\mathbb{C}}

\DeclareMathOperator{\EE}{\mathbb{E}}
\DeclareMathOperator{\PP}{\mathbb{P}}

\newcommand{\psitwo}[1]{\| {#1} \|_{\psi_2(L)}}

\newtheorem{theorem}{Theorem}
\newtheorem{lemma}[theorem]{Lemma}

\newtheorem{remark}[theorem]{Remark}

\newtheorem*{remarksaboutresults}{Remarks About The Results Stated}
\numberwithin{theorem}{section}

\begin{document}

\centerline{}

\centerline{}

\title[Large Salem Sets Avoiding Nonlinear Configurations]{Large Salem Sets Avoiding Nonlinear Configurations} % and More General Spectral Multipliers on Manifolds}
\author[Jacob Denson]{Jacob Denson$^*$}
\address{$^{*}$ University of Madison Wisconsin, Madison, WI, jcdenson@wisc.edu}
%\dedicatory{This paper is dedicated to Professor ABCD}
%\date{Received: xxxxxx; Revised: yyyyyy; Accepted: zzzzzz.}
%\newline \indent $^{*}$ Corresponding author}

\begin{abstract}
    We obtain a family of results on the construction of large Salem sets avoiding patterns, complementing previous constructions of sets with large Hausdorff dimension avoiding patterns. Given a smooth function $f: (\TT^d)^{n-1} \to \TT^d$ satisfying a modest geometric condition, we prove that a generic Salem subset of $\TT^d$ with dimension $d/(n-3/4)$ avoids nontrivial solutions to the equation $x_n = f(x_1,\dots,x_{n-1})$, which, in particular, implies the existence of large Salem subsets of curves avoiding isosceles triangles. We also show a generic Salem subsets of $\TT^d$ with dimension $d/(n-1)$ avoids nontrivial solutions to any $n$ term linear equation with rational coefficients, extending a result of K\"{o}rner to higher dimensions. We also provide results in the `rough pattern setting', proving that for any set $Z \subset (\TT^d)^n$ formed from a countable union of sets with lower Minkowski dimension at most $s$, and for a generic Salem subset $E$ of $\TT^d$ with dimension $(dn-s) / (n-1/2)$, the Cartesian product $E^n$ only intersects $Z$ at points with non-distinct coordinates.
\end{abstract} \maketitle

\section{Introduction}

Geometric measure theory explores the relationship between the geometry of subsets of $\TT^d$, and analytic properties of Borel measures supported on those subsets. From the perspective of harmonic analysis, it is interesting to explore what geometric information can be gathered from the Fourier analytic properties of these measures, in particular, their Fourier decay. Some results show that the support of measures with Fourier decay must contain certain arrangements of points, such as arithmetic progressions \cite{ChanLabaPramanik, HenriotLabaPramanik,LabaPramanik}. In this paper, we work in the opposite direction, showing that \emph{most} sets supporting measures with a certain type of Fourier decay \emph{do not} contain certain configurations. More precisely, for certain sets $Z \subset (\TT^d)^n$, and $s > 0$, we show that a `generic' compact set $E \subset \TT^d$ supporting a measure $\mu$ such that $|\widehat{\mu}(\xi)| \lesssim |\xi|^{\varepsilon -s/2}$ for any $\varepsilon > 0$ also \emph{avoids} the pattern defined by $Z$, in the sense that for any \emph{distinct} points $x_1,\dots,x_n \in E$, $(x_1,\dots,x_n) \not \in Z$. As an example, one might have $Z = \{ (x_1,\dots,x_n) \in \TT^{dn}: x_n = f(x_1,\dots,x_{n-1}) \}$, in which case a set $E$ avoids $Z$ precisely when, for any distinct points $x_1,\dots,x_n \in E$, $x_n \neq f(x_1,\dots,x_{n-1})$, i.e. $E$ does not contain nontrivial solutions to the equation $x_n = f(x_1,\dots,x_{n-1})$.

% This paper focuses on the \emph{pattern avoidance problem} of geometric measure theory, which asks whether a set with large dimension necessarily contains certain types of geometric configurations, such as points forming the vertices of isosceles triangles, points forming arithmetic progressions, or the existence of points generating certain angles or distances from one another.

A useful statistic associated with any Borel set $E \subset \TT^d$ is its \emph{Fourier dimension}; given a finite Borel measure $\mu$, its Fourier dimension $\fordim(\mu)$ is the supremum of all $s \in [0,d]$ such that $\sup_{\xi \in \ZZ^d} |\widehat{\mu}(\xi)| |\xi|^{s/2} < \infty$. The Fourier dimension of a compact Borel set $E$ is then the supremum of $\fordim(\mu)$, where $\mu$ ranges over all Borel probability measures $\mu$ with $\text{supp}(\mu) \subset E$. A particularly tractable family of sets in this scheme are \emph{Salem sets}, sets whose Fourier dimension agrees with their Hausdorff dimension. Often constructions in the pattern avoidance literature focus on constructing sets with large Hausdorff dimension which avoid patterns \cite{OurPaper,PramanikFraser,Mathe}, but it is not necessarily possible to use the same techniques to construct large Salem sets avoiding patterns, or even to construct sets with positive Fourier dimension, since having large Fourier dimension is often a much stronger property than having large Hausdorff dimension. Nonetheless, in this paper we are able to provide techniques proving the existence of large Salem pattern avoiding sets.

Our paper is part of a growing body of literature on \emph{pattern avoidance problems}: given a set $Z \subset \TT^{dn}$, the pattern avoidance problem for $Z$ asks to construct a \emph{pattern avoiding} set $E \subset \TT^d$, a set such that for any \emph{distinct} $x_1,\dots,x_n \in E$, $(x_1,\dots,x_n) \not \in Z$, which is as large as possible with respect to some particular statistic relevant to the problem, such as the Hausdorff or Fourier dimension. The main inspiration for the results of this paper were the results of \cite{OurPaper} on `rough' patterns, which constructed, for any set $Z \subset \TT^{dn}$ formed from the countable union of compact sets with lower Minkowski dimension at most $\alpha$, a set $E$ avoiding $Z$ with
\begin{equation} \label{HausdorffDimensionBoundObtainedinOurPaper}
    \hausdim(E) = \max\left( \frac{dn - \alpha}{n - 1}, d \right).
\end{equation}
While the sets $E$ constructed using this method are not guaranteed to be Salem, the techniques in the method (an iterative random selection procedure) seemed amenable to produce Salem sets, since random methods are often very useful for construction Salem sets (randomly constructed functions are often uncorrelated with oscillation at any particular frequency, and thus tend to have good Fourier decay properties). Our goal was thus to modify the construction of \cite{OurPaper} in order to ensure the resulting sets constructed were Salem. The general baseline in the setting of Salem sets was Theorem 38 of \cite{MyThesis}, which constructed a Salem set $E$ avoiding $Z$ with
\begin{equation} \label{FourierDimensionBoundFromMyTHesis}
    \fordim(E) = \max \left( \frac{dn - \alpha}{n}, d \right).
\end{equation}
The natural conjecture is thus whether the $n$ in the denominator of \eqref{FourierDimensionBoundFromMyTHesis} can be replaced with an $n-1$, like in \eqref{HausdorffDimensionBoundObtainedinOurPaper}. In this paper, we are only able to construct Salem sets with dimension matching the bound in \eqref{HausdorffDimensionBoundObtainedinOurPaper} when $Z$ exhibits a weak translation symmetry, as detailed in Theorem \ref{thirdTheorem} of this paper, but for more general sets we are still able to improve upon the dimension given by \eqref{FourierDimensionBoundFromMyTHesis}, as detailed in Theorems \ref{maintheorem} and \ref{theoremJOICVIOJVI122}.

The methods in this paper are generic, in the sense of the Baire category theorem; we define a complete metric space $\mathcal{X}_\beta$ for each $\beta \in (0,d]$, which consists of all pairs $(E,\mu)$, where $E$ is a compact set, and $\mu$ is a Borel probability measure supported on $E$ with $\fordim(\mu) \geq \beta$. We then show that for an appropriate choice of $\beta$, the family of all pairs $(E,\mu) \in \mathcal{X}_\beta$ such that $E$ is Salem and avoids a pattern is \emph{comeager}, or \emph{generic} in $\mathcal{X}_\beta$ (the complement of a set of first category). In particular, the Baire category theorem then guarantees that there exists many sets $E$ with $\fordim(E) = \beta$ which avoid a given pattern.

Many other approaches \cite{OurPaper,PramanikFraser,Keleti} to the pattern avoidance problem construct large pattern avoiding sets explicitly, exchanging nonconstructive Baire category type methods for various constructive \emph{queuing techniques}. The approaches in this paper can be modified to give a constructive queuing argument for the existence of sets with the properties guaranteed in Theorems \ref{maintheorem}, \ref{theoremJOICVIOJVI122}, and \ref{thirdTheorem}. But using Baire category techniques allows us to avoid the technical numerology that goes into queuing arguments, so that we can focus on the more novel aspects of our analysis.

% These methods were pioneered in \cite{Keleti}, who constructed an explicit set avoiding nontrivial solutions to the equation $x_2 - x_1 = x_4 - x_3$, but \cite{PramanikFraser} showed these techniques held for a much more general family of constructions. Other applications of the method are found in \cite{OurPaper} and \cite{MyThesis}. 

% One advantage to showing elements of $\mathcal{X}_\beta$ generically avoid a pattern is that one can construct sets in $\mathcal{X}_\beta$ avoiding a countable family of patterns by showing that a generic element in $\mathcal{X}_\beta$ avoids each pattern, since the countable intersection of comeager sets is comeager. 

Let us now introduce the three primary results of this paper. Theorem \ref{maintheorem} has the weakest conclusions and it's proof relies on a small modifications of previous techniques that exist in the literature, e.g. in \cite{Korner1}. But the theorem works for the most general family of patterns, and it's proof acts as a good foundation on which to build the more novel and powerful theorems in this paper.

\begin{theorem} \label{maintheorem}
    Fix $0 \leq \alpha < dn$, and let $Z \subset \TT^{dn}$ be a compact set with lower Minkowski dimension at most $\alpha$. Set
    \[ \beta_0 = \min \left( \frac{dn - \alpha}{n-1/2}, d \right). \]
    Then there exists a compact Salem set $E \subset \TT^d$ with $\fordim(E) = \beta_0$, such that for any distinct points $x_1, \dots, x_n \in E$, $(x_1, \dots, x_n) \not \in Z$. Moreover, if $\beta \leq \beta_0$, then the family of all pairs $(E,\mu) \in \mathcal{X}_\beta$ such that $E$ is Salem and avoids the pattern generated by $Z$ is comeager.
\end{theorem}

The remaining two results give a more novel analysis, improving upon the result of Theorem \ref{maintheorem} when the pattern $Z$ satisfies additional regularity conditions. In these latter theorems, we focus on patterns specified by equations of the form
\[ x_n = f(x_1,\dots,x_{n-1}), \]
i.e. where $Z = \{ (x_1,\dots,x_n): x_n = f(x_1,\dots,x_{n-1}) \}$, and our assumptions will be on structure of the function $f$. Under the assumption that $f$ is smooth, and satisfies a regularity condition geometrically equivalent to the graph of $f$ being transverse to any axis-oriented $(n-2) \times d$ dimensional plane in $(\RR^d)^{n-1}$, we are able to improve the Fourier dimension bound obtained using previous construction techniques, though not quite enough to match the Hausdorff dimension bound obtained in \cite{PramanikFraser}, except in the fairly trivial case where $n = 2$.

\begin{theorem} \label{theoremJOICVIOJVI122}
    Consider a smooth function $f: V \to \TT^d$, where $V$ is an open subset of $\TT^{d(n-1)}$, such that for each $k \in \{ 1, \dots, n-1 \}$, the matrix
    \[ D_{x_k} f(x_1,\dots,x_{n-1}) = \begin{pmatrix} \frac{\partial f_i}{\partial (x_k)_j} \end{pmatrix}_{1 \leq i,j \leq d} \]
    is invertible whenever $x_1,\dots,x_{n-1}$ are distinct and $(x_1,\dots,x_{n-1}) \in V$. Then there exists a compact Salem set $E \subset \TT^d$ with dimension
    \[ \beta_0 = \begin{cases} d &: n = 2 \\ d/(n - 3/4) &: n \geq 3 \end{cases} \]
    such that for any distinct points $x_1, \dots, x_n \in E$, with $x_1,\dots,x_{n-1} \in V$,
    \[ x_n \neq f(x_1,\dots,x_{n-1}). \]
    Moreover, if $\beta \leq \beta_0$, then the family of pairs $(E,\mu) \in \mathcal{X}_\beta$ such that $E$ is Salem and does not contain any solutions to the equation $x_n = f(x_1,\dots,x_{n-1})$ is comeager.
\end{theorem}

Finally, we consider patterns defined by equations with some translational symmetry. Here we can construct Salem sets with dimension exactly matching the Hausdorff dimension results obtained in \cite{OurPaper}. The simplest example of such a pattern is that specified by a family of equations of the form $m_1x_1 + \dots + m_nx_n = s$, where at least two of the integers $m_1,\dots,m_n \in \ZZ$ is nonzero, and $s$ ranges over a low dimension set in $\TT^d$. But we can also consider more nonlinear patterns, such as those formed by solutions to an equation
\[ x_n - a x_{n-1} = f(x_1,\dots,x_{n-2}) \]
for a non-zero rational number $a$, and a locally Lipschitz function $f$. Even in the case where $f$ is linear, this theorem gives new results.

\begin{theorem} \label{thirdTheorem}
    Fix $d(n-1) \leq \alpha < dn$, a non-zero rational number $a$, and a locally Lipschitz function $f: V \to \RR$, where $V$ is an open subset of $\TT^{d(n-2)}$. Suppose $F \subset \RR$ is a compact set with lower Minkowski dimension at most $\alpha - d(n-1)$. Set
    \[ \beta_0 = \frac{dn - \alpha}{n - 1}. \]
    Then there exists a compact Salem set $E \subset \TT^d$ with $\fordim(E) = \beta_0$ such that for any distinct points $x_1,\dots,x_n \in E$, with $(x_1,\dots,x_{n-2}) \in V$,
    \[ x_n - a x_{n-1} - f(x_1,\dots,x_{n-2}) \not \in F. \]
    Moreover, if $\beta \leq \beta_0$, then the family of all pairs $(E,\mu) \in \mathcal{X}_\beta$ such that $E$ is Salem and does not contain distinct points satisfying the equation above is comeager in $\mathcal{X}_\beta$.
\end{theorem}

%\begin{theorem} \label{thirdTheorem}
%    Fix $d(n-1) \leq \alpha < dn$, a non-zero rational number $a$, and a locally Lipschitz function $T: V \to \mathcal{E}$, where $V$ is an open subset of $\TT^{d(n-2)}$, and where $\mathcal{E}$ is the family of all compact subsets of $\TT^d$, equipped with the Hausdorff distance metric. Suppose that the sets $T(z)$ locally uniformly have lower Minkowski dimension at most $\alpha - d(n-1)$. Set
    % dn 
%    \[ \beta_0 = \min \left( \frac{dn - \alpha}{n-1}, d \right). \]
    %
%    Then there exists a compact Salem set $E \subset \TT^d$ with $\fordim(E) = \beta_0$, such that for any distinct points $x_1,\dots,x_n \in E$, with $(x_1,\dots,x_{n-2}) \in V$,
    %
%    \[ x_n - a x_{n-1} \not \in T(x_1,\dots,x_{n-2}). \]
    %
%    Moreover, if $\beta \leq \beta_0$, then the family of all pairs $(E,\mu) \in \mathcal{X}_\beta$ such that $E$ is Salem and avoids the pattern generated by $Z$ is comeager.
%\end{theorem}

%\begin{remark}
%    When we say the sets $T(z)$ locally uniformly have lower Minkowski dimension at most $\alpha - d(n-1)$, we mean more precisely that for any $\gamma > \alpha$, and any closed set $W \subset V$, there exists a decreasing sequence $\{ r_i \}$ with $\lim_{i \to \infty} r_i = 0$ such that for $x \in W$, the measure of the $r_i$-neighborhood of $T(z)$, which we denote by $N(T(z),r_i)$, satisfies a bound of the form
    %
%    \[ |N(T(z), r_i)| \leq r_i^{dn-\gamma}. \]
    %
%    for instance, this will be true if 
%\end{remark}

Under the assumptions of Theorem \ref{thirdTheorem}, the set
\[ Z = \{ (x_1,\dots,x_n) \in \TT^d: x_n - ax_{n-1} - f(x_1,\dots,x_{n-2}) \in F \} \]
has lower Minkowski dimension $\alpha$. Observing the value $\beta_0$ in that Theorem, we see that in this setting, we have attained precisely the dimension bound obtained in \cite{OurPaper}, but now in the setting of Salem sets rather than just Hausdorff dimension, though restricted to a more specialized set of patterns.

%A basic example of a function $T$ to which Theorem \ref{thirdTheorem} applies is obtained by fixing a locally Lipschitz continuous function $f: V \to \TT^{d(n-2)}$, and setting
%
%\[ T(x_1,\dots,x_{n-2}) = \{ f(x_1,\dots,x_{n-2}) \}, \]
%
%i.e. $T(x_1,\dots,x_{n-2})$ is the set of size one containing solely $f(x_1,\dots,x_{n-2})$. Applying Theorem \ref{thirdTheorem}, we find a Salem set $E$ with dimension $d/(n-1)$ which avoids solutions to the equation $x_n - a x_{n-1} = f(x_1,\dots,x_{n-2})$ for any $a \in \QQ$ and distinct $x_1,\dots,x_n \in E$. The advantage of considering a `multi-valued function' $T: V \to \mathcal{E}$ instead of a `single-valued' function $f: V \to \TT^d$ in Theorem \ref{thirdTheorem} is that we can consider the constructions of Salem sets $E$ which avoid an \emph{uncountable} family of equations of the form
%
%\[ x_n - ax_{n-1} = f_\alpha(x_1,\dots,x_{n-2}) \]
%
%where $\{ f_\alpha \}$ is an uncountable family of functions which are uniformly Lipschitz functions, and have the property that the sets $T(x_1,\dots,x_{n-2}) = \{ f_\alpha(x_1,\dots,x_{n-2}) \}$ have lower Minkowski dimension at most $\alpha - (n-1)d$, locally uniformly in the sense specified by Theorem \ref{thirdTheorem}. An application of this property to avoiding solutions to an uncountable family of linear equations is detailed in Section \ref{ApplicationsSection}.

\begin{remarksaboutresults}
    \ 
    \begin{enumerate}[leftmargin=*]
        \item[1.] Because we are using Baire category techniques, the results we obtain remain true when, instead of avoiding a single pattern, we avoid a countable family of patterns. This is because the countable union of sets of first category is a set of first category. As an example of this property, we note that the conclusion of Theorem \ref{maintheorem} holds when $Z$ is replaced by a \emph{countable union} of compact sets, each with lower Minkowski dimension at most $\alpha$. Similar generalizations apply to Theorems \ref{theoremJOICVIOJVI122} and \ref{thirdTheorem}.

        \item[2.] It is quite surprising that we are able to generically improve the bound \eqref{FourierDimensionBoundFromMyTHesis} in the sense of the Baire category theorem. In \cite{Schmerkin2}, for each $\beta$, a natural probability measure, the \emph{fractal percolation}, is constructed on the space of all $\beta$-dimensional subsets of $\TT^d$, and the question of for which $\beta$ a fractal percolation almost surely contains a given pattern is studied. Such fractal percolations will almost surely be Salem sets, which connects such questions to our problem. The patterns they study in their paper are slightly different from the regimes of this paper under which sharper results are obtained, i.e. as given in Theorems \ref{theoremJOICVIOJVI122} and \ref{thirdTheorem}. But even using the more general result of Theorem \ref{maintheorem} we see that the parameters under which a random $\beta$ dimensional set almost surely contains a pattern behave quite differently to the parameters under which a generic $\beta$ dimensional set in $\chi_\beta$ contains a pattern.

        More precisely, consider the set $Z \subset \TT^{dn}$ defined to be the family of all translations and dilations of a fixed tuple of points $(y_1,\dots,y_n) \in \TT^{dn}$. Then $Z$ has dimension $d+1$, which we will label $\alpha$. Theorem 1.1 of \cite{Schmerkin2} implies that for any $\beta > (dn - \alpha)/n$, a random fractal percolation $E \subset \TT^d$ will almost surely contain points $x_1,\dots,x_n \in E$ such that $(x_1,\dots,x_n) \in Z$. On the other hand, Theorem \ref{maintheorem} of this paper implies that for any $\beta \leq (dn - \alpha)/(n-1/2)$, a generic element of $\mathcal{X}_\beta$ \emph{does not} contain any instances of the pattern specified by $Z$. Thus the theorems in this paper imply that generic pattern avoidance in a Baire category sense differs from generic pattern avoidance in a probabilistic sense, at least in the range $(dn - \alpha)/n < \beta \leq (dn - \alpha)/(n - 1/2)$.

        \item[3.] If $Z \subset \TT^{dn}$ is a compact set with lower Minkowski dimension $\alpha$ with $0 \leq \alpha < d$, then for any $1 \leq i \leq d$, the set $\TT^d - \pi_i(Z)$ has full Hausdorff dimension $d$, where $\pi_i: \TT^{dn} \to \TT^d$ is given by $\pi_i(x_1,\dots,x_n) = x_i$. Thus the pattern avoidance problem is trivial in this case for Hausdorff dimension. This is no longer true when studying Fourier dimension, since $\TT^d - \pi_i(Z)$ need not be a Salem set, nor even have particularly large Fourier dimension, making the results of this paper still interesting in this range.

        That this is true is hinted at in Example 8 of \cite{Ekstrom2014}, where it is shown that there exists a set $X \subset \TT$ which is the countable union of compact sets $\{ X_k \}$, with $\sup_k \minkdim(X_k) \leq 3/4$, such that $\fordim(\TT - X) \leq 3/4$. Thus $\TT - X$ is not a Salem set, since $\TT$ has Hausdorff dimension one. The pattern
        \[ Z = \bigcup_{i = 0}^{n-1} \mathbf{Q}^i \times X \times \mathbf{Q}^{n-i-1}, \]
        is a countable union of compact sets, each with Minkowski dimension at most $3/4$. On the other hand, for each $i \in \{ 1, \dots, n \}$, we find that
        \[ \fordim(\TT - \pi_i(Z)) \leq \fordim(\TT - X) \leq 3/4. \]
        Thus the trivial solution obtained by removing a projection of $Z$ onto a particular coordinate axis does not necessarily give a pattern avoiding set with optimal Fourier dimension in this setting. Applying Theorem \ref{maintheorem} naively to the pattern $Z$ shows that a generic Salem set $E \subset \TT$ of dimension $(n-3/4)/(n-1/2)$ avoids $Z$, which exceeds the dimension of the trivial construction for all $n > 1$. In fact, a generic Salem set $E \subset \TT$ of dimension $1$ will avoid $Z$, since any subset of $\TT - X$ will avoid $Z$, and Theorem \ref{maintheorem} applied with $Z = X$ implies that a generic Salem set $E$ of dimension 1 will be contained in $\TT - X$.

%        Noting that any set avoiding the pattern $Z' = F$ also avoids $Z$, then Theorem \ref{maintheorem}, applied to the pattern $Z'$, actually shows that a generic full dimensional Salem set $E \subset \TT$ avoids $Z$. This is essentially just a proof that [0,1] - F is a set of full Fourier dimension, so that the pattern here is nontrivial. But a proof of this fact also seems nontrivial in this setting.

%        We have essentially just proved that $[0,1] - F$ is a set of full Fourier dimension.

 %       directly applied to $Z$ constructs a Salem set $E$ avoiding $Z$ with
        %
  %      \[ \fordim(E) = (n - 3/4)/(n-1/2) > 3/4. \]
        %
   %     This is not optimal (since applying Theorem \ref{maintheorem} to the set $Z' = F$ yields a full dimensional Salem set, which also avoids the pattern $E$) but at least indicates how avoiding

    %    (though beats the trivial construction), but if we instead apply Theorem \ref{maintheorem} to the set $Z' = F$, then we obtain a set with full Fourier dimension avoiding $Z$ (this also works as an indirect proof of the fact that $[0,1] - F$ has full Fourier dimension, since Theorem \ref{maintheorem} guarantees the existence of $E \subset [0,1] - F$ with full Fourier dimension).

        \item[4.] If $n = 2$, the problem of avoiding solutions to the equation $y = f(x)$ for a continuous function $f: V \to \TT^d$ is essentially trivial. If there exists $x \in \TT^d$ such that $f(x) \neq x$, there there exists an open set $U$ around $x$ such that $U \cap f(U) = \emptyset$. Then $U$ has full Fourier dimension, and avoids solutions to the equation $y = f(x)$. On the other hand, if $f(x) = x$ for all $x$, then there are no distinct $x$ and $y$ in $[0,1]$ such that $y = f(x)$, and so the problem is also trivial. But it is a less trivial to argue that a \emph{generic} set with full Fourier dimension avoids this pattern, which is proved in Theorem \ref{theoremJOICVIOJVI122}, so we still obtain nontrivial information in this case.

        \item[5.] Working on patterns on $\RR^d$ is not significantly different from working over $\TT^d$. For our purposes, the latter domain has several notational advantages, which is why in this paper we have chosen to work with the pattern avoidance pattern in this setting. But there is no theoretical obstacle in applying the techniques described here to prove the existence of pattern avoiding sets in $\RR^d$. Let us briefly describe how this can be done. Given a finite Borel measure $\mu$ on $\RR^d$, we define the Fourier dimension $\fordim(\mu)$ of $\mu$ to be the supremum of all $s \in [0,d]$ such that $\sup_{\xi \in \RR^d} |\widehat{\mu}(\xi)| |\xi|^{s/2} < \infty$, and define $\fordim(E)$ for a Borel set $E \subset \RR^d$ to be the supremum of the quantities $\fordim(\mu)$, taken over finite Borel measures $\mu$ supported on $E$. It is a simple consequence of the Poisson summation formula that if $\mu$ is a compactly supported finite measure on $\RR^d$, and we consider the \emph{periodization} $\mu^*$ of $\mu$, i.e. the finite Borel measure on $\TT^d$ such that for any $f \in C(\TT^d)$,
    \begin{equation}
        \int_{\TT^d} f(x)\; d\mu^*(x) = \int_{\RR^d} f(x)\; d\mu(x),
    \end{equation}
    then $\fordim(\mu^*) = \fordim(\mu)$. A proof is given in Lemma 39 of \cite{MyThesis}. Since $\mu$ is compactly supported, it is also simple to see that $\hausdim(\mu^*) = \hausdim(\mu)$. It follows that if $E$ is a compact subset of $[0,1)^d$, and $\pi: [0,1)^d \to \TT^d$ is the natural projection map, then $\fordim(E) = \fordim(\pi(E))$ and $\hausdim(E) = \hausdim(\pi(E))$. These results therefore imply we can reduce the study of patterns on $\RR^{dn}$ to patterns on $\TT^{dn}$, and thus obtain analogous results to Theorems \ref{maintheorem}, \ref{theoremJOICVIOJVI122}, and \ref{thirdTheorem} for the construction of Salem sets avoiding patterns in $\RR^d$, where every instance of $\TT$ in the statement of those theorems is replaced by $[0,1]$.
    \end{enumerate}
\end{remarksaboutresults}

\section{Notation} \label{notationSection}

\begin{itemize}[leftmargin=*]
%    \item For a positive integer $N$, we let $[N] = \{ 1, \dots, N \}$.

    \item Given a metric space $X$, a point $x \in X$, and a positive number $\varepsilon > 0$, we let $B_\varepsilon(x)$ denote the open ball of radius $\varepsilon$ around $x$. For $x \in X$, we let $\delta_x$ denote the Dirac delta measure at $x$. For a set $E \subset X$ and $\varepsilon > 0$, we let $N(E,\varepsilon) = \bigcup_{x \in E} B_\varepsilon(x)$ denote the \emph{$\varepsilon$-neighborhood} of the set $E$. For two sets $E_1,E_2 \subset X$, we let
    \[ d(E_1,E_2) = \inf \{ d(x_1,x_2) : x_1 \in E_1, x_2 \in E_2 \}, \]
    and then define the \emph{Hausdorff distance}
    \[ d_{\mathbb{H}}(E_1,E_2) = \max \left( \sup_{x_1 \in E_1} d(x_1,E_2), \sup_{x_2 \in E_2} d(E_1,x_2) \right). \]
    If $d_{\mathbb{H}}(E_1,E_2) < \varepsilon$, then $E_1 \subset N(E_2,\varepsilon)$ and $E_2 \subset N(E_1,\varepsilon)$, and the Hausdorff distance can be described as the infinum of such $\varepsilon$.

    \item A subset of a metric space $X$ is of \emph{first category}, or \emph{meager} in $X$ if it is the countable union of closed sets with empty interior, and is \emph{comeager} if it is the complement of such a set (a countable intersection of open, dense sets). We say a property holds \emph{quasi-always}, or a property is \emph{generic} in $X$, if the set of points in $X$ satisfying that property is comeager. The Baire category theorem then states precisely that any comeager set in a complete metric space is dense.

    \item We let $\TT^d = \RR^d/\ZZ^d$. Given $x \in \TT$, we let $|x|$ denote the minimal absolute value of an element of $\RR$ lying in the coset of $x$. For $x \in \TT^d$, we let $|x| = \sqrt{|x_1|^2 + \dots + |x_d|^2}$. The canonical metric on $\TT^d$ is then given by $d(x,y) = |x - y|$, for $x,y \in \TT^d$.

    For an axis-oriented cube $Q$ in $\TT^d$, and some $t > 0$, we let $tQ$ be the axis-oriented cube in $\TT^d$ with the same center and $t$ times the sidelength.

    We say a family of subsets $\mathcal{A}$ of $\TT^d$ is \emph{downward closed} if, whenever $E \in \mathcal{A}$, any subset of $E$ is also an element of $\mathcal{A}$. The quintessential downward closed family for our purposes, given a set $Z \subset \TT^{dn}$, is the collection of all sets $E \subset \TT^d$ that avoid the pattern $Z$, i.e. such that for any distinct points $x_1,\dots,x_n \in E$, $(x_1,\dots,x_n) \not \in Z$.

    \item For $\alpha \in [0,d]$ and $\delta > 0$, the $(\alpha,\delta)$ \emph{Hausdorff content} of a Borel set $E \subset \TT^d$ is
    \[ H^\alpha_\delta(E) = \inf \left\{ \sum_{k = 1}^\infty \varepsilon_k^\alpha : E \subset \bigcup_{k = 1}^\infty B_{\varepsilon_k}(x_k)\ \text{and $0 < \varepsilon_k \leq \delta$ for all $k \geq 1$} \right\}. \]
    The $\alpha$ dimensional Hausdorff measure of $E$ is equal to
    \[ H^\alpha(E) = \lim_{\delta \to 0} H^\alpha_\delta(E). \]
    The Hausdorff dimension $\hausdim(E)$ of a Borel set $E$ is then the supremum over all $s \in [0,d]$ such that $H^s(E) = 0$.
    %the infinum over all $s \in [0,d]$ such that $H^s(E) = \infty$, or alternatively, 

    %Frostman's lemma (see \cite{Mattila}, Chapter 8) says that if we define the Hausdorff dimension $\hausdim(\mu)$ of a finite Borel measure $\mu$ as the supremum of all $s \in [0,d]$ such that
    %
    %\[ \sup_{x \in \TT^d, \varepsilon > 0} \mu(B_\varepsilon(x)) \varepsilon^{-\alpha}, \]
    %
    %then $\hausdim(E)$ is the supremum of $\hausdim(\mu)$, over all Borel probability measures $\mu$ supported on $E$. This gives a specification of the Hausdorff dimension analogous to the definition of the Fourier dimension of a set $E$ given in the introduction.

%    A family of sets $\{ F_\alpha \}$ is a \emph{strong cover} of a set $E$ if each point in $E$ is contained in infinitely many of the sets $\{ F_\alpha \}$. In Lemma 7.5 of \cite{Tao}, it is proved that if $E$ is a compact subset of $\mathbb{E}$ with $\hausdim(E) \leq \alpha$, then there exists a family of sets $\{ E_\delta \}$, for each $\delta$ ranges over all \emph{hyperdyadic numbers} of the form $2^{-\lfloor(1 + \varepsilon)^k \rfloor}$ for some $k$, where $E_\delta$ is the union of dyadic cubes with sidelength $\delta$ and $|E_\delta| \leq r^{d-\alpha-\varepsilon}$

    For a measurable set $E \subset \TT^d$, we let $|E|$ denote its Lebesgue measure. We define the lower Minkowski dimension of a compact Borel set $E \subset \TT^d$ as
    \[ \lowminkdim(E) = \liminf_{r \to 0} d - \log_r|N(E,r)|. \]
    Thus $\lowminkdim(E)$ is the largest number such that for $\alpha < \lowminkdim(E)$, there exists a decreasing sequence $\{ r_i \}$ with $\lim_{i \to \infty} r_i = 0$ and $|N(E,r_i)| \leq r_i^{d - \alpha}$ for each $i$.
    %In Theorem \ref{thirdTheorem}, we consider a family of sets $\{ T(x_1,\dots,x_{n-2}) \}$ which `locally uniformly' had this property.

    \item At several points in this paper we will need to employ probabilistic concentration bounds. In particular, we use \emph{McDiarmid's inequality}. Let $S$ be a set, let $\{ X_1, \dots, X_N \}$ be an independent family of $S$-valued random variables, and consider a function $f: S^N \to \CC$. Suppose that for each $i \in \{ 1, \dots, N \}$, there exists a constant $A_i > 0$ such that for any $x_1, \dots, x_{i-1}, x_{i+1}, \dots, x_N \in S$, and for each $x_i, x_i' \in S$,
    \[ |f(x_1, \dots, x_i, \dots, x_N) - f(x_1, \dots, x_i', \dots, x_N)| \leq A_i. \]
    Then McDiarmid's inequality guarantees that for all $t \geq 0$,
    \[ \PP \left( |f(X_1, \dots, X_N) - \EE(f(X_1, \dots, X_N))| \geq t \right) \leq 4 \exp \left( \frac{-2t^2}{A_1^2 + \dots + A_N^2} \right). \]
    Proofs of McDiarmid's inequality for real-valued functions are given in many probability texts, for instance, in Theorem 3.11 of \cite{VanHandel}, but can be trivially extended to the complex-valued case by taking a union bound to the inequality for real and imaginary values of $f$.

    A special case of McDiarmid's inequality is \emph{Hoeffding's Inequality}. Hoeffding's inequality is often stated in slightly different ways depending on the context; In this paper we use the following formulation: if $\{ X_1, \dots, X_N \}$ is a family of independent random variables, such that for each $i$, there exists a constant $A_i \geq 0$ such that $|X_i| \leq A_i$ almost surely, then for each $t \geq 0$,
    \[ \PP \left( |(X_1 + \dots + X_N) - \EE(X_1 + \dots + X_N)| \geq t \right) \leq 4 \exp \left(\frac{-t^2}{2(A_1^2 + \dots + A_N^2)} \right). \]
\end{itemize}

\section{Applications of our Results} \label{ApplicationsSection}

\subsection{Arithmetic Patterns}

An important problem in current research on pattern avoidance is to construct sets $E$ which avoid \emph{linear patterns}, i.e. sets $E$ which avoid solutions to equations of the form
\begin{equation}
    m_1x_1 + \dots + m_nx_n = 0
\end{equation}
for distinct points $x_1,\dots,x_n \in E$.
%
%But our theorems can still be applied to certain arithmetic patterns to give new results, which we discuss in this section. We approach a standard nonlinear pattern avoidance problem of this type in the next section.
%
This is one scenario in which we know \emph{upper bounds} on the Fourier dimension of pattern avoiding sets. It is simple to prove that if $E \subset \TT^d$, $\fordim(E) > 2d/n$, and $m_1,\dots,m_n$ are non-zero integers, then $m_1 E + \dots + m_n E$ is an open subset of $\TT^d$. This is obtained by a simple modification of the argument of \cite[Proposition 3.14]{MattilaFourier}. Thus there exists \emph{some} choice of integers $m_1,\dots,m_n$ and distinct points $x_1,\dots,x_n \in E$ such that $m_1x_1 + \dots + m_nx_n = 0$. Recently, under the same assumptions, Liang and Pramanik \cite{LiangPramanik} have shown that for $d = 1$, one can choose these integers $m_1,\dots,m_n$ to satisfy $m_1 + \dots + m_n = 0$. These results drastically contrasts the Hausdorff dimension setting, where there exists sets $E \subset \TT^d$ with $\hausdim(E) = d$ which are linearly independent over the rational numbers, and thus avoid nontrivial solutions to integer equations of an arbitrary size (see \cite{Keleti} for a discussion of the case where $d = 1$, whose proof can be adapted to higher dimensions).

If $\fordim(E) > d/n$, and $m_1,\dots,m_n \neq 0$, then the set $m_1 E + \dots + m_n E$ has \emph{positive Lebesgue measure} \cite[Proposition 3.14]{MattilaFourier}. This does not necessarily mean that $E$ will necessarily contain solutions to the equation $m_1x_1 + \dots + m_nx_n = 0$, but indicates why it might be difficult to  push past the current Fourier dimension estimates obtained in \eqref{FourierDimensionBoundFromMyTHesis}, which construct sets with Fourier dimension $d/n$ avoiding such patterns. The first success in pushing past this barrier was the main result of \cite{Korner2}, which showed that for each $n > 0$, there exists a set $E \subset \TT$ with Fourier dimension $1/(n-1)$ such that for any integers $m_1,\dots,m_n \in \ZZ$, not all zero, and any distinct $x_1,\dots,x_n \in \TT^d$, $m_1x_1 + \dots + m_nx_n \neq 0$. The technique used to control Fourier decay in that paper (bounding first derivatives of distribution functions associated with the construction of a random family of pattern avoiding sets) relies heavily on the one dimensional nature of the problem, which makes it difficult to generalize the proof technique to higher dimensions. The results of this paper imply a more robust $d$-dimensional generalization of the result of \cite{Korner2}.

\begin{theorem} \label{dDimensionalKornerResult}
    Suppose $F \subset \TT^d$ is the countable union of a family of compact sets, each with lower Minkowski dimension $\alpha$. Then there exists a Salem set $E \subset \TT^d$ of dimension $(d - \alpha)/(n-1)$ such that for any $x \in F$, any distinct $x_1,\dots,x_n \in E$, and any integers $m_1,\dots,m_n \in \ZZ$, $m_1x_1 + \dots + m_nx_n \neq x$. Moreover, if $\beta_0 = d/(n-1)$, then for any $\beta \leq \beta_0$, and for a generic set $(E,\mu) \in \mathcal{X}_\beta$, the set $E$ has this property.
\end{theorem}
\begin{proof}
    It will suffice to show that a generic set $(E,\mu) \in \mathcal{X}_\beta$ is Salem and avoids solutions to equations of the form
    \begin{equation} \label{equationOIDJDIOJOIJ2131232143fdfef}
        x_n - a_{n-1} x_{n-1} = x + a_3x_3 + \dots + a_nx_n,
    \end{equation}
    with $a_2,\dots,a_n \in \QQ$, $x \in F$, and where either $a_{n-1} \neq 0$, or $a_2 = a_3 = \dots = a_n = 0$. Without loss of generality, we may assume $F$ is compact and has lower Minkowski dimension $\alpha$. If $a_{n-1} \neq 0$, then Theorem \ref{thirdTheorem} applies directly to the equation
    \begin{equation} \label{FirstKindofEquation}
        x_n - a_{n-1} x_{n-1} - f(x_1,\dots,x_{n-2}) \in F,
    \end{equation}
    where $f(x_1,\dots,x_{n-2}) = a_1x_1 + \dots + a_{n-2}x_{n-2}$. Applying Theorem \ref{thirdTheorem}, we conclude that the set of $(E,\mu) \in \mathcal{X}_\beta$ such that $E$ is Salem and avoids solutions to \eqref{FirstKindofEquation} is comeager. On the other hand, if $a_2 = a_3 = \dots = a_n = 0$, then the equation we must avoid is precisely
    \begin{equation} \label{SecondKindofEquation}
        x_1 \in S,
    \end{equation}
    and it follows from Theorem \ref{maintheorem} with $Z = S$ and $n = 1$ that the set of $(E,\mu) \in \mathcal{X}_\beta$ such that $E$ is Salem and avoids solutions to \eqref{SecondKindofEquation} is comeager. Taking countable unions ranging over the choices of coefficients $a_2,\dots,a_n$ shows that the set of all $(E,\mu) \in \mathcal{X}_\beta$ such that $E$ is Salem and avoids all $n$-variable linear equations is comeager, which completes the proof.
\end{proof}

\begin{remark}
    For \emph{particular} linear patterns, it is certainly possible to improve the result of Theorem \ref{dDimensionalKornerResult}. For instance, Schmerkin \cite{Schmerkin} constructed a set $E \subset \TT$ with $\fordim(E) = 1$ which contains no three numbers forming an arithmetic progressions, i.e. solutions to the linear equation
    \[ (x_3 - x_2) - (x_2 - x_1) = 0. \]
    This equation can also be written as
    \[ x_3 - 2x_1 + x_1 = 0. \]
    Liang and Pramanik \cite{LiangPramanik} generalized this technique by constructing, for any finite family of translation-invariant linear functions $\{ f_i \}$, a set $E \subset \TT$ with $\fordim(E) = 1$ such that for distinct $x_1,\dots,x_n \in E$, and any index $i$, $f_i(x_1,\dots,x_n) \neq 0$. This same paper even constructs a set with Fourier dimension close to one avoiding an uncountable family of translation-invariant linear functions, though only those that are of a very special form. The advantage of Theorem \ref{dDimensionalKornerResult} is that it applies to a very general family of uncountably many linear equations, though one does not obtain as high a Fourier dimension bound as those obtained in \cite{LiangPramanik} and \cite{Schmerkin}. Nonetheless, though we construct sets of dimension $d/(n-1)$, we still remain quite far away from the best known upper bound $2d/n$ of the Fourier dimension of a set avoiding general integer linear equations.
\end{remark}

The arguments in this paper are heavily inspired by the techniques of \cite{Korner2}, but augmented with some more robust probabilistic concentration inequalities and oscillatory integral techniques, which enables us to push the results of \cite{Korner2} to a much more general family of patterns. In particular, Theorem \ref{thirdTheorem} shows that the results of that paper do not depend on the rich arithmetic structure of the equation $m_1x_1 + \dots + m_nx_n = 0$, but rather only on a very weak translation invariance property of the pattern. We are unable to close the gap between the upper bound $2d/n$ of sets avoiding $n$-variable linear equations for $n \geq 3$, which would seem to require utilizing the full linear nature of the equations involved much more heavily than the very weak linearity assumption that Theorem \ref{dDimensionalKornerResult} requires.

\subsection{Isosceles Triangles on Curves, and other Nonlinear Patterns}

Theorems \ref{maintheorem}, \ref{theoremJOICVIOJVI122}, and \ref{thirdTheorem} can be applied to find sets avoiding linear patterns, but the main power of these results that they can be applied to `nonlinear' patterns which cannot be analyzed quite as easily via the Fourier transform. As a result, relatively few results exist showing that sets with large Fourier dimension avoid patterns, though some partial results are given in \cite{HenriotLabaPramanik} for a slightly different regime than that considered here; we do not even know if sets $E \subset \TT$ with Fourier dimension one contain one of the simplest nonlinear patterns $\{ x, y, z, z + (y - x)^2 \}$, and results like that of \cite{Schmerkin} show that even in the linear setting it is difficult to conjecture what might be optimal in this setting. However, Theorem \ref{thirdTheorem} applies to this pattern, since the pattern is specified by the equation $z - x = (y - w)^2$, obtaining a Salem set $E \subset \TT$ of dimension $1/3$ avoiding this pattern, the first proof of the existence of a set with positive Fourier dimension avoiding isosceles triangles in the literature.

In this section, we consider a more geometric example of a non-linear pattern, finding subsets of curves which do not contain isosceles triangles; given a simple segment of a curve given by a smooth map $\gamma : [0,1] \to \RR^d$, we say a set $E \subset [0,1]$ \emph{avoids isosceles triangles on $\gamma$} if for any distinct values $t_1,t_2,t_3 \in [0,1]$, $|\gamma(t_1) - \gamma(t_2)| \neq |\gamma(t_2) - \gamma(t_3)|$, i.e. if $\gamma(E)$ does not contain any three points forming the vertices of an  isosceles triangle. In \cite{PramanikFraser}, methods are provided to construct sets $E \subset [0,1]$ with $\hausdim(E) = \log_3 2 \approx 0.63$ such that $\gamma(E)$ does not contain any isosceles triangles, but $E$ is not guaranteed to be Salem. Using the results of this paper, we can now construct \emph{Salem sets} $E \subset [0,1]$ with $\fordim(E) = 4/9 \approx 0.44$, such that $\gamma(E)$ does not contain any isosceles triangles.

\begin{theorem}
    For any smooth map $\gamma: [0,1] \to \RR^d$ with $\gamma'(t) \neq 0$ for all $t \in [0,1]$, there exists a Salem set $E \subset [0,1]$ with $\fordim(E) = 4/9$ which avoids isosceles triangles on $\gamma$.
\end{theorem}
\begin{proof}
    Assume without loss of generality (working on a smaller portion of the curve if necessary and then rescaling) that there exists a constant $C \geq 1$ such that for any $t,s \in [0,1]$,
    \begin{equation} \label{equationOIVOIJOIJOIJ1231231245532}
        |\gamma(t) - \gamma(s) - (t - s)\gamma'(0)| \leq C (t - s)^2,
    \end{equation}
    \begin{equation} \label{equationDOIJCOIJCOIJCOIJ}
        1/C \leq |\gamma'(t)| \leq C,
    \end{equation}
    and
    \begin{equation} \label{equationCIOJAOIVJVOIJioj1312421541}
        |\gamma'(t) - \gamma'(s)| \leq C |t - s|.
    \end{equation}
    Let $\varepsilon = 1/2C^3$, and let
    \begin{equation}
        F(t_1,t_2,t_3) = |\gamma(t_1) - \gamma(t_2)|^2 - |\gamma(t_2) - \gamma(t_3)|^2.
    \end{equation}
    A simple calculation using \eqref{equationOIVOIJOIJOIJ1231231245532} and \eqref{equationDOIJCOIJCOIJCOIJ} reveals that for $0 \leq t_1,t_2 \leq \varepsilon$,
    \begin{equation} \label{equationCOIJAWOIJCAWOIJWOAI2112412}
        \left| \frac{\partial F}{\partial t_1} \right| = 2 \left| (\gamma(t_1) - \gamma(t_2)) \cdot \gamma'(t_1) \right| \geq (2/C) |t_2 - t_1| - 2C |t_2 - t_1|^2 \geq (1/C) |t_2 - t_1|.
    \end{equation}
    This means that $\partial F / \partial t_1 \neq 0$ unless $t_1 = t_2$. Thus the implicit function theorem implies that there exists a countable family of smooth functions $\{ f_i: U_i \to [0,1] \}$, where $U_i \subset [0,\varepsilon]^2$ for each $i$ and $f_i(t_2,t_3) \neq t_3$ for any $(t_2,t_3) \in U_i$, such that if $F(t_1,t_2,t_3) = 0$ for distinct points $t_1,t_2,t_3 \in [0,\varepsilon]$, then there exists an index $i$ with $(t_2,t_3) \in U_i$ and $t_1 = f_i(t_2,t_3)$. Differentiating both sides of the equation
    \begin{equation}
        |\gamma(f_i(t_2,t_3)) - \gamma(t_2)|^2 = |\gamma(t_2) - \gamma(t_3)|^2
    \end{equation}
    in $t_2$ and $t_3$ shows that
    \begin{equation} \label{firstPartialDerivative}
        \frac{\partial f_i}{\partial t_2}(t_2,t_3) = \frac{(\gamma(f_i(t_2,t_3)) - \gamma(t_3)) \cdot \gamma'(t_2)}{(\gamma(f_i(t_2,t_3)) - \gamma(t_2)) \cdot \gamma'(f_i(t_2,t_3))}
    \end{equation}
    and
    \begin{equation} \label{secondPartialDerivative}
        \frac{\partial f_i}{\partial t_3}(t_2,t_3) = \frac{- (\gamma(t_2) - \gamma(t_3)) \cdot \gamma'(t_3)}{(\gamma(f_i(t_2,t_3)) - \gamma(t_2)) \cdot \gamma'(f_i(t_2,t_3))}.
    \end{equation}
    In order to apply Theorem \ref{theoremJOICVIOJVI122}, we must show that the partial derivatives in \ref{firstPartialDerivative} and \ref{secondPartialDerivative} are both non-vanishing for $t_2,t_3 \in [0,\varepsilon]$. We calculate using \eqref{equationOIVOIJOIJOIJ1231231245532}, \eqref{equationDOIJCOIJCOIJCOIJ} and \eqref{equationCIOJAOIVJVOIJioj1312421541} that
    \begin{align} \label{equationDOIJCOIJCOIJJOIJddwad12}
    \begin{split}
        |(\gamma(f_i(t_2,t_3)) - \gamma(t_3)) \cdot \gamma'(t_2)| &\geq |(\gamma(f_i(t_2,t_3)) - \gamma(t_3)) \cdot \gamma'(t_3)|\\
        &\quad\quad + |(\gamma(f_i(t_2,t_3)) - \gamma(t_3)) \cdot (\gamma'(t_2) - \gamma'(t_3))|\\
        &\geq (1/C) |f_i(t_2,t_3) - t_3| - C^2 |f_i(t_2,t_3) - t_3||t_2 - t_3|\\
        &\geq (1/C - C^2 \varepsilon) |f_i(t_2,t_3) - t_3|\\
        &\geq (1/2) |f_i(t_2,t_3) - t_3|.
    \end{split}
    \end{align}
    Since $f_i(t_2,t_3) \neq t_3$ for all $(t_2,t_3) \in U_i$, it follows from \eqref{firstPartialDerivative} and \eqref{equationDOIJCOIJCOIJJOIJddwad12} that if $(t_2,t_3) \in U_i$ with $t_2 \neq t_3$,
    \begin{equation} \label{nonvanishingt2derivative}
        \frac{\partial f_i}{\partial t_2}(t_2,t_3) \neq 0.
    \end{equation}
    A similar calculation to \eqref{equationCOIJAWOIJCAWOIJWOAI2112412} shows that for $t_2,t_3 \in [0,\varepsilon]$,
    \begin{equation} \label{thirdPartialDerivativeIsNonvanishing}
        |(\gamma(t_2) - \gamma(t_3)) \cdot \gamma'(t_3)| \geq (1/C) |t_2 - t_3|.
    \end{equation}
    Combining \eqref{secondPartialDerivative} with \eqref{thirdPartialDerivativeIsNonvanishing} shows that for $t_2 \neq t_3$ with $(t_2,t_3) \in U_i$,
    \begin{equation} \label{nonvanishingt3derivative}
        \frac{\partial f_i}{\partial t_3}(t_2,t_3) \neq 0.
    \end{equation}
    Now \eqref{nonvanishingt2derivative} and \eqref{nonvanishingt3derivative} imply that each function in the family $\{ f_i \}$ satisfy the hypothesis of Theorem \ref{theoremJOICVIOJVI122}. Thus that theorem implies that for $\beta = 4/9$, each index $i$, and a generic element of $(E,\mu) \in \mathcal{X}_\beta$, the set $E$ is Salem and for any distinct $t_1,t_2,t_3 \in E \cap [0,\varepsilon]$, $f_i(t_1,t_2,t_3) \neq 0$. This means precisely that $|\gamma(t_1) - \gamma(t_2)| \neq |\gamma(t_2) - \gamma(t_3)|$ for any distinct $t_1,t_2,t_3 \in E$. Thus we conclude we can find a Salem set $E \subset [0,\varepsilon]$ with $\fordim(E) = 4/9$ such that $\gamma(E)$ does not contain the vertices of any isosceles triangles.
\end{proof}

Theorem \ref{maintheorem} can also be used to construct sets with a slightly smaller dimension avoiding isosceles triangles on a rougher family of curves. If we consider a Lipschitz function $\gamma: [0,1] \to \RR^{d-1}$, where there exists $M < 1$ with $|\gamma(t) - \gamma(s)| \leq M |t - s|$ for each $t,s \in [0,1]$, then Theorem 3 of \cite{OurPaper} guarantees that the set
\[ Z = \left\{ (x_1,x_2,x_3) \in [0,1]^3 : \begin{array}{c}
            \text{$(x_1,\gamma(x_1)), (x_2,\gamma(x_2)), (x_3,\gamma(x_3))$}\\
            \text{form the vertices of an isosceles triangle.} \end{array} \right\} \]
has lower Minkowski dimension at most two. Thus Theorem \ref{maintheorem} guarantees that there exists a Salem set $E \subset [0,1]$ with $\fordim(E) = 2/5 = 0.4$ such that $\gamma(E)$ avoids all isosceles triangles. The main result of \cite{PramanikFraser} constructs a set $E \subset [0,1]$ with $\hausdim(E) = 0.5$ such that $\gamma(E)$ avoids all isosceles triangles, but this set is not guaranteed to be Salem.

\section{A Metric Space Controlling Fourier Dimension}

In order to work with a Baire category type argument, we must construct an appropriate metric space appropriate for our task, and establish a set of tools for obtaining convergence in this metric space. In later sections we will fix a specific choice of $\beta$ to avoid a particular pattern. But in this section we let $\beta$ be an arbitrary fixed number in $(0,d]$. Our approach in this section is heavily influenced by \cite{Korner2}. However, we employ a Fr\'{e}chet space construction instead of the Banach space construction used in \cite{Korner2}, which enables us to use softer estimates in our arguments, with the disadvantage that we can obtain only Fourier dimension bounds in Theorems \ref{maintheorem}, \ref{theoremJOICVIOJVI122}, and \ref{thirdTheorem} at the endpoint dimensions $\beta_0$ considered in the theorems, rather than the explicit decay estimates as is obtained, for instance, in Theorem 2.4 of \cite{Korner2}:
\begin{itemize}
    \item We let $\mathcal{E}$ denote the family of all compact subsets of $\TT^d$. If we consider the Hausdorff distance $d_{\mathbb{H}}$ between sets, then $(\mathcal{E},d_\mathbb{H})$ forms a complete metric space. %We note that if a sequence $\{ E_k \}$ converges to a set $E$ in the Hausdorff distance, then $E$ is the collection of all values $\lim_{k \to \infty} x_k$, where $\{ x_k \}$ is a convergent sequence with $x_k \in E_k$ for each $k$.

    \item We let $M_*(\beta)$ consist of all finite Borel measures $\mu$ on $\TT^d$ such that for each $\lambda \in [0,\beta)$,
    \[ \| \mu \|_{M(\lambda)} = \sup_{\xi \in \ZZ^d} |\widehat{\mu}(\xi)| |\xi|^{\lambda/2} \]
    is finite. Then $\| \cdot \|_{M(\lambda)}$ is a seminorm on $M_*(\beta)$ for each $\lambda \in [0,\beta)$, and the collection of all such seminorms gives $M_*(\beta)$ the structure of a Frech\'{e}t space (most importantly, this means $M_*(\beta)$ is a \emph{complete} metric space). Under this topology, a sequence of probability measures $\{ \mu_k \}$ converges to a probability measure $\mu$ in $M_*(\beta)$ if and only if for any $\lambda \in [0,\beta)$, $\lim_{k \to \infty} \| \mu_k - \mu \|_{M(\lambda)} = 0$.
\end{itemize}

We now let $\mathcal{X}_\beta$ be the collection of all pairs $(E,\mu) \in \mathcal{E} \times M_*(\beta)$, where $\mu$ is a probability measure such that $\text{supp}(\mu) \subset E$. Then $\mathcal{X}_\beta$ is a closed subset of $\mathcal{E} \times M_*(\beta)$ under the product metric, and thus a complete metrizable space. We remark that for any $\lambda \in [0,\beta)$ and $(E,\mu) \in \mathcal{X}_\beta$,
\begin{equation} \label{equationGFSCSC4}
    \lim_{|\xi| \to \infty} |\widehat{\mu}(\xi)| |\xi|^{\lambda/2} = 0.
\end{equation}
Thus $\fordim(E) \geq \fordim(\mu) \geq \beta$ for each $(E,\mu) \in \mathcal{X}_\beta$. This means that $\mathcal{X}_\beta$ can be thought of as a space of compact sets, augmented with a measure providing a certification guaranteeing the set's Fourier dimension is at least $\beta$.

\begin{comment}
\begin{theorem}
    $\mathcal{X}$ is a closed subset of $\mathcal{E} \times M(\beta)$.
\end{theorem}
\begin{proof}
    Suppose $\{ (E_k,\mu_k) \}$ is a sequence of elements of $\mathcal{X}$ converging to some tuple $(E,\mu) \in \mathcal{E} \times M(\beta)$. Fix $\varepsilon > 0$. Since $E_k \to E$ in the Hausdorff dimension, there exists $k_0$ such that for $k \geq k_0$, $E_k \subset E(\varepsilon)$. Since $\mu_k \to \mu$ weakly, this implies that $\mu$ is a probability measure, and that $\text{supp}(\mu) \subset E(\varepsilon)$. Taking $\varepsilon \to 0$ shows that $\text{supp}(\mu) \subset E$. Again for a fixed $\varepsilon > 0$, applying the triangle inequality and the reverse triangle inequality combined with \eqref{equationGFSCSC4} applied to $\mu_k$, we conclude
    %
    \[ \lim_{|\xi| \to \infty} |\xi|^{\beta/2 - \varepsilon} |\widehat{\mu}(\xi)| = \lim_{|\xi| \to \infty} |\xi|^{\beta/2 - \varepsilon} |\widehat{\mu}(\xi) - \widehat{\mu_k}(\xi)| \leq \| \mu - \mu_k \|_{M(\beta,\varepsilon)}. \]
    %
    Taking $k \to \infty$ shows that
    %
    \[ \lim_{|\xi| \to \infty} |\xi|^{\beta/2 - \varepsilon} |\widehat{\mu}(\xi)| = 0, \]
    %
    which completes the proof.
\end{proof}
\end{comment}

Lemma \ref{LemmaFIOAJFOIWJ} allows us to reduce the proof of density arguments in $\mathcal{X}_\beta$ to the construction of large discrete subsets in $\TT^d$ with well-behaved Fourier analytic properties. We recall that a family $\mathcal{A}$ of subsets of $\TT^d$ is \emph{downward closed} if, whenever $E \subset \mathcal{A}$, any subset of $E$ is also contained in $\mathcal{A}$.

\begin{lemma} \label{LemmaFIOAJFOIWJ}
    Let $\mathcal{A}$ be a downward closed family of subsets of $\TT^d$. Fix $\beta > 0$, $\kappa > 0$, and a large constant $C > 0$. Suppose that for all small $\delta > 0$, and all $\lambda \in [0,\beta)$, there are arbitrarily large integers $N > 0$ for which there exists a finite set $S = \{ x_1, \dots, x_N \}$, positive numbers $\{ a_1, \dots, a_N \}$ such that $\sum_{k = 1}^N a_k = N$, and a quantity $r \geq (2N)^{-1/\lambda}$ such that the following properties hold:
    \begin{enumerate}
        \item[(1)] $N(S,r) \in \mathcal{A}$.
        \item[(2)] For each $\xi \in \ZZ^d - \{ 0 \}$ with $|\xi| \leq (1/r)^{1 + \kappa}$,
        \[ \left| \frac{1}{N} \sum_{k = 1}^N a_k e^{2 \pi i \xi \cdot x_k} \right| \leq C N^{-1/2} \log(N) + \delta |\xi|^{-\lambda/2} \]
    \end{enumerate}
    Then $\{ (E,\mu) \in \mathcal{X}_\beta : E \in \mathcal{A} \}$ is dense in $\mathcal{X}_\beta$.
\end{lemma}

\begin{remark}
    We will be able to take $\delta = 0$ for the applications of Lemma \ref{LemmaFIOAJFOIWJ} in Theorems \ref{maintheorem} and \ref{thirdTheorem} by applying purely probabilistic arguments which give rise to square root cancellation. We only need to take $\delta > 0$ when applying this result to Theorem \ref{theoremJOICVIOJVI122}, because we must apply some oscillatory integral bounds which give an additional decaying factor as $\xi \to \infty$.
\end{remark}

Let us comment on the intuition underlying Lemma \ref{LemmaFIOAJFOIWJ}. Consider a large integer $N$, and suppose there is a discrete family of $N$ points $S = \{ x_1, \dots, x_N \}$ such that $N(S,r)$ does not contain any incidences of a particular pattern. Then $N(S,r)$ is a union of $N$ balls of radius $r$, so if $N \approx r^{-\beta}$, and these balls do not overlap too much, we might expect $N(S,r)$ to behave like an $r$-thickening of a $\beta$-dimensional set. The Fourier analytic properties of $S$ can be understood by taking exponential sums, i.e. considering quantities of the form
\[ \frac{1}{N} \sum_{k = 1}^N a_k e^{2 \pi i \xi \cdot x_k}. \]
where $a_1,\dots,a_N$ are non-negative and sum to one as in Lemma \ref{LemmaFIOAJFOIWJ}. For any set $S$, taking in absolute values gives a trivial bound on the exponential sum
\begin{equation}
    \left| \frac{1}{N} \sum_{k = 1}^N a(x_k) e^{2 \pi i \xi \cdot x_k} \right| \leq 1,
\end{equation}
and this bound can be tight for general sets $S$, for instance, if $S$ behaves like an arithmetic progression with a frequency $\xi$, as happens for
\[ S = \left\{ k \frac{\xi}{|\xi|^2} : 1 \leq k \leq N \right\}. \]
If one can significantly improve upon this bound, one therefore thinks of $S$ as having additional regularity from the perspective of Fourier analysis. The best case we can hope to hold for a `generic' choice of $S$ is a \emph{square root cancellation bound} of the form
\begin{equation} \label{equationvVVVISJDOIAWJDOIAJIO}
    \left| \frac{1}{N} \sum_{k = 1}^N a_k e^{2 \pi i \xi \cdot x_k} \right| \leq C N^{-1/2},
\end{equation}
which, roughly speaking, holds if the points $\{ x_k \}$ are not significantly periodic at the frequency $\xi$. If $\kappa > 0$ is fixed, and equation \eqref{equationvVVVISJDOIAWJDOIAJIO} holds for all $|\xi| \lesssim (1/r)^{1 + \kappa}$, then one might therefore think of $N(S,r)$ as behaving like an $r$-thickening of a \emph{Salem set} with dimension $\beta$. Note that, up to a logarithmic constant, and a negligible term which decays as $|\xi| \to \infty$, Lemma \ref{LemmaFIOAJFOIWJ} obtains such a square root cancellation bound. Since the assumptions of Lemma \ref{LemmaFIOAJFOIWJ} guarantee that we can construct such sets at arbitrarily small scales $r$, it makes sense that one should be able to use such assumptions to construct Salem sets avoiding patterns.

Lemma \ref{LemmaFIOAJFOIWJ} not only guarantees that we can construct Salem sets avoiding patterns, but also guarantees that such pattern avoiding sets are \emph{generic} among all Salem sets. The reason such a result is possible given the assumptions of Lemma \ref{LemmaFIOAJFOIWJ} is that the points in the discrete set $S$ are not significantly periodic at any frequency $\xi$, which heuristically (à la Weyl's equidistribution theorem) implies such points are evenly distributed in $\TT^d$. We do not need equidistribution here, but we are able to obtain in Lemma \ref{LemmaTAOIAWOIDJ12301} that for any $\varepsilon > 0$, if $N$ is taken appropriately large, then the set $S$ guaranteed by the assumptions of Lemma \ref{LemmaFIOAJFOIWJ} is $\varepsilon$-dense in $\TT^d$, i.e. $\TT^d = N(S,\varepsilon)$. This will allow us to approximate an arbitrary compact set of dimension $\beta$ by subsets of $S$, and thus approximate an arbitrary element of $\mathcal{X}_\beta$ by an element which is pattern avoiding, thus guaranteeing that sets containing patterns are of first category.

Because the proof of Lemma \ref{LemmaFIOAJFOIWJ} is somewhat technical, we relegate it to an appendix found at the end of this paper. In the remainder of this section, we use Lemma \ref{LemmaFIOAJFOIWJ}, together with a probabilistic argument to argue some general results about $\mathcal{X}_\beta$ which will be useful in the proofs of Theorems \ref{maintheorem}, \ref{theoremJOICVIOJVI122}, and \ref{thirdTheorem}.

It is a useful heuristic that in a metric space whose elements are sets, quasi-all elements are as `thin as possible'. In particular, we should expect the Hausdorff dimension and Fourier dimension of a generic element of $\mathcal{X}_\beta$ to be as low as possible. For each $(E,\mu) \in \mathcal{X}_\beta$, the condition that $\mu \in M_*(\beta)$ implies that $\fordim(\mu) \geq \beta$, so $\hausdim(E) \geq \fordim(E) \geq \beta$. Since the Fourier dimension and Hausdorff dimension are lower bounded by $\beta$, and this bound is tight, our heuristic thus leads us to believe that for quasi-all $(E,\mu) \in M_*(\beta)$, the set $E$ has both Hausdorff dimension and Fourier dimension equal to $\beta$, i.e. $E$ is a Salem set of dimension $\beta$. We will finish this section with a proof of this fact. This will also give some more elementary variants of the kinds of probabilistic arguments we will later use to prove Theorems \ref{maintheorem}, \ref{theoremJOICVIOJVI122}, and \ref{thirdTheorem}, which will allow us to become more comfortable with these techniques in preparation for the proofs of these theorems.

\begin{lemma} \label{LemmaGISCICS1}
    Fix a positive integer $N$, and $\kappa > 0$. Let $X_1, \dots, X_N$ be independent random variables on $\TT^d$, such that for each $\xi \in \ZZ^d - \{ 0 \}$,
    \begin{equation} \label{equatioNVVVVSXXJVU1132}
        \sum_{k = 1}^N \EE \left( e^{2 \pi i \xi \cdot X_k} \right) = 0.
    \end{equation}
    Then there exists a constant $C$ depending on $d$ and $\kappa$ such that
    \[ \PP \left( \sup_{|\xi| \leq N^{1 + \kappa}} \left| \frac{1}{N} \sum_{k = 1}^N e^{2 \pi i \xi \cdot X_k} \right| \geq C N^{-1/2} \log(N)^{1/2} \right) \leq 1/10. \]
\end{lemma}

\begin{remark}
    In particular, the assumptions of Lemma \ref{LemmaGISCICS1} hold if the random variables $\{ X_1,\dots, X_N \}$ are uniformly distributed on $\TT^d$, since then $\EE(e^{2 \pi i \xi \cdot X_k}) = 0$ for all $\xi \in \ZZ^d - \{ 0 \}$ and $1 \leq k \leq N$.
\end{remark}

\begin{proof}
    For each $\xi \in \ZZ^d$ and $k \in \{ 1, \dots, N \}$, consider the random variable
    \[ Y(\xi,k) = N^{-1} e^{2 \pi i \xi \cdot X_k}. \]    
    Then for each $\xi \in \ZZ^d$,
    \begin{equation} \label{equationPPDOCS999223}
        \sum_{k = 1}^N Y(\xi,k) = \frac{1}{N} \sum_{k = 1}^N e^{2 \pi i \xi \cdot X_k}.
    \end{equation}
    We also note that for each $\xi \in \ZZ^d$ and $k \in \{ 1, \dots, N \}$,
    \begin{equation} \label{equationGFDSCSXAOOO99}
        |Y(\xi,k)| = N^{-1}.
    \end{equation}
    Moreover,
    \begin{equation} \label{equationDOIJWIJCCCCC5555322}
    \begin{split}
        \sum_{k = 1}^N \EE(Y(\xi,k)) = 0.
    \end{split}
    \end{equation}
    Since the family of random variables $\{ Y(\xi,k) \}$ is independent for a fixed $\xi$, we can apply Hoeffding's inequality together with \eqref{equationPPDOCS999223} and \eqref{equationGFDSCSXAOOO99} to conclude that for all $t \geq 0$,
    \begin{equation} \label{equationDDISCCOXOSPP998323}
        \PP \left( \left| \frac{1}{N} \sum_{k = 1}^N e^{2 \pi i \xi \cdot X_k} \right| \geq t \right) \leq 2 e^{-Nt^2/2}.
    \end{equation}
    Taking a union bound obtained by applying \eqref{equationDDISCCOXOSPP998323} over all $|\xi| \leq N^{1 + \kappa}$ gives the existence of a constant $C \geq 10$ depending on $d$ and $\kappa$ such that
    \begin{equation} \label{equationPPDOCS2424}
        \PP \left( \sup_{|\xi| \leq N^{1 + \kappa}} \left| \frac{1}{N} \sum_{k = 1}^N e^{2 \pi i \xi \cdot X_k} \right| \geq t \right) \leq \exp \left( C \log(N) - \frac{5N t^2}{C} \right).
    \end{equation}
    But then setting $t = CN^{-1/2} \log(N)^{1/2}$ in \eqref{equationPPDOCS2424} completes the proof.
\end{proof}

\begin{lemma} \label{lemmaoiajdoijwdowaj}
    For quasi-all $(E,\mu) \in \mathcal{X}_\beta$, $E$ is a Salem set of dimension $\beta$.
\end{lemma}
\begin{proof}
    We shall assume $\beta < d$ in the proof, since when $\beta = d$, $E$ is a Salem set for any $(E,\mu) \in \mathcal{X}_\beta$, and thus the result is trivial. Since the Hausdorff dimension of a measure is an upper bound for the Fourier dimension, it suffices to show that for quasi-all $(E,\mu) \in \mathcal{X}_\beta$, $E$ has Hausdorff dimension at most $\beta$. For each $\alpha > \beta$ and $\delta, s > 0$, we let
    \[ \mathcal{A}(\alpha,\delta,s) = \{ E \subset \TT^d: H^\alpha_\delta(E) < s \}. \]
    and set
    \[ A(\alpha,\delta,s) = \{ (E,\mu) \in \mathcal{X}_\beta: E \in \mathcal{A}(\alpha,\delta,s) \}. \]
    Then $A(\alpha,\delta,s)$ is an open subset of $\mathcal{X}_\beta$, and
    \begin{equation}
        \bigcap_{n = 1}^\infty \bigcap_{m = 1}^\infty \bigcap_{k = 1}^\infty A(\beta + 1/n, 1/m, 1/k)
    \end{equation}
    is precisely the family of $(E,\mu) \in \mathcal{X}_\beta$ such that $E$ has Hausdorff dimension at most $\beta$. Thus it suffices to show that $A(\alpha,\delta,s)$ is dense in $\mathcal{X}_\beta$ for all $\alpha > \beta$, all $\delta > 0$, and all $s > 0$. Since $\mathcal{A}(\alpha,\delta,s)$ is a \emph{downward closed} family of subsets of $\TT^d$, we may apply Lemma \ref{LemmaFIOAJFOIWJ}. Fix a large integer $N$, and set $r = N^{-1/\beta}$, so that $N \geq (1/2) r^{-\lambda}$ satisfies the condition for Lemma \ref{LemmaFIOAJFOIWJ} to apply to these quantities. Lemma \ref{LemmaGISCICS1} shows that there exists a constant $C$ depending on $\beta$ and $d$, as well as $N$ points $S = \{ x_1, \dots, x_N \} \subset \TT^d$ such that for each $|\xi| \leq N^{1 + \kappa}$,
    \begin{equation} \label{equationDDVVIXXSX23}
        \left| \frac{1}{N} \sum_{k = 1}^N e^{2 \pi i \xi \cdot x_k} \right| \leq C N^{-1/2} \log(N)^{1/2}.
    \end{equation}
    Now $N(S,r)$ is a union of $N$ balls of radius $r$, and thus if $r \leq \delta$,
    \begin{equation}
        H^\alpha_\delta(N(S,r)) \leq N r^\alpha = N^{1 - \alpha / \beta}.
    \end{equation}
    Since $\alpha > \beta$, taking $N$ appropriately large gives a set $N(S,r)$ with
    \begin{equation}
        H^\alpha_\delta(N(S,r)) < s.
    \end{equation}
    Thus $N(S,r) \in \mathcal{A}(\alpha,\delta,s)$ for sufficiently large integers $N$. But together with \eqref{equationDDVVIXXSX23}, this justifies that the hypothesis of Lemma \ref{LemmaFIOAJFOIWJ} applies to this scenario. Thus that lemma implies that $A(\alpha,\delta,s)$ is dense in $\mathcal{X}_\beta$, completing the proof.
\end{proof}

This concludes the setup to the proof of Theorems \ref{maintheorem}, \ref{theoremJOICVIOJVI122}, and \ref{thirdTheorem}. All that remains is to show that quasi-all elements of $\mathcal{X}_\beta$ avoid the given set $Z$ for a suitable parameter $\beta$; it then follows from Lemma \ref{lemmaoiajdoijwdowaj} that quasi-all elements of $\mathcal{X}_\beta$ are \emph{Salem} and avoid the given set $Z$ (since the intersection of two generic subsets of $\mathcal{X}_\beta$ is also generic). The advantage of Lemma \ref{LemmaFIOAJFOIWJ}, combined with a Baire category argument, is that we can now reduce our calculations to finding suitable finite families of points with nice Fourier analytic properties.

\section{Random Avoiding Sets for Rough Patterns}

We begin by proving Theorem \ref{maintheorem}, which requires simpler calculations than Theorem \ref{theoremJOICVIOJVI122} and Theorem \ref{thirdTheorem}. In the last section, our results held for an arbitrary $\beta \in (0,d]$. But in this section, we assume
\[ \beta \leq \min \left( d, \frac{dn - \alpha}{n - 1/2} \right), \]
where $\alpha$ and $n$ are as in the statement of Theorem \ref{maintheorem}. Then $\beta$ is small enough to show that the pattern $Z$ described in Theorem \ref{maintheorem} is avoided by a generic element of $\mathcal{X}_\beta$. The construction here is very similar to the construction in \cite{OurPaper}, albeit in a Baire category setting, and with modified parameters to ensure a Fourier dimension bound rather than just a Hausdorff dimension bound.

\begin{lemma} \label{LemmaVIVIJCIJSIJ}
    Let $Z \subset \TT^{dn}$ be a compact set with lower Minkowski dimension at most $\alpha$. Then for quasi-all $(E,\mu) \in \mathcal{X}_\beta$, for any distinct points $x_1, \dots, x_n \in E$, $(x_1, \dots, x_n) \not \in Z$.
\end{lemma}
\begin{proof}
    For any $s > 0$, consider the set
    \[ \mathcal{B}(Z,s) = \left\{ E \subset \TT^d: \begin{array}{c}
            \text{for all $x_1, \dots, x_n \in E$ such that}\\
            \text{$|x_i - x_j| \geq s$ for $i \neq j$, $(x_1, \dots, x_n) \not \in Z$}
        \end{array} \right\}, \]
    and
    \[ B(Z,s) = \{ (E,\mu) \in \mathcal{X}_\beta: E \in \mathcal{B}(Z,s) \}. \]
    Then $B(Z,s)$ is open in $\mathcal{X}_\beta$, and
    \begin{equation}
        \bigcap_{k = 1}^\infty B(Z,1/k)
    \end{equation}
%    We claim this set is open. If $(E_1,\mu_1), (E_2,\mu_2), \dots$ are a sequence of elements of $B(W,s)^c$ converging to some pair $(E,\mu)$, we will show $(E,\mu) \in B(W,s)^c$, from which it follows that $B(W,s)^c$ is closed, and hence $B(W,s)$ is open. For each $k$ we may find $x_1^k,\dots,x_n^k \in E_k$ with $|x_i^k - x_j^k| \geq s$ for $i \neq j$ and with $(x_1^k,\dots,x_n^k) \in W$. Applying the compactness of $\TT^d$ and by passing to a subsequence if necessary, we may assume without loss of generality that there exists $x_1,\dots,x_n \in \TT^d$ such that $\lim_{k \to \infty} x_i^k = x_i$ for all $i \in \{ 1, \dots, n \}$. These elements have the property that $|x_i - x_j| \geq s$ for each $i \neq j$, and the fact that $W$ is closed implies that $(x_1,\dots,x_n) \in W$. Since $(E_k,\mu_k)$ converges to $(E,\mu)$, it follows that $x_1,\dots,x_n \in E$, and thus $(E,\mu) \in B(W,s)^c$.
%    It follows that if $d_\mathbb{H}(E_0,E) \leq \varepsilon$, then for any measure $\mu$ supported on $E$, $(E,\mu) \in B(W,s)$. Thus $B(W,s)$ is an open subset of $\mathcal{X}_\beta$.
    %
    consists of the family of sets $(E,\mu)$ such that for distinct $x_1, \dots, x_n \in E$, $(x_1, \dots, x_n) \not \in Z$. Now for each $k$, the set $\mathcal{B}(Z,1/k)$ is a downward closed family, which means that, after we verify the appropriate hypotheses, we can apply Lemma \ref{LemmaFIOAJFOIWJ} to prove $B(Z,s)$ is dense in $\mathcal{X}_\beta$ for each $s > 0$, which would complete the proof. Thus we must construct a set $S = \{ x_1, \dots, x_N \}$ for $N$ which can be made arbitrarily large, such that $N(S,r) \in \mathcal{B}(Z,s)$ and associate with the set an exponential sum satisfying a square root cancellation bound.

    Let $c = 2 n^{1/2}$. Since $Z$ has lower Minkowski dimension at most $\alpha$, for any $\gamma \in (\alpha,dn]$, we can find arbitrarily small $r \in (0,1)$ such that
    \begin{equation} \label{equationGGSCSAS}
        |N(Z, cr)| \leq r^{dn - \gamma}.
    \end{equation}
    Pick $\lambda \in [0,(dn - \gamma)/(n-1/2))$, and suppose that we can find an integer $M \geq 10$ with
    \begin{equation} \label{equationICCISAXAX122412}
        r^{-\lambda} \leq M \leq r^{-\lambda} + 1.
    \end{equation}
    Let $X_1, \dots, X_M$ be independent and uniformly distributed on $\TT^d$. For each distinct set of indices $k_1, \dots, k_n \in \{ 1, \dots, M \}$, the random vector $X_k = (X_{k_1}, \dots, X_{k_n})$ is uniformly distributed on $\TT^{nd}$, and so \eqref{equationGGSCSAS} and \eqref{equationICCISAXAX122412} imply that
    % \psi_1 = \varepsilon \beta
    % \varepsilon \leq n - 1/2
    % \psi_2 \leq \beta/2
    % \psi_2 = 2 \varepsilon \beta
    \begin{equation} \label{equationGGASDCJWIJSFGGGG}
        \PP(d(X_k,Z) \leq cr) \leq |N(Z,cr)| \leq r^{dn - \gamma} \lesssim M^{- \frac{dn - \gamma}{\lambda}} \leq M^{-(n-1/2)},
    \end{equation}
    If $M_0$ denotes the number of indices $k$ such that $d(X_k,Z) \leq cr$, then by linearity of expectation, since there are at most $M^n$ such indices, we conclude from \eqref{equationGGASDCJWIJSFGGGG} that there is a constant $C > 0$ such that
    \begin{equation} \label{equationDDASGVV}
        \EE(M_0) \leq (C/10) M^{1/2}.
    \end{equation}
    Applying Markov's inequality to \eqref{equationDDASGVV}, we conclude that
    \begin{equation} \label{equationFGGGSC}
        \PP(M_0 \geq C M^{1/2}) \leq 1/10.
    \end{equation}
    %
    %For each cube $I$ with sides parallel to the axis of $\TT^d$, and with sidelength $K^{-1/4d}$, we find
    %
%    \begin{equation} \label{equationDIOJOIDJSOIJ2222312414}
%    \begin{split}
%        \PP \left( \text{there is $i \in \{ 1, \dots, K \}$ such that $X_i \in I$} \right) &= 1 - (1 - K^{-1/4})^K \geq 1 - K^{-3/4}.
%    \end{split}
%    \end{equation}
    %
%    Now $\TT^d$ is covered by a family of at most $K^{1/4}$ such cubes $I$. If $F = \{ X_1, \dots, X_K \}$, then a union bound applying \eqref{equationDIOJOIDJSOIJ2222312414} repeatedly shows
    %
%    \begin{equation} \label{equationOIJIOWJDIOWJ23122142}
%        \PP \left( d_H(F,\TT^d) \leq d^{1/2} K^{-1/4d} \right) \geq 1 - K^{-1/2}.
%    \end{equation}
    %
%    In particular, we conclude from \eqref{equationOIJIOWJDIOWJ23122142} that there is $K_0$ depending only on $d$ and $\varepsilon$ such that if $K \geq K_0$, then
    %
%    \begin{equation} \label{equationOIWJOIAWJDOIWJ}
%         \PP \left( d_H(F,\TT^d) \leq d^{1/2} K^{-1/4d} \right) \leq 1/10.
%    \end{equation}
    %
    Fix some small $\kappa > 0$. Taking a union bound to \eqref{equationFGGGSC} %\eqref{equationOIWJOIAWJDOIWJ}
    and the results of Lemma \ref{LemmaGISCICS1}, we conclude that if $M$ is sufficiently large, there exists $M$ distinct points $x_1, \dots, x_M \in \TT^d$ and a constant $C > 0$ such that the following two statements hold:
    \begin{itemize}
        \item[(1)] Let $I$ be the set of indices $k_n \in \{ 1, \dots, M \}$ with the property that we can find distinct indices $k_1, \dots, k_{n-1} \in \{ 1, \dots, M \}$ such that if $X = (X_{k_1}, \dots, X_{k_n})$, then $d(X,Z) \leq cr$. Then $\#(I) \leq C M^{1/2}$.

        \item[(2)] For $0 < |\xi| \leq M^{1 + \kappa}$,
        \[ \left| \frac{1}{M} \sum_{k = 1}^M e^{2 \pi i \xi \cdot x_k} \right| \leq C M^{-1/2} \log(M)^{1/2}. \]
    \end{itemize}
    Now set $S = \{ x_k : k \not \in I \}$ and let $N = \#(S)$. Then Property (1) implies that
    \begin{equation}
        N \geq M - \#(I) \geq M - C M^{1/2}.
    \end{equation}
    Thus for $M \geq 4C^2$,
    \begin{equation} \label{equatiojnOICOJWIAOJDIOj21i4j214ioJOIJIOJD}
        N \geq (1/2) M \geq (1/2) r^{-\lambda}.
    \end{equation}
    Property (1) and (2) imply that for $0 < |\xi| \leq N^{1 + \kappa}$,
    \begin{equation} \label{equationFIOJIOJOIadawda}
    \begin{split}
        \left| \frac{1}{N} \sum_{x \in S} e^{2 \pi i \xi \cdot x} \right| &\leq \left| \frac{1}{N} \sum_{k = 1}^N e^{2 \pi i \xi \cdot x_k} \right| + \left| \frac{1}{N} \sum_{k \in I} e^{2 \pi i \xi \cdot x_k} \right|\\
        &\leq 2C M^{-1/2} \log(M)^{1/2} + \#(I)/N \\
        &\lesssim N^{-1/2} \log(N)^{1/2} + N^{-1/2}\\
        &\lesssim N^{-1/2} \log(N)^{1/2}.
    \end{split}
    \end{equation}
    As long as we can show that $N(S,r) \in \mathcal{B}(Z,s)$, then \eqref{equatiojnOICOJWIAOJDIOj21i4j214ioJOIJIOJD} and \eqref{equationFIOJIOJOIadawda} allows us to apply Lemma \ref{LemmaFIOAJFOIWJ}, completing the proof that $B(Z,s)$ is dense. To check this, consider $n$ points $y_1,\dots,y_n \in N(S,r)$, with $|y_i - y_j| \geq s$ for any two indices $i \neq j$. Provided that $s > 10r$, we can therefore find distinct indices $k_1, \dots, k_n \not \in I$ such that for each $i \in \{ 1, \dots, n \}$, $|x_{k_i} - y_i| \leq r$, which means if we set $x = (x_{k_1}, \dots, x_{k_n})$ and $y = (y_1, \dots, y_n)$, then
    \begin{equation} \label{equationFISICISCI232222452}
        |x - y| \leq cr/2.
    \end{equation}
    Since $k_n \not \in I$, $d(x,Z) \geq cr$, which combined with \eqref{equationFISICISCI232222452} implies
    \begin{equation} \label{equationSICSICI}
        d(y,Z) \geq d(x,Z) - |x - y| \geq cr/2.
    \end{equation}
    Thus in particular, we conclude $y \not \in Z$. But this means we have proved precisely that $N(S,r) \in \mathcal{B}(Z,s)$. Thus Lemma \ref{LemmaFIOAJFOIWJ} implies that $B(Z,s)$ is dense in $\mathcal{X}_\beta$ for each $s > 0$, completing the proof.
\end{proof}

The Baire category theorem, applied to the result of Lemma \ref{LemmaVIVIJCIJSIJ}, shows that a pattern avoiding set exists in $\mathcal{X}_\beta$, completing the proof of Theorem \ref{maintheorem}. Before we move onto the proof of Theorem \ref{theoremJOICVIOJVI122}, let us discuss the main obstacle which prevents us from finding Salem sets with dimension
\begin{equation} \label{equationOIOIVJIOVJ2131245151csacs}
    \frac{dn - \alpha}{n - 1},
\end{equation}
avoiding the pattern $Z$, instead only obtaining Salem sets with dimension at most
\begin{equation} \label{equationCIONCIOJIOAJ12312}
    \frac{dn - \alpha}{n-1/2},
\end{equation}
We begin by noting that the quantity \eqref{equationOIOIVJIOVJ2131245151csacs} is the maximum dimension one can obtain using the randomized selection method used in Lemma \ref{LemmaVIVIJCIJSIJ} to choose the set $S$, since if one alters the parameters used in the proof so that $M \gtrsim r^{-\frac{dn-\alpha}{n-1}}$, then the expectation bounds used in the proof above to control $M_0$ cannot even guarantee that $S$ is non-empty with positive probability. More precisely, for general parameters $M$ and $r$, one can guarantee with high probability that $\#(I) \lesssim M^n r^{dn - \alpha}$. For $M \gg r^{-\frac{dn-\alpha}{n-1}}$, one also expects to have $\#(I) \gg M$, and in this situation we will have $S = \emptyset$, so the construction above does not work at all. In this proof however, we were forced to choose $M$ much smaller than $r^{-\frac{dn-\alpha}{n-1}}$, i.e. we chose $M \approx r^{-\frac{dn - \alpha}{n - 1/2}}$ so that we could guarantee that $\#(I) \lesssim N^{1/2}$. The importance of this is that the trivial bound
\begin{equation}
    \left| \sum_{k \in I} e^{2 \pi i (\xi \cdot X_k)} \right| \leq \#(I)
\end{equation}
obtained by the triangle inequality was then enough to obtain the square root cancellation bound in equation \eqref{equationFIOJIOJOIadawda}. On the other hand, if we were able to show that the set $I$ itself satisfied a square root cancellation bound of the form
\begin{equation} \label{equationDOIJCVOIVJOI213123}
    \left| \sum_{k \in I} e^{2 \pi i (\xi \cdot X_k)} \right| \lesssim \#(I)^{1/2},
\end{equation}
then there would be no barrier to choosing $M \approx r^{-\frac{dn-\alpha}{n-1}}$, which would allow us to prove the existence of a pattern avoiding set with Fourier dimension matching the quantity in equation \eqref{equationCIONCIOJIOAJ12312}, matching that of the Hausdorff dimension bound obtained in \cite{OurPaper}. Under stronger assumptions on the pattern we are trying to avoid, which form the hypotheses of Theorem \ref{theoremJOICVIOJVI122}, we are able to justify that some kind of square root cancellation, like that of \eqref{equationDOIJCVOIVJOI213123} takes place, though with an additional term that we are only able to bound appropriately for $n > 2$ using an inclusion-exclusion argument combined with some oscillatory integrals if we set $N \approx r^{-\frac{dn-\alpha}{n-3/4}}$. Under the hypothesis of Theorem \ref{thirdTheorem}, we are able to make this additional term vanish completely, which will enable us to set $N \approx r^{-\frac{dn - \alpha}{n - 1}}$, thus obtaining sets with Fourier dimension matching the quantity in equation \eqref{equationCIONCIOJIOAJ12312}, and completely recovering the dimension bound of \cite{OurPaper} in the setting of Salem sets.

\section{Concentration Bounds for Smooth Surfaces}

In this section we prove Theorem \ref{theoremJOICVIOJVI122} using some more robust probability concentration calculations, which allow us to justify the kinds of square root cancellation alluded to at the end of the last section. We set
\[ \beta \leq \begin{cases} d &: n = 2 \\ d/(n - 3/4) &: n \geq 3 \end{cases}. \]
For such $\beta$, we now prove that elements of the space $\mathcal{X}_\beta$ will generically avoid patterns given by an equation $x_n = f(x_1,\dots,x_{n-1})$, where $f$ satisfies the hypotheses of Theorem \ref{theoremJOICVIOJVI122}.

\begin{lemma} \label{lemmaOIOICJOIJOISJOIJS}
    Suppose $f: V \to \TT^d$ satisfy the hypothesis of Theorem \ref{theoremJOICVIOJVI122}. Then for quasi-all $(E,\mu) \in \mathcal{X}_\beta$, and for any distinct points $x_1,\dots,x_n \in E$, $x_n \neq f(x_1,\dots,x_{n-1})$.
\end{lemma}
\begin{proof}
    Given any family of disjoint, closed cubes $R_1,\dots,R_n \subset \TT^d$ such that $(R_1 \times \dots \times R_n) \cap V$ is a closed set, we let
    \[ \mathcal{H}(R_1,\dots,R_n) = \{ E \subset \TT^d: \text{for all $x_i \in R_i \cap E$, $x_n \neq f(x_1,\dots,x_{n-1})$} \}, \]
    and let
    \[ H(R_1,\dots,R_n) = \{ (E,\mu) \in \mathcal{X}_\beta: E \in \mathcal{H}(R_1,\dots,R_n) \}. \]
    Then $H(R_1,\dots,R_n)$ is an open subset of $\mathcal{X}_\beta$. For the purpose of a Baire category argument, this proof will follow by showing $H(R_1,\dots,R_n)$ is dense in $\mathcal{X}_\beta$ for any family of disjoint cubes $\{ R_1,\dots, R_n \}$, each having common sidelength $s$ for some $s > 0$, such that if $Q_i = 2R_i$ for each $i$, then $Q_2 \times \dots \times Q_n \subset V$, and $d(R_i,R_j) \geq 10s$ for each $i \neq j$. Since $\mathcal{H}(R_1,\dots,R_n)$ is a downward closed family of sets, we will prove this result by applying Lemma \ref{LemmaFIOAJFOIWJ}. Thus for a suitable choice of $r > 0$, we must construct a large discrete set $S$ such that $N(S,r) \in \mathcal{H}(R_1,\dots,R_n)$ and whose exponential sums exhibit square root cancellation.

    Since $f$ is smooth, we can fix a constant $L \geq 0$ such that for any $x,y \in Q_1 \times \dots \times Q_{n-1}$,
    \begin{equation} \label{eqiatiojawoij2134141235235231}
        |f(x) - f(y)| \leq L|x - y|.
    \end{equation}
    Fix a family of non-negative bump functions $\psi_0,\psi_1,\dots,\psi_n \in C^\infty(\TT^d)$, such that for $i \in \{ 1,\dots,n \}$, $\psi_i(x) = 1$ for $x \in 1.5 \cdot Q_i$, $\psi_i(x) = 0$ for $x \not \in R_i$, and $\psi_0(x) + \dots + \psi_n(x) = 1$ for $x \in \TT^d$. For $i \in \{ 0, \dots, n \}$, let $A_i = \int \psi_i(x)\; dx$ denote the total mass of $\psi_i$. Now fix a large integer $M > 0$, and consider a family of independent random variables
    \[ \{ X_i(k) : 0 \leq i \leq n, 1 \leq k \leq M \}, \]
    where the random variable $X_i(k)$ is chosen with respect to the probability density function $A_i^{-1} \psi_i$. Fix $\lambda \in [0,\beta)$, and set $r = M^{-1/\lambda}$, i.e. so that $M = r^\lambda$. Let $c = 2 n^{1/2} (L+1)$, and define $I$ to be the set of all indices $k_n \in \{ 1, \dots, M \}$ such that there are indices $k_1,\dots,k_{n-1} \in \{ 1,\dots,M \}$ with the property that
    % r <= log(K)^{1/d} K^{(1-n)/d}
    % r^d = log(K) K^{\varepsilon_1 - (n - 1)}
    % Then we get a deviation of K^{1/2} log(K) instead which is prefectly fine!!!!!
    % On the other hand, we get that the expected value has very good decay bounds.
    \begin{equation} \label{equationDDIVJIVCJIJSDDKODKI23125}
        |X_n(k_n) - f(X_1(k_1),\dots,X_{n-1}(k_{n-1}))| \leq cr.
    \end{equation}
    Now \eqref{equationDDIVJIVCJIJSDDKODKI23125} implies that if $k_n \not \in I$, then for any $k_1,\dots,k_{n-1} \in \{ 1, \dots, M \}$,
    \begin{equation}
        |X_n(k_n) - f(X_1(k_1),\dots,X_{n-1}(k_{n-1}))| > cr.
    \end{equation}
    Thus if we set
    \[ S = \{ X_i(k) : 0 \leq i \leq n-1, 1 \leq k \leq M \} \cup \{ X_n(k) : k \not \in I \} \]
    Then we claim that $N(S,r) \in \mathcal{H}(R_1,\dots,R_n)$ for suitably small $r$; to see this, suppose there were distinct $y_1,\dots,y_n \in N(S,r)$ such that $y_1 \in R_1, \dots, y_n \in R_n$, and $y_n = f(y_1,\dots,y_{n-1})$. We may pick $x_1,\dots,x_n \in S$ such that $|x_i - y_i| \leq r$ for each $i$. Since $d(R_i,R_j) = 10s$ for $i \neq j$, if $r \leq s$, then it cannot be true that $x_i = X_j(k)$ for some $j \in \{ 1, \dots, n \}$ and $k \in \{ 1, \dots, M \}$. Since $\psi_i(x) = 1$ on $1.5 R_i$, if $r < 0.5 s$, we have $d(\text{supp}(\psi_0), R_i) \geq 0.5 s$ and so it also cannot be true that $x_i = X_0(k)$ for some $k \in \{ 1, \dots, M \}$. Thus there must be $k_i \in \{ 1, \dots, M \}$ such that $x_i = X_i(k_i)$. But by assumption $k_n \not \in I$, so we have
    \begin{equation} \label{equationODIJDOJWADOIJAWOIJ12312}
        |X_n(k_n) - f(X_1(k_1),\dots,X_{n-1}(k_{n-1}))| > cr.
    \end{equation}
    Thus \eqref{eqiatiojawoij2134141235235231} and \eqref{equationODIJDOJWADOIJAWOIJ12312} imply that
    \begin{equation}
        0 = |y_n - f(y_1,\dots,y_n)| \geq cr > 0,
    \end{equation}
    which gives a contradiction, proving that $N(S,r) \in \mathcal{H}(R_1,\dots,R_n)$. The remainder of the proof focuses on bounding exponential sums associated with $S$, so that we may apply Lemma \ref{LemmaFIOAJFOIWJ} and thus prove the conclusion of the theorem.

    Consider the random exponential sums
    \[ F(\xi) = \sum_{i \in \{ 0, \dots, n-1 \}} \sum_{k = 1}^M A_i e^{2 \pi i \xi \cdot X_i(k)} + \sum_{k \not \in I} A_n e^{2 \pi i \xi \cdot X_n(k)}. \]
    Controlling $|F(\xi)|$ with high probability will justify an application of Lemma \ref{LemmaFIOAJFOIWJ}. To analyze $F$, introduce the auxiliary exponential sums
    \[ G(\xi) = \sum_{i = 0}^n \sum_{k = 1}^M A_i e^{2 \pi i \xi \cdot X_i(k)} \]
    and
    \[ H(\xi) = \sum_{k \in I} A_n e^{2 \pi i \xi \cdot X_n(k)}. \]
    Then $F(\xi) = G(\xi) - H(\xi)$.
    Obtaining a bound on $G(\xi)$ is simple since it is a sum of $(n+1) \cdot M$ independent random variables. For non-zero $\xi \in \ZZ^d$,
    \begin{equation}
    \begin{split}
        \EE(G(\xi)) &= \sum_{i = 0}^n M A_i \int (\psi_i(x) / A_i) e^{2 \pi i \xi \cdot x}\; dx\\
        &= M \sum_{i = 0}^n \int \psi_i(x) e^{2 \pi i \xi \cdot x}\; dx\\
        &= M \int_{\TT^d} e^{2 \pi i \xi \cdot x}\; dx = 0.
    \end{split}
    \end{equation}
    Applying Lemma \ref{LemmaGISCICS1}, we conclude that for any fixed $\kappa > 0$, there is $C > 0$ such that
    \begin{equation} \label{equationCOIACOIAJCPPPPP}
        \PP \left( \sup_{|\xi| \leq N^{1 + \kappa}} |G(\xi)| \geq C M^{1/2} \log(M)^{1/2} \right) \leq 1/10.
    \end{equation}
    Analyzing $H(\xi)$ requires a more subtle concentration bound, which we delegate to a series of lemmas following this proof:
    \begin{itemize}
        \item In Lemma \ref{lemma24901401921209}, we will employ some concentration bounds to show that
        \begin{equation} \label{equationCICCIJCISJCOXOXP12}
            \PP \left( \sup_{|\xi| \leq N^{1 + \kappa}} | H(\xi) - \EE(H(\xi)) | \geq C M^{1/2} \log(M)^{1/2} \right) \leq 1/10.
        \end{equation}
    
        \item In Lemma \ref{lemmaOIJIOCJSOIJSIOJ123} we will show that for any $\delta > 0$, there exists $r_1 > 0$ such that for $r \leq r_1$ and any nonzero $\xi \in \ZZ^d$,
        \begin{equation} \label{equationCIJCIJIJXSO}
            |\EE(H(\xi))| \leq \delta M |\xi|^{-\beta/2} + O(M^{1/2}).
        \end{equation}
    \end{itemize}
    Analogous to equation \eqref{equationGGASDCJWIJSFGGGG} in Lemma \ref{LemmaVIVIJCIJSIJ}, for any indices $k_1,\dots,k_n \in \{ 1, \dots, M \}$, we have
    \begin{equation} \label{equationCCSCSOIJIOJ12312421412s}
        \PP \Big( |X_n(k_n) - f(X_1(k_1),\dots,X_{n-1}(k_{n-1}))| \leq 2 n^{1/2} \cdot (L+1) \cdot r \Big) \lesssim_{n,L} r^d \lesssim M^{-d/\lambda}.
    \end{equation}
    Thus if $M_0$ denotes the number of tuples of indices $(k_1,\dots,k_n)$ such that \eqref{equationDDIVJIVCJIJSDDKODKI23125} holds, then \eqref{equationCCSCSOIJIOJ12312421412s} implies that
    \begin{equation} \label{equatiocaiodjioj1231312412}
        \EE(M_0) \lesssim M^{n-d/\lambda}
    \end{equation}
    Applying Markov's inequality to \eqref{equatiocaiodjioj1231312412}, we conclude that there exists a constant $C > 0$ such that
    \begin{equation} \label{equationCOIJCOISAJCOIJ1243124124}
        \PP(M_0 \geq C M^{d/\lambda - n}) \leq 1/10.
    \end{equation}
    Taking a union bound to \eqref{equationCOIACOIAJCPPPPP}, \eqref{equationCICCIJCISJCOXOXP12}, and \eqref{equationCOIJCOISAJCOIJ1243124124}, and then applying \eqref{equationCIJCIJIJXSO}, we conclude that there exists $C > 0$ and a particular instantiation of the random variables $\{ X_i(k) \}$ such that for any $0 < |\xi| \leq M^{1 + \kappa}$,
    \begin{equation} \label{equatioACOINAWCOWANCOIWAN}
        |G(\xi)| \leq CM^{1/2} \log(M)^{1/2},
    \end{equation}
    and
    \begin{equation} \label{equatioNCOIJCOIJOIJO12312421412}
        |H(\xi)| \leq C M^{1/2} \log(M)^{1/2} + \delta M |\xi|^{-\beta/2}.
    \end{equation}
    And
    \begin{equation} \label{equationCOIACOIAWJOIJo21412412}
        \#(I) \leq C M^{d/\lambda - n}.
    \end{equation}
    Since $\lambda < \beta_0 < d/(n-1)$, the inequality $d/\lambda - n < 1$ holds. Thus \eqref{equationCOIACOIAWJOIJo21412412} implies that for sufficiently large $M > 0$, if $N = \#(S)$, then
    \begin{equation} \label{equationCOIJCOIJVIOJOIJIOJ4234141}
        N \geq M - C M^{d/\lambda - n} \geq (1/2) M \geq (1/2) r^{-\lambda}.
    \end{equation}
    Putting \eqref{equationCIJCIJIJXSO}, \eqref{equatioACOINAWCOWANCOIWAN}, \eqref{equatioNCOIJCOIJOIJO12312421412}, and the fact that $F(\xi) = G(\xi) + H(\xi)$ together, if we set $\tilde{a}(X_i(k)) = A_i$ for each $i$ and $k$, then
    \begin{equation} \label{equationDOIJAOIDJWOIJ}
        \left| \frac{1}{N} \sum_{x \in S} \tilde{a}(x) e^{2 \pi i \xi \cdot x} \right| \lesssim C N^{1/2} \log(N)^{1/2} + \delta |\xi|^{-\beta/2}.
    \end{equation}
    Since $\sum_{x \in S} \tilde{a}(x) \geq N$, if we set
    \[ a(x) = N \cdot \frac{\tilde{a}(x)}{\sum_{x \in S} \tilde{a}(x)}, \]
    then \eqref{equationDOIJAOIDJWOIJ}, \eqref{equationCOIJCOIJVIOJOIJIOJ4234141} and the fact that $N(S,r) \in \mathcal{H}(R_1,\dots,R_n)$ imply that the sum
    \[ \frac{1}{N} \sum_{x \in S} a(x) e^{2 \pi i \xi \cdot x} \]
    satisfies the assumptions of Lemma \ref{LemmaFIOAJFOIWJ} for arbitrarily large $N$. We therefore conclude by that Lemma that $H(R_1,\dots,R_n)$ is dense in $\mathcal{X}_\beta$.
\end{proof}

Our proof of Theorem \ref{theoremJOICVIOJVI122} will be complete once we prove \eqref{equationCICCIJCISJCOXOXP12} and \eqref{equationCIJCIJIJXSO}, i.e. once we prove Lemmas \ref{lemma24901401921209} and \ref{lemmaOIJIOCJSOIJSIOJ123}.

\begin{lemma} \label{lemma24901401921209}
    Let $H(\xi)$ be the random exponential sum described in Lemma \ref{lemmaOIOICJOIJOISJOIJS}. Then
    \[ \PP \left( \sup_{|\xi| \leq M^{1 + \kappa}} | H(\xi) - \EE(H(\xi)) | \geq C M^{1/2} \log(M)^{1/2} \right) \leq 1/10. \]
    for some universal constant $C > 0$.
\end{lemma}

%Consider a family of points $\{ X_i(k): 1 \leq i \leq n, 1 \leq k \leq K \}$ in $[0,1]$, let $f: [0,1]^{n-1} \to [0,1]$ be a function, and let $S$ be the set of all $k_3$ such that there exists $X_1(k_1)$ and $X_2(k_2)$ such that $|X_3(k_3) - f(X_1(k_1),X_2(k_2))| \leq r$. Altering one of the variables $X_i(k)$ for $1 \leq i \leq n-1$ can completely alter the set $S$.

Before we prove this Lemma, let us describe the idea behind it's proof. The result is a \emph{concentration bound} for the random quantity $H(\xi)$, a standard topic in the theory of high dimensional probability. The basic heuristic of this topic is that for an arbitrary function $F(X_1,\dots,X_N)$ of many independent random inputs, we have $F(X_1,\dots,X_N) \approx \EE(F(X_1,\dots,X_N))$ with high probability, provided that each of the random inputs $\{ X_i \}$ has a small influence on the overall output of $f$. McDiarmid's and Hoeffding's inequalities, described in the notation section of this paper, are two classic results in this theory. A major difference between the two inequalities is that McDiarmid's inequality can be applied to \emph{nonlinear functions} $F$, whereas Hoeffding's inequality can only be applied when $F$ is linear.

Since $H(\xi)$ is a nonlinear function of the independent random quantities $\{ X_i(k) \}$, McDiarmid's inequality presents itself as a useful concentration bound. However, a naive application of McDiarmid's inequality fails here, because changing a single random variable $X_i(k)$ for $1 \leq i \leq n-1$ while fixing all other random variables can change the indices contained in the set $I$ by as much as $O(M)$, and thus change $H(\xi)$ by as much as $O(M)$ as a result (see Figure 1 for an example of this phenomenon). McDiarmid's inequality then gives that $|H(\xi) - \EE(H(\xi))| \lesssim M$ with high probability, which is not tight enough to obtain square root cancellation like what we obtained in \eqref{equationCOIACOIAJCPPPPP}. On the other hand, it seems that a single variable $X_i(k)$ only changes $H(\xi)$ by $O(M)$ when the other random variables $\{ X_1(k) \}$ are configured in a very particular way, which is unlikely to happen. Thus we should expect that adjusting a single random variable $X_i(k)$ does not influence the value of $H(\xi)$ much when \emph{averaged} over the possible choices of $\{ X_n(k) \}$. This leads us to first average over this first set of random variables, and then apply McDiarmid's inequality, which yields the correct concentration result.

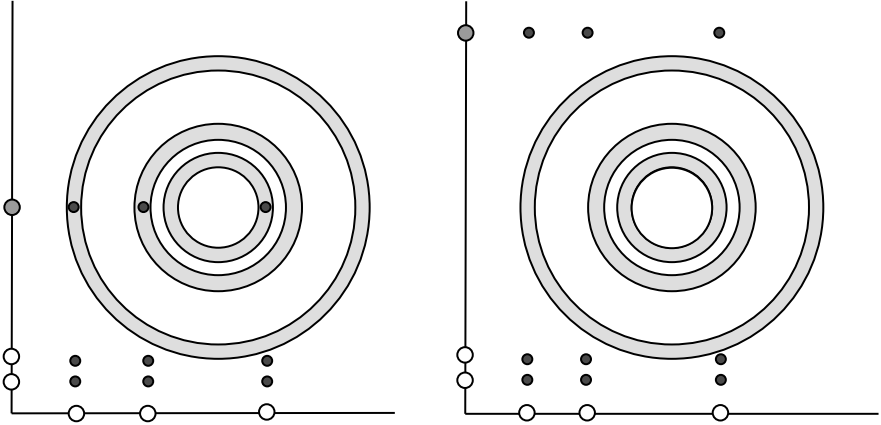
\begin{figure}

\begin{center}

\tikzset{every picture/.style={line width=0.75pt}} %set default line width to 0.75pt        

\begin{tikzpicture}[x=0.75pt,y=0.75pt,yscale=0.8,xscale=0.8]
%uncomment if require: \path (0,299); %set diagram left start at 0, and has height of 299

%Straight Lines [id:da8194240959270346] 
\draw    (314,19.7) -- (314.6,279.97) ;
%Straight Lines [id:da12205278341739523] 
\draw    (314,19.7) -- (574.55,19.7) ;
%Straight Lines [id:da8187123168571464] 
\draw    (28,20) -- (269.6,20.27) ;
%Straight Lines [id:da796553864028118] 
\draw    (28,20) -- (28.6,280.27) ;
%Shape: Circle [id:dp7829643160935992] 
\draw  [fill={rgb, 255:red, 222; green, 222; blue, 222 }  ,fill opacity=1 ] (348.81,149.87) .. controls (348.81,97.14) and (391.55,54.4) .. (444.27,54.4) .. controls (497,54.4) and (539.74,97.14) .. (539.74,149.87) .. controls (539.74,202.59) and (497,245.34) .. (444.27,245.34) .. controls (391.55,245.34) and (348.81,202.59) .. (348.81,149.87) -- cycle ;
%Shape: Circle [id:dp3106807534625411] 
\draw  [fill={rgb, 255:red, 255; green, 255; blue, 255 }  ,fill opacity=1 ] (357.85,149.87) .. controls (357.85,102.14) and (396.55,63.45) .. (444.27,63.45) .. controls (492,63.45) and (530.7,102.14) .. (530.7,149.87) .. controls (530.7,197.6) and (492,236.29) .. (444.27,236.29) .. controls (396.55,236.29) and (357.85,197.6) .. (357.85,149.87) -- cycle ;
%Shape: Circle [id:dp30551387107036876] 
\draw  [fill={rgb, 255:red, 222; green, 222; blue, 222 }  ,fill opacity=1 ] (391.46,149.87) .. controls (391.46,120.7) and (415.1,97.05) .. (444.27,97.05) .. controls (473.45,97.05) and (497.09,120.7) .. (497.09,149.87) .. controls (497.09,179.04) and (473.45,202.68) .. (444.27,202.68) .. controls (415.1,202.68) and (391.46,179.04) .. (391.46,149.87) -- cycle ;
%Shape: Circle [id:dp0407471735325019] 
\draw  [fill={rgb, 255:red, 255; green, 255; blue, 255 }  ,fill opacity=1 ] (401.58,149.87) .. controls (401.58,126.29) and (420.7,107.17) .. (444.27,107.17) .. controls (467.85,107.17) and (486.97,126.29) .. (486.97,149.87) .. controls (486.97,173.45) and (467.85,192.56) .. (444.27,192.56) .. controls (420.7,192.56) and (401.58,173.45) .. (401.58,149.87) -- cycle ;
%Shape: Circle [id:dp7794418165975352] 
\draw  [fill={rgb, 255:red, 222; green, 222; blue, 222 }  ,fill opacity=1 ] (409.8,149.87) .. controls (409.8,130.83) and (425.24,115.39) .. (444.27,115.39) .. controls (463.31,115.39) and (478.75,130.83) .. (478.75,149.87) .. controls (478.75,168.91) and (463.31,184.34) .. (444.27,184.34) .. controls (425.24,184.34) and (409.8,168.91) .. (409.8,149.87) -- cycle ;
%Shape: Circle [id:dp8077673979250265] 
\draw  [fill={rgb, 255:red, 255; green, 255; blue, 255 }  ,fill opacity=1 ] (418.9,149.87) .. controls (418.9,135.85) and (430.26,124.49) .. (444.27,124.49) .. controls (458.29,124.49) and (469.65,135.85) .. (469.65,149.87) .. controls (469.65,163.88) and (458.29,175.24) .. (444.27,175.24) .. controls (430.26,175.24) and (418.9,163.88) .. (418.9,149.87) -- cycle ;
%Shape: Circle [id:dp6998803528397256] 
\draw  [fill={rgb, 255:red, 222; green, 222; blue, 222 }  ,fill opacity=1 ] (62.81,149.87) .. controls (62.81,97.14) and (105.55,54.4) .. (158.27,54.4) .. controls (211,54.4) and (253.74,97.14) .. (253.74,149.87) .. controls (253.74,202.59) and (211,245.34) .. (158.27,245.34) .. controls (105.55,245.34) and (62.81,202.59) .. (62.81,149.87) -- cycle ;
%Shape: Circle [id:dp4255421774727244] 
\draw  [fill={rgb, 255:red, 255; green, 255; blue, 255 }  ,fill opacity=1 ] (71.85,149.87) .. controls (71.85,102.14) and (110.55,63.45) .. (158.27,63.45) .. controls (206,63.45) and (244.7,102.14) .. (244.7,149.87) .. controls (244.7,197.6) and (206,236.29) .. (158.27,236.29) .. controls (110.55,236.29) and (71.85,197.6) .. (71.85,149.87) -- cycle ;
%Shape: Circle [id:dp06617638524915026] 
\draw  [fill={rgb, 255:red, 222; green, 222; blue, 222 }  ,fill opacity=1 ] (105.46,149.87) .. controls (105.46,120.7) and (129.1,97.05) .. (158.27,97.05) .. controls (187.45,97.05) and (211.09,120.7) .. (211.09,149.87) .. controls (211.09,179.04) and (187.45,202.68) .. (158.27,202.68) .. controls (129.1,202.68) and (105.46,179.04) .. (105.46,149.87) -- cycle ;
%Shape: Circle [id:dp9352030665686459] 
\draw  [fill={rgb, 255:red, 255; green, 255; blue, 255 }  ,fill opacity=1 ] (115.58,149.87) .. controls (115.58,126.29) and (134.7,107.17) .. (158.27,107.17) .. controls (181.85,107.17) and (200.97,126.29) .. (200.97,149.87) .. controls (200.97,173.45) and (181.85,192.56) .. (158.27,192.56) .. controls (134.7,192.56) and (115.58,173.45) .. (115.58,149.87) -- cycle ;
%Shape: Circle [id:dp17989716938459888] 
\draw  [fill={rgb, 255:red, 222; green, 222; blue, 222 }  ,fill opacity=1 ] (123.8,149.87) .. controls (123.8,130.83) and (139.24,115.39) .. (158.27,115.39) .. controls (177.31,115.39) and (192.75,130.83) .. (192.75,149.87) .. controls (192.75,168.91) and (177.31,184.34) .. (158.27,184.34) .. controls (139.24,184.34) and (123.8,168.91) .. (123.8,149.87) -- cycle ;
%Shape: Circle [id:dp6851632806180186] 
\draw  [fill={rgb, 255:red, 255; green, 255; blue, 255 }  ,fill opacity=1 ] (132.9,149.87) .. controls (132.9,135.85) and (144.26,124.49) .. (158.27,124.49) .. controls (172.29,124.49) and (183.65,135.85) .. (183.65,149.87) .. controls (183.65,163.88) and (172.29,175.24) .. (158.27,175.24) .. controls (144.26,175.24) and (132.9,163.88) .. (132.9,149.87) -- cycle ;
%Shape: Circle [id:dp40350144630782714] 
\draw  [fill={rgb, 255:red, 255; green, 255; blue, 255 }  ,fill opacity=1 ] (64,19.88) .. controls (64,17.18) and (66.18,15) .. (68.88,15) .. controls (71.57,15) and (73.75,17.18) .. (73.75,19.88) .. controls (73.75,22.57) and (71.57,24.75) .. (68.88,24.75) .. controls (66.18,24.75) and (64,22.57) .. (64,19.88) -- cycle ;
%Shape: Circle [id:dp08899969161544696] 
\draw  [fill={rgb, 255:red, 255; green, 255; blue, 255 }  ,fill opacity=1 ] (418.9,149.57) .. controls (418.9,135.55) and (430.26,124.19) .. (444.27,124.19) .. controls (458.29,124.19) and (469.65,135.55) .. (469.65,149.57) .. controls (469.65,163.58) and (458.29,174.94) .. (444.27,174.94) .. controls (430.26,174.94) and (418.9,163.58) .. (418.9,149.57) -- cycle ;
%Shape: Circle [id:dp5743651141366687] 
\draw  [fill={rgb, 255:red, 74; green, 74; blue, 74 }  ,fill opacity=1 ] (64,150.14) .. controls (64,148.41) and (65.41,147) .. (67.14,147) .. controls (68.88,147) and (70.28,148.41) .. (70.28,150.14) .. controls (70.28,151.88) and (68.88,153.28) .. (67.14,153.28) .. controls (65.41,153.28) and (64,151.88) .. (64,150.14) -- cycle ;
%Shape: Circle [id:dp13210263164899094] 
\draw  [fill={rgb, 255:red, 74; green, 74; blue, 74 }  ,fill opacity=1 ] (108,150.14) .. controls (108,148.41) and (109.41,147) .. (111.14,147) .. controls (112.88,147) and (114.28,148.41) .. (114.28,150.14) .. controls (114.28,151.88) and (112.88,153.28) .. (111.14,153.28) .. controls (109.41,153.28) and (108,151.88) .. (108,150.14) -- cycle ;
%Shape: Circle [id:dp74213242291322] 
\draw  [fill={rgb, 255:red, 74; green, 74; blue, 74 }  ,fill opacity=1 ] (185,150.14) .. controls (185,148.41) and (186.41,147) .. (188.14,147) .. controls (189.88,147) and (191.28,148.41) .. (191.28,150.14) .. controls (191.28,151.88) and (189.88,153.28) .. (188.14,153.28) .. controls (186.41,153.28) and (185,151.88) .. (185,150.14) -- cycle ;
%Shape: Circle [id:dp6666612504590449] 
\draw  [fill={rgb, 255:red, 74; green, 74; blue, 74 }  ,fill opacity=1 ] (65,53.14) .. controls (65,51.41) and (66.41,50) .. (68.14,50) .. controls (69.88,50) and (71.28,51.41) .. (71.28,53.14) .. controls (71.28,54.88) and (69.88,56.28) .. (68.14,56.28) .. controls (66.41,56.28) and (65,54.88) .. (65,53.14) -- cycle ;
%Shape: Circle [id:dp25417231201667445] 
\draw  [fill={rgb, 255:red, 74; green, 74; blue, 74 }  ,fill opacity=1 ] (65,40.14) .. controls (65,38.41) and (66.41,37) .. (68.14,37) .. controls (69.88,37) and (71.28,38.41) .. (71.28,40.14) .. controls (71.28,41.88) and (69.88,43.28) .. (68.14,43.28) .. controls (66.41,43.28) and (65,41.88) .. (65,40.14) -- cycle ;
%Shape: Circle [id:dp20816207183269508] 
\draw  [fill={rgb, 255:red, 74; green, 74; blue, 74 }  ,fill opacity=1 ] (111,40.14) .. controls (111,38.41) and (112.41,37) .. (114.14,37) .. controls (115.88,37) and (117.28,38.41) .. (117.28,40.14) .. controls (117.28,41.88) and (115.88,43.28) .. (114.14,43.28) .. controls (112.41,43.28) and (111,41.88) .. (111,40.14) -- cycle ;
%Shape: Circle [id:dp8890493225780333] 
\draw  [fill={rgb, 255:red, 74; green, 74; blue, 74 }  ,fill opacity=1 ] (111,53.14) .. controls (111,51.41) and (112.41,50) .. (114.14,50) .. controls (115.88,50) and (117.28,51.41) .. (117.28,53.14) .. controls (117.28,54.88) and (115.88,56.28) .. (114.14,56.28) .. controls (112.41,56.28) and (111,54.88) .. (111,53.14) -- cycle ;
%Shape: Circle [id:dp7412941805834966] 
\draw  [fill={rgb, 255:red, 74; green, 74; blue, 74 }  ,fill opacity=1 ] (186,53.14) .. controls (186,51.41) and (187.41,50) .. (189.14,50) .. controls (190.88,50) and (192.28,51.41) .. (192.28,53.14) .. controls (192.28,54.88) and (190.88,56.28) .. (189.14,56.28) .. controls (187.41,56.28) and (186,54.88) .. (186,53.14) -- cycle ;
%Shape: Circle [id:dp48065283288266936] 
\draw  [fill={rgb, 255:red, 74; green, 74; blue, 74 }  ,fill opacity=1 ] (186,40.14) .. controls (186,38.41) and (187.41,37) .. (189.14,37) .. controls (190.88,37) and (192.28,38.41) .. (192.28,40.14) .. controls (192.28,41.88) and (190.88,43.28) .. (189.14,43.28) .. controls (187.41,43.28) and (186,41.88) .. (186,40.14) -- cycle ;
%Shape: Circle [id:dp34091589540709266] 
\draw  [fill={rgb, 255:red, 74; green, 74; blue, 74 }  ,fill opacity=1 ] (350,41.14) .. controls (350,39.41) and (351.41,38) .. (353.14,38) .. controls (354.88,38) and (356.28,39.41) .. (356.28,41.14) .. controls (356.28,42.88) and (354.88,44.28) .. (353.14,44.28) .. controls (351.41,44.28) and (350,42.88) .. (350,41.14) -- cycle ;
%Shape: Circle [id:dp8489646709173403] 
\draw  [fill={rgb, 255:red, 74; green, 74; blue, 74 }  ,fill opacity=1 ] (387,41.14) .. controls (387,39.41) and (388.41,38) .. (390.14,38) .. controls (391.88,38) and (393.28,39.41) .. (393.28,41.14) .. controls (393.28,42.88) and (391.88,44.28) .. (390.14,44.28) .. controls (388.41,44.28) and (387,42.88) .. (387,41.14) -- cycle ;
%Shape: Circle [id:dp311004704357545] 
\draw  [fill={rgb, 255:red, 74; green, 74; blue, 74 }  ,fill opacity=1 ] (472,41.14) .. controls (472,39.41) and (473.41,38) .. (475.14,38) .. controls (476.88,38) and (478.28,39.41) .. (478.28,41.14) .. controls (478.28,42.88) and (476.88,44.28) .. (475.14,44.28) .. controls (473.41,44.28) and (472,42.88) .. (472,41.14) -- cycle ;
%Shape: Circle [id:dp31170134224605195] 
\draw  [fill={rgb, 255:red, 74; green, 74; blue, 74 }  ,fill opacity=1 ] (350,54.14) .. controls (350,52.41) and (351.41,51) .. (353.14,51) .. controls (354.88,51) and (356.28,52.41) .. (356.28,54.14) .. controls (356.28,55.88) and (354.88,57.28) .. (353.14,57.28) .. controls (351.41,57.28) and (350,55.88) .. (350,54.14) -- cycle ;
%Shape: Circle [id:dp28532980461694424] 
\draw  [fill={rgb, 255:red, 74; green, 74; blue, 74 }  ,fill opacity=1 ] (387,54.14) .. controls (387,52.41) and (388.41,51) .. (390.14,51) .. controls (391.88,51) and (393.28,52.41) .. (393.28,54.14) .. controls (393.28,55.88) and (391.88,57.28) .. (390.14,57.28) .. controls (388.41,57.28) and (387,55.88) .. (387,54.14) -- cycle ;
%Shape: Circle [id:dp6020013292698594] 
\draw  [fill={rgb, 255:red, 74; green, 74; blue, 74 }  ,fill opacity=1 ] (472,54.14) .. controls (472,52.41) and (473.41,51) .. (475.14,51) .. controls (476.88,51) and (478.28,52.41) .. (478.28,54.14) .. controls (478.28,55.88) and (476.88,57.28) .. (475.14,57.28) .. controls (473.41,57.28) and (472,55.88) .. (472,54.14) -- cycle ;
%Shape: Circle [id:dp023435344947860814] 
\draw  [fill={rgb, 255:red, 74; green, 74; blue, 74 }  ,fill opacity=1 ] (351,260.14) .. controls (351,258.41) and (352.41,257) .. (354.14,257) .. controls (355.88,257) and (357.28,258.41) .. (357.28,260.14) .. controls (357.28,261.88) and (355.88,263.28) .. (354.14,263.28) .. controls (352.41,263.28) and (351,261.88) .. (351,260.14) -- cycle ;
%Shape: Circle [id:dp781542616798407] 
\draw  [fill={rgb, 255:red, 74; green, 74; blue, 74 }  ,fill opacity=1 ] (388,260.14) .. controls (388,258.41) and (389.41,257) .. (391.14,257) .. controls (392.88,257) and (394.28,258.41) .. (394.28,260.14) .. controls (394.28,261.88) and (392.88,263.28) .. (391.14,263.28) .. controls (389.41,263.28) and (388,261.88) .. (388,260.14) -- cycle ;
%Shape: Circle [id:dp7192093882668958] 
\draw  [fill={rgb, 255:red, 74; green, 74; blue, 74 }  ,fill opacity=1 ] (471,260.14) .. controls (471,258.41) and (472.41,257) .. (474.14,257) .. controls (475.88,257) and (477.28,258.41) .. (477.28,260.14) .. controls (477.28,261.88) and (475.88,263.28) .. (474.14,263.28) .. controls (472.41,263.28) and (471,261.88) .. (471,260.14) -- cycle ;
%Shape: Circle [id:dp08487765818129678] 
\draw  [fill={rgb, 255:red, 255; green, 255; blue, 255 }  ,fill opacity=1 ] (109,19.88) .. controls (109,17.18) and (111.18,15) .. (113.88,15) .. controls (116.57,15) and (118.75,17.18) .. (118.75,19.88) .. controls (118.75,22.57) and (116.57,24.75) .. (113.88,24.75) .. controls (111.18,24.75) and (109,22.57) .. (109,19.88) -- cycle ;
%Shape: Circle [id:dp7550770077214879] 
\draw  [fill={rgb, 255:red, 255; green, 255; blue, 255 }  ,fill opacity=1 ] (184,20.88) .. controls (184,18.18) and (186.18,16) .. (188.88,16) .. controls (191.57,16) and (193.75,18.18) .. (193.75,20.88) .. controls (193.75,23.57) and (191.57,25.75) .. (188.88,25.75) .. controls (186.18,25.75) and (184,23.57) .. (184,20.88) -- cycle ;
%Shape: Circle [id:dp011014467225936797] 
\draw  [fill={rgb, 255:red, 255; green, 255; blue, 255 }  ,fill opacity=1 ] (23,39.88) .. controls (23,37.18) and (25.18,35) .. (27.88,35) .. controls (30.57,35) and (32.75,37.18) .. (32.75,39.88) .. controls (32.75,42.57) and (30.57,44.75) .. (27.88,44.75) .. controls (25.18,44.75) and (23,42.57) .. (23,39.88) -- cycle ;
%Shape: Circle [id:dp9016398775074254] 
\draw  [fill={rgb, 255:red, 255; green, 255; blue, 255 }  ,fill opacity=1 ] (23,55.88) .. controls (23,53.18) and (25.18,51) .. (27.88,51) .. controls (30.57,51) and (32.75,53.18) .. (32.75,55.88) .. controls (32.75,58.57) and (30.57,60.75) .. (27.88,60.75) .. controls (25.18,60.75) and (23,58.57) .. (23,55.88) -- cycle ;
%Shape: Circle [id:dp5891247736437845] 
\draw  [fill={rgb, 255:red, 155; green, 155; blue, 155 }  ,fill opacity=1 ] (23.43,150.01) .. controls (23.43,147.32) and (25.61,145.13) .. (28.3,145.13) .. controls (30.99,145.13) and (33.18,147.32) .. (33.18,150.01) .. controls (33.18,152.7) and (30.99,154.88) .. (28.3,154.88) .. controls (25.61,154.88) and (23.43,152.7) .. (23.43,150.01) -- cycle ;
%Shape: Circle [id:dp5966776679620737] 
\draw  [fill={rgb, 255:red, 255; green, 255; blue, 255 }  ,fill opacity=1 ] (348,20.36) .. controls (348,17.67) and (350.18,15.48) .. (352.88,15.48) .. controls (355.57,15.48) and (357.75,17.67) .. (357.75,20.36) .. controls (357.75,23.05) and (355.57,25.23) .. (352.88,25.23) .. controls (350.18,25.23) and (348,23.05) .. (348,20.36) -- cycle ;
%Shape: Circle [id:dp7740110171501559] 
\draw  [fill={rgb, 255:red, 255; green, 255; blue, 255 }  ,fill opacity=1 ] (386,20.36) .. controls (386,17.67) and (388.18,15.48) .. (390.88,15.48) .. controls (393.57,15.48) and (395.75,17.67) .. (395.75,20.36) .. controls (395.75,23.05) and (393.57,25.23) .. (390.88,25.23) .. controls (388.18,25.23) and (386,23.05) .. (386,20.36) -- cycle ;
%Shape: Circle [id:dp6095822005010809] 
\draw  [fill={rgb, 255:red, 255; green, 255; blue, 255 }  ,fill opacity=1 ] (470,20.36) .. controls (470,17.67) and (472.18,15.48) .. (474.88,15.48) .. controls (477.57,15.48) and (479.75,17.67) .. (479.75,20.36) .. controls (479.75,23.05) and (477.57,25.23) .. (474.88,25.23) .. controls (472.18,25.23) and (470,23.05) .. (470,20.36) -- cycle ;
%Shape: Circle [id:dp7682016765751116] 
\draw  [fill={rgb, 255:red, 255; green, 255; blue, 255 }  ,fill opacity=1 ] (309,40.88) .. controls (309,38.18) and (311.18,36) .. (313.88,36) .. controls (316.57,36) and (318.75,38.18) .. (318.75,40.88) .. controls (318.75,43.57) and (316.57,45.75) .. (313.88,45.75) .. controls (311.18,45.75) and (309,43.57) .. (309,40.88) -- cycle ;
%Shape: Circle [id:dp7837323222556574] 
\draw  [fill={rgb, 255:red, 255; green, 255; blue, 255 }  ,fill opacity=1 ] (309,56.88) .. controls (309,54.18) and (311.18,52) .. (313.88,52) .. controls (316.57,52) and (318.75,54.18) .. (318.75,56.88) .. controls (318.75,59.57) and (316.57,61.75) .. (313.88,61.75) .. controls (311.18,61.75) and (309,59.57) .. (309,56.88) -- cycle ;
%Shape: Circle [id:dp030269638935122578] 
\draw  [fill={rgb, 255:red, 155; green, 155; blue, 155 }  ,fill opacity=1 ] (309.43,260.01) .. controls (309.43,257.32) and (311.61,255.13) .. (314.3,255.13) .. controls (316.99,255.13) and (319.18,257.32) .. (319.18,260.01) .. controls (319.18,262.7) and (316.99,264.88) .. (314.3,264.88) .. controls (311.61,264.88) and (309.43,262.7) .. (309.43,260.01) -- cycle ;

\end{tikzpicture}

\caption[LoF entry]{The two diagrams displayed indicate two instances of the set $S$ for the function $f(x_1,x_2) = (x_1 - 1/2)^2 + (x_2 - 1/2)^2$. Here $M = 3$, $n = 3$, the values on the $x$-axis represent the values $X_1(1), X_1(2),$ and $X_1(3)$, the values on the $y$-axis represent the values $X_2(1), X_2(2)$, and $X_2(3)$, the dark points represent the family of all pairs $(X_1(k_1),X_2(k_2))$, and the annuli represent the $O(r)$-neighborhoods of $f^{-1}(X_3(1)), f^{-1}(X_3(2))$, and $f^{-1}(X_3(3))$. In this setup, $S$ consists of all of the values $\{ X_2(k) \}$, as well as all values of $X_1(k)$ such that none of the dark points on the vertical line above $X_1(k)$ intersect any of the annuli. The two diagrams only differ as a result of adjusting a single variable $X_2(k_0)$, indicated by the shaded value on the $y$-axis. For the values represented in the left diagram, $I = \emptyset$, whereas for the values represented in the right diagram, $I$ contains every index, and this completely alters the exponential sums associated with $S$.} %Nonetheless, we see that such a drastic change can only occur if the values of $\{ X_1(k) \}$ are arranged in a highly particular manner., so that these points intersect the annuli in a very particular position. This is why we are able to show that the value of the exponential sums associated with $S$ are not changed drastically when the value is \emph{averaged} over all possible choices of the variables $\{ X_1(k) \}$, which enables us to apply McDiarmid's inequality after this averaging process.}
\label{thefigure}
\end{center}
\end{figure}

\begin{proof}[Proof of Lemma \ref{lemma24901401921209}]
    Consider the random set $\Omega$ of values $x_n \in Q_n$ such that there are $k_1,\dots,k_{n-1} \in \{ 1,\dots,M \}$ with
    \begin{equation}
        |x_n - f(X_1(k_1),\dots,X_{n-1}(k_{n-1}))| \leq cr.
    \end{equation}
    Then
    \begin{equation}
        H(\xi) = A_n \sum_{k = 1}^M Z(k,\xi).
    \end{equation}
    where
    \[ Z(k,\xi) = \begin{cases} e^{2 \pi i \xi \cdot X_n(k)} &: X_n(k) \in \Omega, \\ 0 &: X_n(k) \not \in \Omega \end{cases}. \]
    If $\Sigma$ is the $\sigma$-algebra generated by the random variables
    \[ \{ X_i(k) : i \in \{ 1, \dots, n-1 \}, k \in \{ 1, \dots, M \} \}, \]
    then $\Omega$ is measurable with respect to $\Sigma$. Thus the random variables $\{ Z(k,\xi) \}$ are \emph{conditionally independent} given $\Sigma$. Since we have $|Z(k,\xi)| \leq 1$ almost surely, Hoeffding's inequality thus implies that for all $t \geq 0$,
    \begin{equation} \label{equationCOIJCOIJX1232312}
        \PP \left( \left| H(\xi) - \EE(H(\xi)|\Sigma) \right| \geq t \right) \leq 4 \exp \left( \frac{-t^2}{2M} \right).
    \end{equation}
    It is simple to see that
    \begin{equation}
        \EE(H(\xi) | \Sigma) = A_n M \int_\Omega \psi_n(x) e^{2 \pi i \xi \cdot x}\; dx.
    \end{equation}
    Since
    \begin{equation}
        \Omega = \bigcup \left\{ N(f(X_1(k_1),\dots,X_{n-1}(k_{n-1})), cr) : 1 \leq k_1,\dots,k_{n-1} \leq N \right\}.
    \end{equation}
    % K^{n-1} = r^{(n-1) \varepsilon_1/2 - d}
    we therefore see that varying each random variable $X_i(k)$, for $1 \leq i \leq n-1$ while fixing the other random variables adjusts at most $M^{n-2}$ of the balls forming $\Omega$, and thus varying $X_i(k)$ while fixing the other random variables changes $\EE(H(\xi)|\Sigma)$ by at most
    \begin{equation}
        M \cdot 2 \cdot (2cr)^d \cdot M^{n-2} \lesssim_{n,d,L} r^d M^{n-1} \lesssim 1.
    \end{equation}
    % K^{-(n-1)} = r^d
    % 
    Thus McDiarmid's inequality shows that there exists a constant $C$ depending on $d$, $n$, and $L$, such that for any $t \geq 0$,
    \begin{equation} \label{equationCIJCIJIVJIO}
        \PP \left( |\EE(H(\xi)|\Sigma) - \EE(H(\xi))| \geq t \right) \leq 4 \exp \left( \frac{-t^2}{C M} \right).
    \end{equation}
    Combining \eqref{equationCOIJCOIJX1232312} and \eqref{equationCIJCIJIVJIO}, we conclude that there exists a constant $C > 0$ such that for each $\xi \in \ZZ^d$,
    \begin{equation} \label{equationCNCIJIJOJOPPPPOPODAW}
        \PP \left( | H(\xi) - \EE(H(\xi)) | \geq t  \right) \leq 8 \exp \left( \frac{-M t^2}{C} \right).
    \end{equation}
    Applying a union bound to \eqref{equationCNCIJIJOJOPPPPOPODAW} over all $0 < |\xi| \leq M^{1 + \kappa}$ shows that there exists a constant $C > 0$ such that
    \[ \PP \left( \sup_{|\xi| \leq M^{1 + \kappa}} | H(\xi) - \EE(H(\xi)) | \geq C M^{1/2} \log(M)^{1/2} \right) \leq 1/10. \qedhere \]
\end{proof}

The analysis of \eqref{equationCIJCIJIJXSO} requires a different class of probabilistic techniques. For any set $E \subset \TT^{d(n-1)}$, let $A(E)$ denote the event that there exists indices $k_1,\dots, k_{n-1}$ such that
\[ (X_1(k_1), \dots, X_{n-1}(k_{n-1})) \in E. \]
Understanding the quantity $\EE(H(\xi))$ will follow from an analysis of the values $\PP(A(E))$. To see why this is true, note that because the random variables $\{ X_n(k) \}$ are all identically distributed, for a fixed $\xi$, all the quantities $\{ Z(k,\xi) \}$ are identically distributed, and so
\begin{align*}
    \EE(H(\xi)) &= M A_n \cdot \EE(Z(1,\xi))\\
    &= M A_n \cdot \EE( \mathbf{I}(X_n(1) \in \Omega) e^{2 \pi i \xi \cdot X_n(1)} )\\
    &= M A_n \cdot \int \PP(X_n(1) \in \Omega | X_n(1) = x_n) e^{2 \pi i \xi \cdot x_n} d\PP(X_n(1) = x_n)\\
    &= M A_n \cdot \int \psi_n(x_n) \PP(1 \in I | X_n(1) = x_n) e^{2 \pi i \xi \cdot x_n}\; dx_n.
\end{align*}
If we write $E_{x_n} = f^{-1}(B_{cr}(x_n))$, then $\PP(1 \in I | X_n(1) = x_n) = \PP(A(E_{x_n}))$, so that
\begin{equation} \label{ejowiajeoijo1ij2312j}
    \EE(H(\xi)) = MA_n \cdot \int \psi_n(x_n) \PP(A(E_{x_n})) \cdot e^{2 \pi i \xi \cdot x_n}\; dx_n.
\end{equation}
If $n = 2$, then we can explicitly calculate $\PP(A(E)) \approx 1 - (1 - |E|)^M$ for any set $E$, which makes this analysis of $\EE[H(\xi)]$ more tractable. If $n > 2$, the random vectors
\[ \{ (X_1(k_1), \dots, X_{n-1}(k_{n-1})) : 1 \leq k_1, \dots k_{n-1} \leq M \} \]
are \emph{not} independent of one another, which makes an analysis of the quantities $\PP(A(E))$ difficult. An exception to this is when $E = E_1 \times \dots \times E_{n-1}$ is a Cartesian product, in which case
\[ \PP(A(E)) = \prod_{i = 1}^{n-1} \PP(\text{There is $k$ such that}\ X_i(k) \in E_i) \approx \prod_{i = 1}^{n-1} (1 - (1 - |E_i|)^M). \]
Our strategy to understanding the sets $E_{x_n}$ for $n > 2$ is therefore to apply the Whitney decomposition Lemma, writing $E_{x_n} = \bigcup_i Q_i$ for a family of almost disjoint, axis-oriented cubes $Q_i$. Then $A(E_{x_n}) = \bigcup_i A(Q_i)$, so we can approximate $\PP(A(E_{x_n}))$ by applying the inclusion-exclusion principle. This leads to a sufficient approximation to the values $\PP(A(E_{x_n}))$ provided that $M \ll r^{-d/(n-3/4)}$. This is the only part of the proof of Theorem \ref{theoremJOICVIOJVI122} where the dimension bound becomes tight; increasing the dimension bound in Theorem \ref{theoremJOICVIOJVI122} will be immediate if we can improve the following Lemma, i.e. finding a better analysis of $\EE[H(\xi)]$.

\begin{lemma} \label{lemmaOIJIOCJSOIJSIOJ123}
    Let $H(\xi)$ be the random exponential sum described in Lemma \ref{lemmaOIOICJOIJOISJOIJS}. Then there exists $C > 0$ such that for any $\delta > 0$, there exists $M_0 > 0$ such that for $M \geq M_0$,
    \[ |\EE(H(\xi))| \leq \delta M |\xi|^{-\beta/2} + C M^{1/2}. \]
\end{lemma}
\begin{proof}
    We break the analysis of $\EE(H(\xi))$ into two cases, depending on whether $n = 2$ or $n > 2$. Let's start with the case $n = 2$, in which case our assumptions imply that $f$ is a diffeomorphism if the cubes $R_1$ and $R_2$ in which we are choosing random points are chosen small enough. For each $x_2 \in \TT^d$, a change of variables shows that
    \begin{equation}
    \begin{split}
        \PP(A(E_{x_2})) &= 1 - \left( 1 - \int_{f^{-1}(B_{cr}(x_2))} \psi_1(x_1)\; dx_1 \right)^M\\
        &= 1 - \left( 1 - \int_{B_{cr}(x_2)} \frac{(\psi_1 \circ f^{-1})(x_1)}{|\det(Df)(f^{-1}(x_1))|}\; dx_1 \right)^M\\
        &= 1 - \left( 1 - \int_{B_{cr}(x_2)} \tilde{\psi_1}(x_1)\; dx_1 \right)^M,
    \end{split}
    \end{equation}
    where
    \[ \tilde{\psi_1}(x_1) = \frac{(\psi_1 \circ f^{-1})(x_1)}{|\det(Df)(f^{-1}(x_1))|}. \]
    If we define $g(x_2) = \PP(A(E_{x_2}))$, then  $\EE(H(\xi)) = M A_2 \cdot \widehat{\psi_2 g}(\xi)$. We can obtain a bound on $\EE(H(\xi))$ by bounding the partial derivatives of $\psi_2 g$. Bernoulli's inequality implies that
    \begin{equation}
        g(x) = 1 - \left( 1 - \int_{B_{cr}(x)} \tilde{\psi_1}(x_1)\; dx_1 \right)^M \lesssim_L M r^d \lesssim M^{1 - d/\lambda}.
    \end{equation}
    On the other hand, for any multi-index $\alpha$ with $|\alpha| > 0$, $\partial^\alpha g(x)$ is a sum of terms of the form
    \begin{equation} \label{equationDOIJACOIJCIOJ3123123214312}
        (-1)^m \frac{M!}{(M-m)!} \left( 1 - \int_{B_{cr}(x)} \tilde{\psi_1}(x_1)\; dx_1 \right)^{M-m} \left( \prod_{i = 1}^{m} \int_{B_{cr}(x)} \partial_{\alpha_i} \tilde{\psi_1}(x_1)\; dx_1 \right),
    \end{equation}
    where $\alpha_i \neq 0$ for any $i$ and $\alpha = \alpha_1 + \dots + \alpha_m$. This implies $0 < m \leq |\alpha|$ for any terms in the sum. Now the bound $|\partial_{\alpha_i} \tilde{\psi_1}(x_2)| \lesssim_{\alpha_i} 1$ implies that
    \begin{equation} \label{equationDOIAJCOIJAWCIOJAWOIJWAOI}
        \left| \int_{B_{cr}(x)} \partial_{\alpha_i} \tilde{\psi_1}(x_2)\; dx_2 \right| \lesssim_{\alpha_i} r^d.
    \end{equation}
    Applying \eqref{equationDOIAJCOIJAWCIOJAWOIJWAOI} to \eqref{equationDOIJACOIJCIOJ3123123214312} enables us to conclude that
    \begin{equation} \label{equationIOJACIOAJCOIWAJ}
        |\partial_\alpha g(x)| \lesssim_\alpha \max_{0 < m \leq |\alpha|} M^m r^{md} \leq M^{1 - d/\lambda},
    \end{equation}
    Since the fact that $\psi_2 \in C^\infty(\mathbb{T}^d)$ implies that $\| \partial_\alpha \psi_2 \|_{L^\infty(\TT^d)} \lesssim_\alpha 1$ for any multi-index $\alpha$, the product rule applied to \eqref{equationIOJACIOAJCOIWAJ} implies that $\| \partial_\alpha (\psi_2 g) \|_{L^\infty(\TT^d)} \lesssim_\alpha M^{1 - d/\lambda}$ for all $\alpha > 0$, which means that for any $T > 0$ and $\xi \neq 0$,
    \begin{equation} \label{equationCIOCIJIXJIXJI}
        |\EE(H(\xi))| \lesssim_T M^{2 - d/\lambda} |\xi|^{-T}.
    \end{equation}
    Since $\lambda < d$, $2 - d / \lambda < 1$, so setting $T = \beta/2$, fixing $\delta > 0$, and then choosing $M_0$ appropriately, if $M \geq M_0$, \eqref{equationCIOCIJIXJIXJI} shows that
    \begin{equation}
        |\EE(H(\xi))| \leq \delta M |\xi|^{-\beta/2}.
    \end{equation}
    This completes the proof in the case $n = 2$.

    Now we move on to the case where $n \geq 3$, which is made more complicated by the lack of an explicit formula for $\PP(A(E_{x_n}))$. For any cube $Q \in \TT^{d(n-1)}$ and any indices $1 \leq k_1,\dots,k_{n-1} \leq K$, set $k = (k_1,\dots,k_{n-1})$ and let $A(Q;k)$ denote the event that $(X_1(k_1),\dots,X_{n-1}(k_{n-1})) \in Q$. Then
    \begin{equation}
        A(Q) = \bigcup_k A(Q;k).
    \end{equation}
    For any cube $Q$ and index $k$,
    \begin{equation}
        \PP(A(Q;k)) = \int_Q \psi_1(x_1) \dots \psi_{n-1}(x_{n-1})\; dx_1 \dots dx_{n-1},
    \end{equation}
    and so
    \begin{equation} \label{equationOIJCOIJOCIJOAIJOIJAWD}
        \sum_k \PP(A(Q;k)) = M^{n-1} \int_Q \psi_1(x_1) \cdots \psi_{n-1}(x_{n-1})\; dx_1 \dots dx_{n-1}.
    \end{equation}
    An application of inclusion exclusion to \eqref{equationOIJCOIJOCIJOAIJOIJAWD} thus shows that
    \begin{equation} \label{equationIOJVOIVJOVIJPSPOPCOISAPCOIACC}
    \begin{split}
        &\left| \PP(A(Q)) - M^{n-1} \int_Q \psi_1(x_1) \cdots \psi_{n-1}(x_{n-1})\; dx_1 \dots dx_{n-1} \right|\\
        &\quad\quad\quad\leq \sum_{k \neq k'} \PP(A(Q;k) \cap A(Q;k')).
    \end{split}
    \end{equation}
    For each $k,k'$, the quantity $\PP(A(Q;k) \cap A(Q;k'))$ depends on the number of indices $i$ such that $k_i = k_i'$. In particular, if $I \subset \{ 1, \dots, n-1 \}$ is the set of indices where the quantity agrees, then
    \begin{equation}
    \begin{split}
        \PP(A(Q;k) \cap A(Q;k')) = \left( \prod_{i \in I} \int_{Q_i} \psi_i(x)\; dx \right) \cdot \left( \prod_{i \not \in I} \left( \int_{Q_i} \psi_i(x)\; dx \right)^2 \right).
    \end{split}
    \end{equation}
    In particular, if $Q$ has sidelength $l$ and $\#(I) = m$, then $\PP(A(Q;k) \cap A(Q;k')) \lesssim l^{d(2n - m - 2)}$. For each $m$, there are at most $M^{2n - m - 2}$ pairs $k$ and $k'$ with $\#(I) = m$. And so provided $l^d \leq 1/M$,
    \begin{equation} \label{equationCIOJIOJIJCS312412412}
        \sum_{k \neq k'} \PP(A(Q;k) \cap A(Q;k')) \lesssim \sum_{m = 0}^{n-2} (M \cdot l^d)^{2n-m-2} \lesssim M^n l^{dn}.
    \end{equation}
    Thus we conclude from \eqref{equationIOJVOIVJOVIJPSPOPCOISAPCOIACC} and \eqref{equationCIOJIOJIJCS312412412} that
    \begin{equation} \label{equationCIOJOIJXOISJOIj}
        \PP(A(Q)) = M^{n-1} \int_Q \psi_1(x_1) \dots \psi_{n-1}(x_{n-1}) dx_1 \dots dx_{n-1} + O(M^n l^{dn}).
    \end{equation}
    Since $f$ is a submersion, for each $x_n$, $E_{x_n}$ is contained in a $O(r)$-thickening of a $d(n-2)$ dimensional surface in $\TT^{d(n-1)}$. Applying the Whitney covering lemma, we can find a family of almost disjoint dyadic cubes $\{ Q_{ij} : j \geq 0 \}$ such that
    \begin{equation} \label{equationCOIJIOVJVIVIVIII2231}
        E_{x_n} = \bigcup_{i = 0}^\infty \bigcup_{j = 1}^{n_i} Q_{ij}(x_n),
    \end{equation}
    where for each $i \geq 0$, $Q_{ij}$ is a sidelength $r/2^i$ cube, and $n_i \lesssim (r/2^i)^{-d(n-2)}$. It follows from \eqref{equationCOIJIOVJVIVIVIII2231} that
    \begin{equation} \label{equationCOJIAWOIJCAWOIJOI}
        A(E_{x_n}) = \bigcup_{i,j} A(Q_{ij}).
    \end{equation}
    Since $n \geq 3$, we can use \eqref{equationCIOJOIJXOISJOIj} to calculate that
    \begin{equation} \label{equationCOIJCOIJSI}
    \begin{split}
        &\left| \sum_{i,j} \PP(A(Q_{ij})) - M^{n-1} \int_{E_{x_n}} \psi_1(x_1) \dots \psi_{n-1}(x_{n-1})\; dx \right|\\
        &\quad\quad\quad\lesssim \sum_{i = 0}^\infty (r/2^i)^{-d(n-2)} \cdot \left( M^n (r/2^i)^{dn} \right)\\
        &\quad\quad\quad\lesssim r^{2d} M^n \leq M^{-1/2}.
    \end{split}
    \end{equation}
    % r \leq 1/M^{(n/2 + 1/4)/d}
    Thus an inclusion exclusion bound together with \eqref{equationCOJIAWOIJCAWOIJOI} and \eqref{equationCOIJCOIJSI} implies that
    \begin{equation} \label{equationvVIDJDIJ21312ffijsijds}
    \begin{split}
        &\Big| \PP(A(E_{x_n})) - M^{n-1} \int_{E_{x_n}} \psi_1(x_1) \dots \psi_{n-1}(x_{n-1})\; dx \Big|\\
        &\quad\quad\quad \lesssim M^{-1/2} + \sum_{(i_1,j_1) \neq (i_2,j_2)} \PP(A(Q_{i_1j_1}) \cap A(Q_{i_2j_2})).
    \end{split}
    \end{equation}
    The quantity $\PP(A(Q_{i_1j_1}) \cap A(Q_{i_2j_2}))$ depends on the relation between the various sides of $Q_{i_1j_1}$ and $Q_{i_2j_2}$. Without loss of generality, we may assume that $i_1 \geq i_2$. If $I(Q_{i_1j_1},Q_{i_2j_2})$ is the set of indices $1 \leq k \leq n-1$ where $Q_{i_1j_1k} \subset Q_{i_2j_2k}$, and $\#(I(Q_{i_1j_1}, Q_{i_2j_2})) = m$, then
    \begin{equation} \label{equationVOIJVIJISJCISJCIEWJRIJI43234}
    \begin{split}
        \PP(A(Q_{i_1j_1}) \cap A(Q_{i_2j_2})) &\lesssim (M(r/2^{i_1})^d)^m \cdot (M(r/2^{i_1})^d \cdot M(r/2^{i_2})^d)^{n-m-1}\\
        &= 2^{-d[(n-1)i_1 + (n-m-1)i_2]} (Mr^d)^{2n - m-2}.
    \end{split}
    \end{equation}
    %
%    \[ \PP(A(Q) \cap A(Q')) = \prod_{i \in I} B_i(Q_i) \cdot \prod_{i \not \in I} C_i(Q_i) \]
    %
%    where
    %
%    \[ B_i(Q_i) = 1 - \left(1 - \int_{Q_i} \psi_i(x)\; dx \right)^K \]
    %
%    and
%    \[ C_i(Q_i) = 1 - \left( 1 - \int_{Q_i} \psi_i(x)\; dx \right)^K - \left( 1 - \int_{Q_i'} \psi_i(x)\; dx \right)^K + \left( 1 - \int_{Q_i \cup Q_i'} \psi_i(x)\; dx \right)^K. \]
    %
%    Then $|B_i(Q_i)| \lesssim K r^d$ and $|C_i(Q_i)| \lesssim K^2 r^{2d}$. If $\#(I) = m$, then we would conclude that $\PP(A(Q) \cap A(Q')) \lesssim K^{2n-m-2} r^{d(2n - m-2)}$.
    The condition that $D_{x_k} f$ is invertible for all $k$ on the domain of $f$ implies that any axis-oriented plane in $\TT^{dn}$ intersects transversally with the level sets of $f$. In particular, this means that the intersection of a $O(r/2^{i_1})$ thickening of a codimension $dm$ axis-oriented hyperplane intersects a $O(r/2^{i_1})$ thickening of $\partial E_{x_n}$ (which has codimension $d$) in a set with volume $O \left( (r/2^{i_1})^d (r/2^{i_1})^{dm} \right)$, and intersects a $O(r/2^{i_2})$ thickening of $\partial E_{x_n}$ in a set with volume $O \left( (r/2^{i_2})^d (r/2^{i_1})^{dm} \right)$. As a particular example of this, for any distinct indices $j_1,\dots,j_m \in \{ 1,\dots, n-1 \}$, and any family of integers $0 \leq n_{11},\dots,n_{md} \leq 2^{i_1}/r$, the set
    \begin{equation}
        \left\{ x \in E_{x_n} : \frac{n_{11}}{2^{i_1}} \leq x_{j_1 1} \leq \frac{(n_{11} + 1)}{2^{i_1}}, \dots, \frac{n_{md}}{2^{i_1}} \leq x_{j_m d} \leq \frac{n_{md} + 1}{2^{i_1}} \right\}
    \end{equation}
    contains at most
    \begin{equation} \label{equationCOIJCOIJSOIJIOJ424141241}
        O \left( (r/2^{i_1})^d (r/2^{i_1})^{dm} (r/2^{i_1})^{-d(n-1)} \right) = O \left( 2^{d(n-m-2)i_1} r^{-d(n-m-2)} \right)
    \end{equation}
    sidelength $r/2^{i_1}$ dyadic cubes in the decomposition of $E_{x_n}$, and at most
    \begin{equation} \label{equationVOIJVOIJVOIJPOPOJP1212312}
        O \left( (r/2^{i_2})^d (r/2^{i_1})^{dm} (r/2^{i_2})^{-d(n-1)} \right) = O \left( 2^{d(n-2) i_2 - (dm) i_1} r^{-d(n-m-2)} \right)
    \end{equation}
    sidelength $r/2^{i_2}$ dyadic cubes in the decomposition of $E_{x_n}$. Letting the integers $\{ n_{kl} \}$ vary over all possible choices we conclude from \eqref{equationCOIJCOIJSOIJIOJ424141241} and \eqref{equationVOIJVOIJVOIJPOPOJP1212312} that for each $i_1$ and $i_2$ there are at most
    \begin{equation} \label{equationCOIJCOIJSOIJOIJ241240912490124091}
    \begin{split}
        &O \left( (2^{i_1}/r)^{dm} \left( 2^{d(n-m-2)i_1} r^{-d(n-m-2)} \right) \left( 2^{d(n-2) i_2 - (dm) i_1} r^{-d(n-m-2)} \right) \right)\\
        &\quad\quad = O \left( 2^{d(n - m - 2)i_1 + d(n-2) i_2} r^{-d(2n - m - 4)} \right)
    \end{split}
    \end{equation}
    pairs $Q_{i_1j_1}$ and $Q_{i_2j_2}$ with $I(Q_{i_1j_1},Q_{i_2j_2}) = m$. Thus we conclude from \eqref{equationVOIJVIJISJCISJCIEWJRIJI43234} and \eqref{equationCOIJCOIJSOIJOIJ241240912490124091} that
    \begin{equation} \label{equationPPOPOKPOPPPPPPDSDSD}
    \begin{split}
        \sum_{(i,j) \neq (i',j')}& \PP(A(Q_{ij}) \cap A(Q_{i'j'}))\\
        &\lesssim \sum_{m = 0}^{n-2} \sum_{i_1 \geq i_2} \left( 2^{d(n-m-2) i_1 + d(n-2) i_2} r^{-d(2n-m-4)} \right)\\
        &\quad\quad\quad\quad\quad\quad\quad\quad\left( 2^{-d((n-1)i_1 + (n-m-1) i_2)} (Mr^d)^{2n - m - 2} \right)\\
        &\lesssim r^{2d} \sum_{m = 0}^{n-2} M^{2n-m-2} \sum_{i_1 \geq i_2} 2^{-d(m+1)i_1 + d(m-1)i_2}\\
        &\lesssim \sum_{m = 0}^{n-2} M^{2n-m-2} r^{2d}\\
        &\lesssim M^{2(n-1)} r^{2d} \lesssim M^{-1/2}.
    \end{split}
    \end{equation}
    Returning to the bound in \eqref{equationvVIDJDIJ21312ffijsijds}, \eqref{equationPPOPOKPOPPPPPPDSDSD} implies that
    \begin{equation} \label{equationFFOGOOBOBOOOTTUUYYUYU7412412}
        \left| \PP(A(E_{x_n})) - M^{n-1} \int_{E_{x_n}} \psi_1(x_1) \dots \psi_{n-1}(x_{n-1})\; dx_1 \dots dx_{n-1} \right| \lesssim M^{-1/2}.
    \end{equation}
    Returning even further back to \eqref{ejowiajeoijo1ij2312j}%{equationGGIJICJIjjpopwaowarr}
    , recalling that $E_{x_n} = f^{-1}(B_r(x_n))$, \eqref{equationFFOGOOBOBOOOTTUUYYUYU7412412} implies
    \begin{equation} \label{equationOIJCIOJSOIJ121231}
        \left| \EE(H(\xi)) - A_n \cdot M^n \int_{\TT^d} \psi_n(x_n) \int_{f^{-1}(B_r(x_n))} \psi_1(x_1) \dots \psi_{n-1}(x_{n-1}) e^{2 \pi i \xi \cdot x_n}\; dx_1\; \dots\; dx_n \right| \lesssim M^{1/2}.
    \end{equation}
    Applying the co-area formula, writing $\psi(x) = \psi_1(x_1) \dots \psi_n(x_n)$, we find
    \begin{equation} \label{equationCOIJCOIJSIOJOI1231}
    \begin{split}
        \int_{\TT^d} & \int_{f^{-1}(B_r(x_n))} \psi(x) e^{2 \pi i \xi \cdot x_n}\; dx_1\; \dots\; dx_n\\
        &= \int_{B_r(0)} \int_{\TT^d} \int_{f^{-1}(x + v)} \psi(x) e^{2 \pi i \xi \cdot x_n}\; dH^{n-2}(x_1,\dots,x_{n-1})\; dx_n\; dv\\
        &= \int_{B_r(0)} \int_{\TT^{d(n-1)}} \psi(x,f(x) - v) \cdot e^{2 \pi i \xi \cdot (f(x) - v)} |Jf(x)|\; dx\; dv\\
        &= \int_{B_r(0)} \int_{\TT^{d(n-1)}} \tilde{\psi}(x,v) \cdot e^{2 \pi i \xi \cdot (f(x) - v)}\; dx\; dv.
    \end{split}
    \end{equation}
    % xi = 0
    % and f(x) = v
    where $\tilde{\psi}(x,v) = \psi(x,f(x) - v) \cdot |Jf(x)|$, and $Jf$ is the rank-$d$ Jacobian of $f$. A consequence of \eqref{equationCOIJCOIJSIOJOI1231} in light of \eqref{equationOIJCIOJSOIJ121231} is that it reduces the study of $\EE(H(\xi))$ to a standard oscillatory integral. If we look at the phase
    \[ \phi(x,v) = \xi \cdot (f(x) - v) \]
    then we see that $\nabla_x \phi(x,v) = Df(x)^T \xi$, which is only equal to zero if $\xi = 0$ since $Df$ is surjective on the domain of $f$ (this is implied by the stronger assumption that $f$ is a diffeomorphism on each variable). Thus the oscillatory integral above has no stationary points in the $x$-variable. Integrating by parts in the $x$-variable thus allows us to conclude that for all $|v| \leq 1$ and $T > 0$,
    \begin{equation} \label{awdoiajdoawijdao41412412312}
        \left|\int_{\TT^{d(n-1)}} \tilde{\psi}(x,v) \cdot e^{2 \pi i \xi \cdot (f(x) - v)}\; dx \right| \lesssim_T |\xi|^{-T}.
    \end{equation}
    Now the bound in \eqref{awdoiajdoawijdao41412412312} can be applied with \eqref{equationCOIJCOIJSIOJOI1231} to conclude that
    \begin{equation} \label{qjweoiqwjeoi3423412321321}
        \left| \int_{\TT^d} \int_{f^{-1}(B_r(x))} \psi(x) e^{2 \pi i \xi \cdot x_n}\; dx_2\; \dots\; dx_n\; dx_1 \right| \lesssim_T r^d |\xi|^{-T}.
    \end{equation}
    In particular, taking $T = \beta/2$ here, combined with \eqref{equationOIJCIOJSOIJ121231}, \eqref{qjweoiqwjeoi3423412321321}, we find that
    \begin{equation}
        |\EE(H(\xi)) | \lesssim M^n r^d |\xi|^{-\beta/2} + M^{1/2} \lesssim M^{3/4} |\xi|^{-\beta/2} + M^{1/2}.
    \end{equation}
    Thus there exists $C > 0$ such that for any $\delta > 0$, there is $r_0 > 0$ such that for $r \leq r_0$, and any nonzero $\xi \in \ZZ^d$,
    \[ |\EE(H(\xi))| \leq \delta M |\xi|^{-\beta/2} + CM^{1/2}. \qedhere \]
\end{proof}

The proof of Lemma \ref{lemmaOIJIOCJSOIJSIOJ123} is the only obstacle preventing us from constructing a Salem set $X$ avoiding the pattern defined by $Z$ with
\[ \fordim(X) = \frac{d}{n-1}. \]
All other aspects of the proof carry through for $d/(n-3/4) \leq \beta \leq d/(n-1)$. The problem with Lemma \ref{lemmaOIJIOCJSOIJSIOJ123} in this scenario is that if we try to repeat the proof when $n \geq 3$ and $M \gg r^{-d/(n-3/4)}$, there is too much `overlap' between the various cubes we use in our covering argument in the various axis; thus the inclusion-exclusion argument found in this proof cannot be used to control $\EE(H(\xi))$ in a significant way. We believe our method can construct Salem sets with Fourier dimension $d/(n-1)$, but new tools are required to improve the estimates on $\EE(H(\xi))$. In the next section, we are able to modify our construction for patterns satisfying a weak translation invariance by a simple trick: we will modify the analogous exponential sums $H(\xi)$ so that $\EE(H(\xi)) = 0$ for all $\xi \neq 0$, so that the analogue of Lemma \ref{lemmaOIJIOCJSOIJSIOJ123} is trivial in this setting, and we thus obtain a construction of Salem sets avoiding patterns with dimension exactly matching those obtained in the Hausdorff dimension setting.

\section{Expectation Bounds for Translational Patterns}

The proof of Theorem \ref{thirdTheorem} uses very similar arguments to Theorem \ref{theoremJOICVIOJVI122}. The concentration bound arguments will be very similar to those applied in the last section. The difference here is that the translation-invariance of the pattern can be used to bypass estimating the expected values like those which caused us the most difficulty in Theorem \ref{thirdTheorem}. We can therefore construct Salem sets avoiding the pattern with dimension exactly matching the Hausdorff dimension of the sets which would be constructed using the method of \cite{OurPaper}. In this section, let
\[ \beta \leq \min \left( \frac{dn - \alpha}{n-1}, d \right). \]
We then show that generic elements of $\mathcal{X}_\beta$ avoid patterns satisfying the assumptions of Theorem \ref{thirdTheorem}.

\begin{lemma}
    Fix $a \in \QQ - \{ 0 \}$, and let $f: V \to \RR$ and $F \subset \RR$ satisfy the assumptions of Theorem \ref{thirdTheorem}. Then for quasi-all $(E,\mu) \in \mathcal{X}_\beta$, and any distinct points $(x_1,\dots,x_n) \in E$,
    \[ x_n - ax_{n-1} - f(x_1,\dots,x_{n-2}) \not \in F. \]
\end{lemma}
\begin{proof}
    Set
    \[ W = \{ (x_1,\dots,x_n) \in \TT^{2d} \times V : x_n - ax_{n-1} - f(x_1,\dots,x_{n-2}) \in F \}. \]
    The assumption that $f$ is a locally Lipschitz map, and thus continuous, implies that for any disjoint, closed cubes $R_1,\dots,R_n \subset \TT^d$ such that $R_1 \times \dots \times R_{n-2} \subset V$, $(R_1 \times \dots \times R_n) \cap W$ will be a closed set. It follows that if we set
    \[ \mathcal{H}(R_1,\dots,R_n) = \{ E \subset \TT^d : (R_1 \times \cdots \times R_n) \cap W \cap E^n = \emptyset \} \]
    and
    \[ H(R_1,\dots,R_n) = \{ (E,\mu) \in \mathcal{X}_\beta: E \in \mathcal{H}(R_1,\dots,R_n) \}, \]
    then $H(R_1,\dots,R_n)$ is an open subset of $\mathcal{X}_\beta$, and $\mathcal{H}(R_1,\dots,R_n)$ is a downward closed family of sets. The proof will be complete will be proved that for each positive integer $m$, and any choice of cubes $R_1,\dots,R_n$ with common sidelength $1/2am$, with $d(R_i,R_j) \geq 10/am$ for $i \neq j$, and with $Q_1 \times \dots \times Q_{n-2}$ a closed subset of $V$, where $Q_i = 2R_i$, then the set $H(R_1,\dots,R_n)$ is dense in $\mathcal{X}_\beta$. To prove $H(R_1,\dots,R_n)$ is dense, we may assume without loss of generality that the set $F$ is $1/m$ periodic, i.e. $F + k / m = F$ for any $k \in \ZZ^d$, by replacing $F$ with a finite union of it's translates. The set $H(R_1,\dots,R_n)$ is downward closed, so we can apply Lemma \ref{LemmaFIOAJFOIWJ}. 
%    Moreover, the set $\tilde{S}$ will then be a $1/m$ periodic subset of $\TT^d$. If we define
    %
%    \[ \tilde{W} = \{ (x_1,\dots,x_n) \in \TT^{2d} \times V : x_n - ax_{n-1} \in \tilde{S}(x_1,\dots,x_{n-2}) \}. \]
    %
%    and
    %
%    \[ \tilde{H}(R_1,\dots,R_n) = \{ (E,\mu) \in \mathcal{X}_\beta: (R_1 \times \cdots \times R_n) \cap W \cap E^n = \emptyset \}, \]
    %
%    then $\tilde{H}(R_1,\dots,R_n)$ is an open subset of $H(R_1,\dots,R_n)$. Showing $\tilde{H}(R_1,\dots,R_n)$ is dense thus implies that $H(R_1,\dots,R_n)$ is dense. Thus we need only concentrate on the proof that $\tilde{H}(R_1,\dots,R_n)$ is dense, which is why the periodicity assumptions above are harmless.

    Since $f$ is a locally Lipschitz map, we may fix $L > 0$ such that for $x_1,x_2 \in R_1 \times \dots \times R_{n-2}$,
    \begin{equation} \label{equationDIOJAWOICJAOIVNAEN2313}
        |f(x_1) - f(x_2)| \leq L |x_1 - x_2|.
    \end{equation}
    Fix a large integer $M > 0$, $\lambda \in [0,\beta)$, $\gamma \in [0,\alpha)$ and pick $r > 0$ such that $r^{-\lambda} \leq M \leq r^{-\lambda} + 1$. If $c = 2(1 + |a| + L n^{1/2})$, and $r$ is suitably small, then
    \begin{equation} \label{equationAIOCJAOICJAWOICNIJNIJNQORI}
        |N(F,cr)| \leq r^{dn - \gamma}
    \end{equation}
    For $1 \leq i \leq n$, consider a family of independent random variables $\{ X_i(k): 1 \leq k \leq M \}$, such that $X_i(k)$ is uniformly distributed on $Q_i$ for each $i$, as well as another independent family of random variables $\{ X_0(k): 1 \leq k \leq M \}$ uniformly distributed in $\TT^d - (Q_1 \cup \dots \cup Q_n)$. Let $I$ be the set of indices $k_n \in \{ 1, \dots, N \}$ such that there are indices $k_1,\dots,k_{n-1} \in \{ 1, \dots, N \}$ with the property that
    \begin{equation}
        d_{\mathbb{H}} \Big( X_n(k_n) - a X_{n-1}(k_{n-1}) - f \big( X_1(k_1), \dots, X_{n-2}(k_{n-2}) \big), F \Big) \leq cr.
    \end{equation}
    % r + |a| r + L n^{1/2} r
    %
    If
    \[ S = \{ X_i(k) : 0 \leq i \leq n-1, 1 \leq k \leq N \} \cup \{ X_n(k): k \not \in I \}. \]
    then \eqref{equationDIOJAWOICJAOIVNAEN2313} implies that $N(S,r) \in \mathcal{H}(W;R_1,\dots,R_n)$.
    %The proof of this is analogous to the proof of the same property for the set $S$ in Lemma \ref{lemmaOIOICJOIJOISJOIJS}, which also selected points to remove on the `double' of the cubes $R_i$, though with respect to a different family of probability distributions.

    We claim that for each $x \in Q_n$, the quantity
    \begin{align*}
        P(x) &= \PP ( 1 \in I | X_n(1) = x)%\\
%        &= \PP \left( \begin{array}{c}
%            \text{There exists $k_1, \dots, k_{n-1} \in \{ 1, \dots, N \}$ such that}\\
 %           \text{ $d_{\mathbb{H}}(x - a X_{n-1}(k_{n-1}), T(X_1(k_1),\dots, X_{n-2}(k_{n-2}))) \leq 2 n^{1/2} \cdot (L+1) \cdot r$.} \end{array} \right).
    \end{align*}
    is \emph{independent} of $x$. To see this, we note that because $F$ is $1/m$ periodic, the quantity
    \[ d_{\mathbb{H}} \Big( x - a x_{n-1} - f \big( x_1,\dots,x_{n-2} \big), F \Big) \]
    depends only on $x_1,\dots,x_{n-2}$, and the value of $x - a x_{n-1}$ in $\TT^d / (\ZZ^d / m)$.  Because $a X_{n-1}(k_{n-1})$ is uniformly distributed in $a Q_n$, which is an axis-oriented cube with sidelength $1/m$, it follows that the distribution of the random variable
    \[ x - a X_{n-1}(k_{n-1}) \]
    modulo $\TT^d / (\ZZ^d / m)$, is independent of $x$, which yields the claim.

    Let $P$ denote the common quantity of the values $P(x)$, and let
    \[ G(x_1,\dots,x_{n-1}) = d_{\mathbb{H}} \Big( - a x_{n-1} - f \big( x_1, \dots, x_{n-2}) \big), F \Big). \]
    A union bound, together with \eqref{equationAIOCJAOICJAWOICNIJNIJNQORI}, implies
    \begin{equation} \label{equatojaciojwaofij231o41412414}
    \begin{split}
        P &= \PP \left( \bigcup_{k_1,\dots,k_{n-1}} \Big\{ G(X_1(k_1),\dots, X_{n-1}(k_{n-1})) \leq cr \Big\} \right)\\
        &\leq \sum_{k_1,\dots,k_{n-1}} \PP \left( G(X_1(k_1), \dots, X_{n-1}(k_{n-1})) \leq cr \right) \\
        &\lesssim M^{n-1} |N(F,cr)| \leq M^{n-1} r^{dn-\gamma} \lesssim M^{n - 1 - (dn - \gamma)/\lambda}
    \end{split}
    \end{equation}
    Because $(n-1) - (dn - \gamma)/\lambda < 0$, \eqref{equatojaciojwaofij231o41412414} implies that for suitably large integers $M$ depending on $n$, $d$, $\lambda$, and $\gamma$,
    \begin{equation}
        P \leq 1/2.
    \end{equation}
    Set $A_0 = (1 - P) |\TT^d - Q_1 - \dots - Q_n|$, set $A_i = (1 - P) |Q_i|$ if $i \in \{ 1, \dots, n - 1 \}$, and let $A_n = |Q_n|$. Define
    \[ F(\xi) = \sum_{i = 0}^{n-1} \sum_{k = 1}^M A_i e^{2 \pi i \xi \cdot X_i(k)} + \sum_{k \not \in I} A_n e^{2 \pi i \xi \cdot X_n(k)}. \]
    The choice of coefficients is made so that $\sum A_i \gtrsim 1$, and for any $\xi \neq 0$, $\EE(F(\xi)) = 0$. Indeed, we have
    \begin{align*}
        \EE(F(\xi)) &= \frac{M A_0}{|\TT^d - Q_1 - \dots - Q_n|} \int_{\TT^d - Q_1 - \dots - Q_n} e^{2 \pi i \xi \cdot x}\; dx\\
        &\quad\quad + \sum_{i = 1}^{n-1} \frac{M A_i}{|Q_i|} \int_{Q_i} e^{2 \pi i \xi \cdot x}\; dx\\
        &\quad\quad + \frac{M A_n}{|Q_n|} \int_{R_n} [1 - P(x)] e^{2 \pi i \xi \cdot x}; dx\\
        &= M (1 - P) \int_{\TT^d} e^{2 \pi i \xi \cdot x} = 0.
    \end{align*}
    Split up $F$ into the sum of two exponential sums
    \[ G(\xi) = \sum_{i = 0}^n \sum_{k = 1}^M A_i e^{2 \pi i \xi \cdot X_i(k)} \]
    and
    \[ H(\xi) = \sum_{k \in I} A_n e^{2 \pi i \xi \cdot X_n(k)}. \]
    Applying Lemma \ref{LemmaGISCICS1}, we conclude that for any fixed $\kappa > 0$, there is $C > 0$ such that
    \begin{equation} \label{adwadaf213213421ihwuifhdwquih}
        \PP \left( \sup_{|\xi| \leq N^{1 + \kappa}} |G(\xi) - \EE(G(\xi))| \geq C M^{1/2} \log(M)^{1/2} \right) \leq 1/10.
    \end{equation}
    Lemma \ref{lastconcentrationbound}, which follows from a very similar argument to Lemma \ref{lemma24901401921209} in the last section, implies that
    \begin{equation} \label{equfhwqfhwuqh213421u4hifhaf}
    \begin{split}
        \PP & \left( \sup_{|\xi| \leq M^{1 + \kappa}} | H(\xi) - \EE(H(\xi)) | \geq C M^{1/2} \log(M)^{1/2} \right) \leq 1/10.
    \end{split}
    \end{equation}
    A union bound applied to \eqref{adwadaf213213421ihwuifhdwquih} and \eqref{equfhwqfhwuqh213421u4hifhaf} implies that, since $\EE(F(\xi)) = 0$,
    \begin{equation}
        \PP \left( \sup_{|\xi| \leq M^{1 + \kappa}} | F(\xi) | \geq C M^{1/2} \log(M)^{1/2} \right) \leq 1/5
    \end{equation}
    Set $N = \#(S)$. Then
    \begin{equation} \label{equationDIOWAJDOIWAJIO}
        N \geq M \gtrsim (1/2) r^{-\lambda}.
    \end{equation}
    Now \eqref{adwadaf213213421ihwuifhdwquih}, \eqref{equfhwqfhwuqh213421u4hifhaf}, and \eqref{equationDIOWAJDOIWAJIO} imply that there exists a constant $C > 0$ and an instantiation of the random variables $\{ X_i(k) \}$ such that if $\tilde{a}(X_i(k)) = A_i$, then
    \begin{equation} \label{equationdIOJAWOidjawodij1243123}
        \left| \frac{1}{N} \sum_{x \in S} \tilde{a}(x) e^{2 \pi i \xi \cdot x} \right| \lesssim C N^{1/2} \log(N)^{1/2}.
    \end{equation}
    Since $\sum_{x \in S} \tilde{a}(x) \gtrsim N$, if we set
    \[ a(x) = N \cdot \frac{\tilde{a}(x)}{\sum_{x' \in S} a(x')}, \]
    then the sum
    \[ \frac{1}{N} \sum_{x \in S} a(x) e^{2 \pi i \xi \cdot x} \]
    satisfies the assumptions of Lemma \ref{LemmaFIOAJFOIWJ} for arbitrarily large $N$. We therefore conclude by that Lemma that $H(R_1,\dots,R_n)$ is dense in $\mathcal{X}_\beta$.
\end{proof}

All that remains to prove Theorem \ref{thirdTheorem} is to prove Lemma \ref{lastconcentrationbound}.

\begin{lemma} \label{lastconcentrationbound}
    For any $\kappa > 0$, there exists $C > 0$ such that
    \[ \PP \left( \sup_{|\xi| \leq M^{1 + \kappa}} | H(\xi) - \EE(H(\xi)) | \geq C M^{1/2} \log(M)^{1/2} \right) \leq 1/10 \]
\end{lemma}
\begin{proof}
    Consider the random set $\Omega$ of values $x_n \in Q_n$ such that there are $k_1,\dots,k_{n-1} \in \{ 1,\dots,M \}$ with
    \begin{equation}
        x_n - f(X_1(k_1), \dots, X_{n-2}(k_{n-2})) \in N(F,cr).
    \end{equation}
    Then
    \begin{equation}
        H(\xi) = \frac{A_n}{M} \sum_{k = 1}^M Z(k,\xi).
    \end{equation}
    where
    \[ Z(k,\xi) = \begin{cases} e^{2 \pi i \xi \cdot X_n(k)} &: X_n(k) \not \in \Omega, \\ 0 &: X_n(k) \in \Omega \end{cases}. \]
    If $\Sigma$ is the $\sigma$ algebra generated by the random variables
    \[ \{ X_i(k) : i \in \{ 1, \dots, n-1 \}, k \in \{ 1, \dots, M \} \}, \]
    then $\Omega$ is measurable with respect to $\Sigma$. Thus the random variables $\{ Z(k,\xi) \}$ are \emph{conditionally independent} given $\Sigma$. Since we have $|Z(k,\xi)| \leq 1$ almost surely, Hoeffding's inequality thus implies that for all $t \geq 0$,
    \begin{equation} \label{equationCOIJCOIJX1232312ssss}
        \PP \left( \left| H(\xi) - \EE(H(\xi)|\Sigma) \right| \geq t \right) \leq 4 \exp \left( \frac{-t^2}{2M} \right).
    \end{equation}
    It is simple to see that
    \begin{equation}
        \EE(H(\xi) | \Sigma) = A_n M \int_\Omega \psi_n(x) e^{2 \pi i \xi \cdot x}\; dx.
    \end{equation}
    Since
    \begin{equation}
        \Omega = \bigcup \left\{ f(X_1(k_1), \dots, X_{n-2}(k_{n-2})) + N(F, cr) : 1 \leq k_1,\dots,k_{n-1} \leq K \right\}.
    \end{equation}
    % K^{n-1} = r^{(n-1) \varepsilon_1/2 - d}
    we see that varying each random variable $X_i(k)$, for $1 \leq i \leq n-1$ while fixing the other random variables adjusts at most $M^{n-2}$ of the sets forming $\Omega$, each of which having volume $O_{d,n,L}(r^{dn - \alpha})$, and thus varying $X_i(k)$ while fixing the other random variables changes $\EE(H(\xi)|\Sigma)$ by at most
    \begin{equation}
        M \cdot O_{d,n,L}(r^{dn - \alpha}) \cdot M^{n-1} \lesssim 1.
    \end{equation}
    % r^{dn - s} M^{n-2}
    % M^{-1}
    Thus McDiarmid's inequality shows that there exists $C > 0$ such that for any $t \geq 0$,
    \begin{equation} \label{equationCIJCIJIVJIOsssssdadwadwad}
        \PP \left( |\EE(H(\xi)|\Sigma) - \EE(H(\xi))| \geq t \right) \leq 4 \exp \left( \frac{-t^2}{CM} \right).
    \end{equation}
    Combining \eqref{equationCOIJCOIJX1232312ssss} and \eqref{equationCIJCIJIVJIOsssssdadwadwad}, we conclude that there exists $C > 0$ such that for each $\xi \in \ZZ^d$,
    \begin{equation} \label{equationCNCIJIJOJOPPPPOPODAWssssssssss}
        \PP \left( | H(\xi) - \EE(H(\xi)) | \geq t  \right) \leq 8 \exp \left( \frac{-t^2}{CM} \right).
    \end{equation}
    Applying a union bound to \eqref{equationCNCIJIJOJOPPPPOPODAWssssssssss} over all $0 < |\xi| \leq M^{1 + \kappa}$ shows that there exists a constant $C > 0$ such that
    \[ \PP \left( \sup_{|\xi| \leq M^{1 + \kappa}} | H(\xi) - \EE(H(\xi)) | \geq C M^{1/2} \log(M)^{1/2} \right) \leq 1/10. \qedhere \]
\end{proof}

\section{Appendix: Justifying Discretization}

The main goal of this appendix is a proof of Lemma \ref{LemmaFIOAJFOIWJ}. Throughout this section, we will apply mollification. So we fix a smooth, non-negative function $\phi \in C^\infty(\TT^d)$ such that $\phi(x) = 0$ for $|x| \geq 2/5$ and $\int_{\TT^d} \phi(x)\; dx = 1$.
%
\begin{comment}
\begin{theorem} \label{equationASFGCISIX}
    There exists a smooth probability density $\phi \in C^\infty(\TT^d)$ such that $\phi(x) = 0$ for $|x| \geq 2/5$, and such that for each $x \in \TT^d$
    %
    \[ \sum_{k \in \{ 0, 1 \}^d} \phi(x + k/2) = 2^d. \]
\end{theorem}
\begin{proof}
    Let $\psi$ be a non-negative smooth function on $\TT$ such that $\psi(x) = \psi(- x)$ for all $x \in \TT$, $\psi(x) = 1$ for $|x| \leq 1/10$, $\psi(x) = 0$ for $|x| \geq 2/10$, and $0 \leq \psi(x) \leq 1$ for all $x \in \TT$. Then define $\eta$ to be the non-negative, $C^\infty$ function
    %
    \[ \eta(x) = \frac{1}{2} - \frac{\psi(x) + \psi(x + 1/2)}{2}. \]
    %
    If we define
    %
    \[ \phi_0(x) = 2(\psi(x) + \eta(x)), \]
    %
    then $\phi_0(x) + \phi_0(x + 1/2) = 2$ for all $x \in \TT$. Moreover, if $|x| \geq 2/5$, then $\psi(x) = 0$, and since this implies $|x + 1/2| \leq 1/10$, we find $\eta(x) = 0$. Thus $\phi_0(x) = 0$ for $|x| \geq 2/5$. But the condition $\phi_0(x) + \phi_0(x + 1/2) = 2$ implies that $\phi_0$ is a probability density function. Thus it suffices to define
    %
    \[ \phi(x_1, \dots, x_d) = \phi_0(x_1) \dots \phi_0(x_d). \qedhere \]
\end{proof}
\end{comment}
%
For each $r \in (0,1)$, we can then define $\phi_r \in C^\infty(\TT^d)$ by writing
\[ \phi_r(x) = \begin{cases} r^{-d} \phi(x/r) &: |x| < r, \\ 0 &: \text{otherwise}. \end{cases} \]
The following standard properties hold:
\begin{enumerate}
    \item[(1)] For each $r \in (0,1)$, $\phi_r$ is a non-negative smooth function with
    \begin{equation}
        \int_{\TT^d} \phi_r(x)\; dx = 1,
    \end{equation}
    and $\phi_r(x) = 0$ for $|x| \geq r$.

    \item[(2)] For any $r \in (0,1)$,
    \begin{equation} \label{equationDIOJAOIJVIV23242}
        \big\| \widehat{\phi}_r \big\|_{L^\infty(\ZZ^d)} = 1.
    \end{equation}

%    \item For any positive integer $N$, if $\varepsilon = 1/N$ and $x \in \TT^d$,
    %
%    \begin{equation} \label{equation5550002352124124512}
%        \sum_{k \in [2N]^d} \phi_{1/N}(x + k/2N) = (2N)^d.
%    \end{equation}

    \item[(3)] For each $\xi \in \ZZ^d$,
    \begin{equation} \label{approximationtoidentitypointwiseconvergence}
        \lim_{r \to 0} \widehat{\phi}_r(\xi) = 1.
    \end{equation}

    \item[(4)] For each $T > 0$, for all $r > 0$, and for any non-zero $\xi \in \ZZ^d$,
    \begin{equation} \label{molificationdecaybound}
        |\widehat{\phi}_r(\xi)| \lesssim_T r^{-T} |\xi|^{-T}.
    \end{equation}
\end{enumerate}
We will prove Lemma \ref{LemmaFIOAJFOIWJ} after a series of more elementary lemmas which give results about the metric space $\mathcal{X}_\beta$.

\begin{lemma} \label{smoothdensitylemma}
    The set of all $(E,\mu) \in \mathcal{X}_\beta$ with $\mu \in C^\infty(\TT^d)$ and $\text{supp}(\mu) = E$ is dense in $\mathcal{X}_\beta$.
\end{lemma}
\begin{proof}
    Let 
    \begin{equation}
        \tilde{\mathcal{X}_\beta} = \{ (E,\mu) \in \mathcal{X}_\beta : \text{supp}(\mu) = E \}.
    \end{equation}
    We begin by proving that the set of all $(E,\mu) \in \tilde{\mathcal{X}_\beta}$ such that $\mu \in C^\infty(\TT^d)$ is dense in $\tilde{\mathcal{X}_\beta}$. Fix $(E_0,\mu_0) \in \tilde{\mathcal{X}_\beta}$. For each $r \in (0,1)$, consider the convolved measure $\mu_r = \mu_0 * \phi_r$. Then $\mu_r \in C^\infty(\TT^d)$ and $\text{supp}(\mu_r) = E_r$. We claim that $\lim_{r \to 0} (E_r,\mu_r) = (E_0,\mu_0)$, which would complete the proof. Since $d_{\mathbb{H}}(E_0,E_r) \leq r$, we find that $\lim_{r \to 0} E_r = E_0$ holds with respect to the Hausdorff metric. Now fix $\lambda \in (0,\beta]$ and $\delta > 0$. For each $\xi \in \ZZ^d$, $|\widehat{\mu}_r(\xi)| = |\widehat{\phi}_r(\xi)| |\widehat{\mu}_0(\xi)|$, so
    \begin{equation} \label{equationFFSCI}
        |\xi|^{\lambda/2} |\widehat{\mu}_r(\xi) - \widehat{\mu}_0(\xi)| = |\xi|^{\lambda/2} |\widehat{\phi}_r(\xi) - 1| |\widehat{\mu}_0(\xi)|.
    \end{equation}
    We control \eqref{equationFFSCI} using the fact that $|\widehat{\mu}_0(\xi)|$ is small when $\xi$ is large, and $|\widehat{\phi}_r(\xi) - 1|$ is small when $\xi$ is small. Since $(E_0,\mu_0) \in \mathcal{X}_\beta$, we can apply \eqref{equationGFSCSC4} to find $R > 0$ such that for $|\xi| \geq R$,
    \begin{equation} \label{equationDIICSIC}
        |\xi|^{\lambda/2} |\widehat{\mu}_0(\xi)| \leq \delta/2.
    \end{equation}
    Combining \eqref{equationFFSCI}, \eqref{equationDIICSIC}, and \eqref{equationDIOJAOIJVIV23242}, for $|\xi| \geq R$ we find that
    \begin{equation} \label{equationDSCISIIXX}
        |\xi|^{\lambda/2} |\widehat{\mu}_r(\xi) - \widehat{\mu}_0(\xi)| \leq \delta.
    \end{equation}
    On the other hand, \eqref{approximationtoidentitypointwiseconvergence} shows that there exists $r_0 > 0$ such that for $r \leq r_0$ and $|\xi| \leq R$,
    \begin{equation} \label{equationDISCIIS}
        |\xi|^{\lambda/2} |\widehat{\phi}_r(\xi) - 1| \leq \delta.
    \end{equation}
    The $(L^1,L^\infty)$ bound for the Fourier transform implies that $|\widehat{\mu}_0(\xi)| \leq \mu_0(\TT^d) = 1$, which combined with \eqref{equationDISCIIS} gives that for $r \leq r_0$ and $|\xi| \leq R$,
    \begin{equation} \label{equatioNFISISCISI}
        |\xi|^{\lambda/2} |\widehat{\mu}_r(\xi) - \widehat{\mu_0}(\xi)| \leq \delta.
    \end{equation}
    Putting together \eqref{equationDSCISIIXX} and \eqref{equatioNFISISCISI} shows that for $r \leq r_0$, $\| \mu_r - \mu_0 \|_{M(\lambda)} \leq \delta$. Since $\delta$ and $\lambda$ were arbitrary, we conclude that $\lim_{r \to 0} \mu_r = \mu_0$. Thus the set of all pairs $(E,\mu) \in \tilde{\mathcal{X}_\beta}$ with $\mu \in C^\infty(\TT^d)$ is dense in $\tilde{\mathcal{X}_\beta}$.

    Our proof will therefore be complete if we can show that $\tilde{\mathcal{X}_\beta}$ is dense in $\mathcal{X}_\beta$. We prove this using a Baire category argument. For each closed cube $Q \subset \TT^d$, let
    \[ A(Q) = \{ (E,\mu) \in \TT^d: (E \cap Q) = \emptyset\ \text{or}\ \mu(Q) > 0 \}. \]
    Then $A(Q)$ is an open set. If $\{ Q_k \}$ is a countable sequence enumerating all cubes with rational corners in $\TT^d$, then
    \begin{equation}
        \bigcap_{k = 1}^\infty A(Q_k) = \tilde{\mathcal{X}_\beta}.
    \end{equation}
    Thus it suffices to show that $A(Q)$ is dense in $\mathcal{X}_\beta$ for each closed cube $Q$. To do this, we fix $(E_0,\mu_0) \in \mathcal{X}_\beta - A(Q)$, $\lambda \in [0,\beta)$, and $\varepsilon > 0$, and try and find $(E,\mu) \in A(Q)$ with $d_\mathbb{H}(E,E_0) \leq \varepsilon$ and $\| \mu_0 - \mu \|_{M(\lambda)} \leq \varepsilon$.

    Because $(E_0,\mu_0) \in \mathcal{X}_\beta - A(Q)$, we know $E_0 \cap Q \neq \emptyset$ and $\mu_0(Q) = 0$. Find a smooth probability measure $\nu$ supported on $N(E_0,\varepsilon) \cap Q$ and, for $t \in (0,1)$, define $\mu_t = (1 - t) \mu_0 + t \nu$. Then $\text{supp}(\mu_t) \subset N(E_0,\varepsilon)$, so if we let $E = \text{supp}(\nu) \cup \text{supp}(\mu)$, then $d_\mathbb{H}(E,E_0) \leq \varepsilon$. Clearly $(E,\mu_t) \in A(Q)$ for $t > 0$. And
    \begin{equation}
        \| \mu_t - \mu_0 \|_{M(\lambda)} \leq t \left( \| \mu_0 \|_{M(\lambda)} + \| \nu \|_{M(\lambda)} \right),
    \end{equation}
    so if we choose $t \leq \varepsilon \cdot (\| \mu \|_{M(\lambda)} + \| \nu \|_{M(\lambda)})^{-1}$ we find $\| \mu_t - \mu \|_{M(\lambda)} \leq \varepsilon$. Since $\varepsilon$ was arbitrary, we conclude $A(Q)$ is dense in $\mathcal{X}_\beta$.
    % TODO: CAN WE DO THIS FOR CURVED SURFACES?
    % \int e^{2 \pi i (xi * x + eta f(x))} [mu(x) - mu(y) \phi(x-y)] dx dy
\end{proof}

\begin{remark}
    The reason we must work with the metric space $\mathcal{X}_\beta$ rather than the smaller space $\tilde{\mathcal{X}_\beta} \subset \mathcal{X}_\beta$ is that $\tilde{\mathcal{X}_\beta}$ is not a closed subset of $\mathcal{X}_\beta$, and so is not a complete metric space, preventing the use of the Baire category theorem. However, the latter arguments in the proof of Lemma \ref{smoothdensitylemma} shows that quasi-all elements of $\mathcal{X}_\beta$ belong to $\tilde{\mathcal{X}_\beta}$, so that one can think of $\mathcal{X}_\beta$ and $\tilde{\mathcal{X}_\beta}$ as being equal `generically'.
\end{remark}

The density argument of Lemma \ref{LemmaFIOAJFOIWJ} requires constructing approximations to an arbitrary element of $(E_0,\mu_0) \in \mathcal{X}_\beta$ by $(E,\mu) \in \mathcal{X}_\beta$ such that $E \in \mathcal{A}$. We do this by multiplying $\mu_0$ by a smooth function $f \in C^\infty(\TT^d)$ which cuts off parts of $\mu_0$ which cause the support of $\mu_0$ to fail to be in $\mathcal{A}$. As long as $\mu_0$ is appropriately smooth, and the Fourier transform of $f$ decays appropriately quickly, the next lemma shows that $f \mu_0 \approx \mu_0$.

\begin{lemma} \label{LemmaTTSICICS}
    Consider a finite measure $\mu_0$ on $\TT^d$, as well as a smooth probability density function $f \in C^\infty(\TT^d)$. If we define $\mu = f \mu_0$, then for any $\lambda \in [0,d)$,
    \[ \| \mu - \mu_0 \|_{M(\lambda)} \lesssim_d \| \mu_0 \|_{M(d+1)} \| f \|_{M(\lambda)}. \]
\end{lemma}
\begin{proof}
    Since $\widehat{\mu} = \widehat{f} * \widehat{\mu_0}$, and $\widehat{f}(0) = 1$, for each $\xi \in \ZZ^d$ we have
    \begin{equation} \label{equationPPYTUECUUCS}
    \begin{split}
        |\xi|^{\lambda/2} |\widehat{\mu}(\xi) - \widehat{\mu}_0(\xi)| &= |\xi|^{\lambda/2} \left| \sum_{\eta \neq \xi} \widehat{f}(\xi - \eta) \widehat{\mu}_0(\eta) \right|.
    \end{split}
    \end{equation}
    If $|\eta| \leq |\xi|/2$, then $|\xi|/2 \leq |\xi - \eta| \leq 2 |\xi|$, so
    \begin{equation} \label{equationPPDOSO}
        |\xi|^{\lambda/2} |\widehat{f}(\xi - \eta)| \leq \| f \|_{M(\lambda)} |\xi|^{\lambda/2} |\xi-\eta|^{-\lambda/2} \leq 2^{\lambda/2} \| f \|_{M(\lambda)} \lesssim_d \| f \|_{M(\lambda)}.
    \end{equation}
    Thus the bound \eqref{equationPPDOSO} implies
    \begin{equation} \label{equationGGPSOVVCSI}
    \begin{split}
        |\xi|^{\lambda/2} \left| \sum_{0 \leq |\eta| \leq |\xi|/2} \widehat{f}(\xi - \eta) \widehat{\mu}_0(\eta) \right| &\lesssim_{\mu_0,d} \| \mu_0 \|_{M(d+1)} \| f \|_{M(\lambda)} \left( 1 + \sum_{0 < |\eta| \leq |\xi|/2} \frac{1}{|\eta|^{d+1}} \right)\\
        &\lesssim_d \| \mu_0 \|_{M(d+1)} \| f \|_{M(\lambda)} \leq \| \mu_0 \|_{M(d+1)} \| f \|_{M(\lambda)}.
    \end{split}
    \end{equation}
    % If mu_0 is convolved with phi_r, then for any T > 0, we get
    % 1 + sum r^{-T}/|eta|^{T + beta/2} = 1 + r^{-T} / |xi|^{T-d + beta/2}
    % r^{-1 + (\beta/2)/(T + \beta/2)} <= |\eta|
    % O(|xi|^d) for |xi| <= 1/r
    % 
    On the other hand, for all $\eta \neq \xi$,
    \begin{equation} \label{equationGGDPSOX}
    \begin{split}
        |\widehat{f}(\xi - \eta)| \leq  \| f \|_{M(\lambda)} |\xi - \eta|^{-\lambda} \leq \| f \|_{M(\lambda)}.
    \end{split}
    \end{equation}
    Thus we calculate that
    \begin{equation} \label{equationGGHOODPPS}
    \begin{split}
        |\xi|^{\lambda/2} \left| \sum_{\substack{|\eta| > |\xi|/2\\ \eta \neq \xi}} \widehat{f}(\xi - \eta) \widehat{\mu}_0(\eta) \right| &\lesssim_{d,\mu_0} \| \mu_0 \|_{M(d+1)} \| f \|_{M(\lambda)} \cdot |\xi|^{\lambda/2} \sum_{|\eta| > |\xi|/2} \frac{1}{|\eta|^{d+1}}\\
        &\lesssim_d \| \mu_0 \|_{M(d+1)} \| f \|_{M(\lambda)}.
    \end{split}
    \end{equation}
    Combining \eqref{equationPPYTUECUUCS}, \eqref{equationGGPSOVVCSI} and \eqref{equationGGHOODPPS} completes the proof.
\end{proof}

The bound in Lemma \ref{LemmaTTSICICS}, if $\| f \|_{M(\lambda)}$ is taken appropriately small, also implies that the Hausdorff distance between the supports of $\mu$ and $\mu_0$ are small.

\begin{lemma} \label{LemmaTAOIAWOIDJ12301}
    Fix a probability measure $\mu_0 \in C^\infty(\TT^d)$ and $\lambda \in [0,d)$. For any $\varepsilon > 0$, there exists $\delta > 0$ such that if $\mu \in C^\infty(\TT^d)$, $\text{supp}(\mu) \subset \text{supp}(\mu_0)$, and $\| \mu_0 - \mu \|_{M(\lambda)} \leq \delta$, then $d_\mathbb{H}(\text{supp}(\mu),\text{supp}(\mu_0)) \leq \varepsilon$.
\end{lemma}
\begin{proof}
    Consider any cover of $\text{supp}(\mu_0)$ by a family of radius $\varepsilon/3$ balls $\{ B_1,\dots,B_N \}$, and for each $i \in \{ 1, \dots, N \}$, consider a smooth function $f_i \in C_c^\infty(B_i)$ such that there is $s > 0$ with
    \begin{equation} \label{equationCIJCIJCIJ}
        \int f_i(x) d\mu_0(x) \geq s
    \end{equation}
    for each $i \in \{ 1, \dots, N \}$. Fix $A > 0$ with
    \begin{equation} \label{equationvVVIJIJX}
        \sum_{\xi \neq 0} |\widehat{f}_i(\xi)| \leq A
    \end{equation}
    for all $i \in \{ 1, \dots, N \}$ as well. Set $\delta = s/2A$. If $\| \mu_0 - \mu \|_{M(\lambda)} \leq \delta$, we apply Plancherel's theorem together with \eqref{equationCIJCIJCIJ} and \eqref{equationvVVIJIJX} to conclude that
    \begin{equation} \label{equationIVIJVIVJIVJ}
    \begin{split}
        \left| \int f_i(x) d\mu(x)\; dx - \int f_i(x) d\mu_0(x) \right| &= \left| \sum_{\xi \in \ZZ^d} \widehat{f}_i(\xi) \left( \widehat{\mu}(\xi) - \widehat{\mu}_0(\xi) \right) \right|\\
        &\leq A \| \mu_0 - \mu \|_{M(\lambda)}\\
        &\leq s/2.
    \end{split}
    \end{equation}
    Thus we conclude from \eqref{equationCIJCIJCIJ} and \eqref{equationIVIJVIVJIVJ} that
    \begin{equation} \label{equationVIOJVIOSJCICJXXXXX}
        \int f_i(x) d\mu(x)\; dx \geq \int f_i(x) d\mu_0(x) - s/2 \geq s/2 > 0.
    \end{equation}
    Since equation \eqref{equationVIOJVIOSJCICJXXXXX} holds for each $i \in \{ 1,\dots, N \}$, the support of $\mu$ intersects every ball in $\{ B_1, \dots, B_N \}$. Combined with the assumption that $\text{supp}(\mu) \subset \text{supp}(\mu_0)$, this implies that $d_\mathbb{H}(\mu_0,\mu) \leq \varepsilon$.
\end{proof}

Now we have the technology to prove Lemma \ref{LemmaFIOAJFOIWJ}.

\begin{proof}[Proof of Lemma \ref{LemmaFIOAJFOIWJ}]
    Fix $(E_0,\mu_0) \in \mathcal{X}_\beta$. By Lemma \ref{smoothdensitylemma}, without loss of generality, we may assume that $\mu_0 \in C^\infty(\TT^d)$ and that $\text{supp}(\mu_0) = E_0$. Our goal, for any $\lambda \in [0,\beta)$ and $\delta_0 > 0$, is to find $(E,\mu) \in \mathcal{X}_\beta$ such that $E \in \mathcal{A}$, $d_{\mathbb{H}}(E,E_0) \leq \delta_0$, and $\| \mu - \mu_0 \|_{M(\gamma)} \leq \delta_0$.

    Fix $\delta > 0$, $\varepsilon > 0$, and $\lambda \in (\gamma,\beta)$, and consider a set $S = \{ x_1, \dots, x_N \}$ and $\{ a_1, \dots, a_N \}$ satisfying the assumptions of the Lemma. If we set
    \[ \eta = \frac{1}{N} \sum_{k = 1}^N a_k \delta_{x_k}, \]
    then $\eta$ is a probability measure, and Property (2) implies that for $|\xi| \leq (1/r)^{1 + \kappa}$,
    \begin{equation} \label{equationBLCHACHWODWORH}
        |\widehat{\eta}(\xi)| \leq C N^{-1/2} \log(N) + \delta |\xi|^{-\lambda/2}.
    \end{equation}
    Consider the function $f = \eta * \phi_r$, where $\phi_r$ is the mollifier defined in the notation section. For each $\xi \in \ZZ^d$,
    \begin{equation} \label{equation6666GGCIS}
        \widehat{f}(\xi) = \widehat{\eta}(\xi) \widehat{\phi}_r(\xi).
    \end{equation}
    % N = r^{\varepsilon - \beta}
    % |\xi| \leq 1/r = N^{1/(\beta - \varepsilon)}
    % 
    For $|\xi| \leq 1/r$, \eqref{equationBLCHACHWODWORH} and \eqref{equationDIOJAOIJVIV23242} together with \eqref{equation6666GGCIS} imply that there is $\kappa_1 > 0$ depending on $\beta,\lambda$, and $\gamma$ such that
    \begin{equation}
        |\widehat{f}(\xi)| \leq C N^{-1/2} \log(N) + \delta |\xi|^{-\lambda/2} \leq (C N^{-\kappa_1} + \delta) |\xi|^{-\gamma/2}.
    \end{equation}
    Thus if $N$ is suitably large, we conclude that for $|\xi| \leq 1/r$,
    \begin{equation} \label{equationOIJVIOJOIJ21314}
        |\widehat{f}(\xi)| \leq 2 \delta |\xi|^{-\gamma/2}.
    \end{equation}
    If $(1/r) \leq |\xi| \leq (1/r)^{1+\kappa}$, \eqref{molificationdecaybound} implies $|\widehat{\phi}_r(\xi)| \lesssim_\beta r^{-\beta/2} |\xi|^{-\beta/2}$, which together with \eqref{equationDIOJAOIJVIV23242} and \eqref{equationBLCHACHWODWORH} applied to \eqref{equation6666GGCIS} allows us to conclude that there is $\kappa_2 > 0$ depending on $\beta$, $\lambda$, and $\gamma$, such that
    \begin{equation} \label{equationGGOOSC66341}
    \begin{split}
        |\widehat{f}(\xi)| &= \delta |\xi|^{-\lambda/2} + O_\beta \left( N^{-1/2} \log(N) \cdot r^{-\beta/2} |\xi|^{-\beta/2} \right)\\
        &\leq \left( \delta + O_{\beta,\kappa} \left( N^{-1/2} \log(N) \cdot r^{-\beta/2} |\xi|^{-(\beta - \lambda)/2} \right) \right) |\xi|^{- \lambda/2}\\
        &\leq \left( \delta + O_{\beta,\kappa} \left( N^{-1/2} \log(N) r^{-\lambda/2} \right) \right) |\xi|^{-\lambda/2}\\
        &\leq ( \delta + O_{\beta,\kappa} (N^{-\kappa_2(\beta,\varepsilon)}) ) |\xi|^{-\gamma/2}.
    \end{split}
    \end{equation}
    Thus if $N$ is sufficiently large, then for $(1/r) \leq |\xi| \leq (1/r)^{1+\kappa}$,
    \begin{equation} \label{equationaJJDIWJDIWJDIWJIDJW44141}
        |\widehat{f}(\xi)| \leq 2 \delta |\xi|^{-\gamma/2}.
    \end{equation}
    Finally, if $|\xi| \geq (1/r)^{1 + \kappa}$, we apply \eqref{molificationdecaybound} for $T \geq \beta/2$ together with the bound $\| \widehat{\eta} \|_{L^\infty(\ZZ^d)} = 1$, which follows because $\eta$ is a probability measure, to conclude that
    \begin{equation} \label{equationGGUSCCCYVSSXX998723}
    \begin{split}
        |\widehat{f}(\xi)| &\lesssim_T r^{-T} |\xi|^{-T}\\
        &= r^{-T} |\xi|^{\beta/2 - T} \cdot |\xi|^{-\beta/2}\\
        &\leq r^{-T} (1/r)^{(\beta/2 - T)(1 + \kappa)} \cdot |\xi|^{-\beta/2}\\
        &= r^{\kappa T - (\beta/2)(1 + \kappa)} \cdot |\xi|^{-\beta/2}.
    \end{split}
    \end{equation}
    If we choose $T > (\beta/2)(1 + 1/\kappa)$, then as $r \to 0$, $r^{\kappa T - (\beta/2)(1 + \kappa)} \to 0$. Thus we conclude from \eqref{equationGGUSCCCYVSSXX998723} that if $N$ is sufficiently large, then for $|\xi| \geq (1/r)^{1 + \kappa}$
    \begin{equation} \label{equationBBCDSGDCC77}
        |\widehat{f}(\xi)| \leq 2 \delta |\xi|^{-\beta/2} \leq 2 \delta |\xi|^{- \gamma/2}.
    \end{equation}
    Combining \eqref{equationOIJVIOJOIJ21314}, \eqref{equationaJJDIWJDIWJDIWJIDJW44141} and \eqref{equationBBCDSGDCC77} shows that if $N$ is sufficiently large,
    \begin{equation} \label{equationCIOJCIOJAOCIJAWDI41412421}
        \| f \|_{M(\gamma)} \leq 2 \delta.
    \end{equation}
    Intuitively, if $\delta \ll 1$, then the Fourier transform of $f$ approximately looks like the Dirac delta function at the origin in $\TT^d$, so we should expect $f \approx 1$ on $\TT^d$. In particular, we should expect that $f \mu_0 \approx \mu_0$. Since $\text{supp}(f \mu_0) \subset \text{supp}(f) \subset N(S,r)$, we know that $\text{supp}(f \mu_0) \in \mathcal{A}$. Carrying out all these details numerically will complete the proof of density.

    We start by applying Lemma \ref{LemmaTTSICICS} using \eqref{equationCIOJCIOJAOCIJAWDI41412421}, which implies that if $\rho = f \mu_0$, then
    \begin{equation}
        \| \rho - \mu_0 \|_{M(\gamma)} \lesssim_{d,\mu_0} \| f \|_{M(\gamma)} \leq 2 \delta.
    \end{equation}
    Using \eqref{equationCIOJCIOJAOCIJAWDI41412421} and the fact that $\widehat{\mu}_0 \in L^1(\ZZ^d)$ because $\mu_0 \in C^\infty(\TT^d)$, we find that
    \begin{equation} \label{equationFIOJCIOWJCOIJFIO}
    \begin{split}
        \rho(\TT^d) &= (\widehat{f} * \widehat{\mu}_0)(0) \geq 1 - \sum_{|\xi| \neq 0} |\widehat{f}(\xi)| |\widehat{\mu}_0(-\xi)| \geq 1 - O_{\mu_0}(2 \delta).
    \end{split}
    \end{equation}
    Thus if we define $\mu = \rho / \rho(\TT^d)$, then for $\delta \leq 1$,
    \begin{equation} \label{equationCIOJCOISJCOIJDIOJEO13123}
    \begin{split}
        \| \mu - \mu_0 \|_{M(\gamma)} &\leq \| \mu - \rho \|_{M(\gamma)} + \| \rho - \mu_0 \|_{M(\gamma)}\\
        &= (1 / \rho(\TT^d) - 1) \cdot \| \rho \|_{M(\gamma)} + \| \rho - \mu_0 \|_{M(\gamma)}\\
        &\lesssim_{\mu_0} \delta \| \rho \|_{M(\gamma)} + \| \rho - \mu_0 \|_{M(\gamma)}\\
        &\lesssim \delta \| \mu_0 \|_{M(\gamma)} + \| \rho - \mu_0 \|_{M(\gamma)} \lesssim \delta.
    \end{split}
    \end{equation}
    If we take $\delta$ suitably small, \eqref{equationCIOJCOISJCOIJDIOJEO13123} implies that $\| \mu - \mu_0 \|_{M(\gamma)} \leq \delta_0$. Since $\text{supp}(\mu) \subset \text{supp}(\mu_0)$, Lemma \ref{LemmaTAOIAWOIDJ12301} implies that if $\delta$ is taken even smaller, then $d_{\mathbb{H}}(E, E_0) \leq \delta_0$. Thus if we set $E = \text{supp}(\mu)$, then $E \in \mathcal{A}$ since $E \subset N(S,r)$ and $\mathcal{A}$ is downward closed, and since $\delta_0$ and $\gamma$ were arbitrary, this completes the proof of density.
\end{proof}

\bibliographystyle{amsplain}
\bibliography{FourierDimensionNonlinearPatterns}

% c = f(x_2,...,x_n)
% Tangent space equals common null space of df^1,...,df^n, and this null space is d(n-2) dimensional.
% so df^1,..., df^n spans a d dimensional subspace of linear functionals.
% If we add in dx^j_{i_1},...,dx^j_{i_m}, then this spans a dm space of functions and the nullspace is d(n-m-1) dimensional.
% Hopefully the intersection is d(n-m-2) dimensional, so the set
%{ df^1,...,df^n,dx^j_{i_1},..,dx^j_{i_m} }
% should have dimension d(m+1)

% This holds if the map x -> f(x_2,...,x_{i-1},x,x_{i+1},...,x_n) is a diffeomorphism.

% Is there a determinant condition for this? If we remove the m minors corresponds to dx^j_{i_1}, the remaining d(n-1) -> d
% d(n-m-1) -> d

% So if we add in dx_{i_1},...,dx_{i_{n-m}}
% y = (x_1 + x_2 - 2x_3)^2
% df = 2(x_1 + x_2 - 2x_3) (dx_1 + dx_2 - 2dx_3)
% m = 1: should be 2 dimensional CHECK
% m = 2: should be 3 dimensional CHECK

\end{document}